\documentclass{amsart}
\usepackage{graphicx}
\usepackage{xcolor}
\usepackage{fullpage}

\usepackage{mathtools}
\mathtoolsset{showonlyrefs} 

\graphicspath{{Fig/}}

\newtheorem{thm}{Theorem}

\newtheorem{fact}[thm]{Fact}

\theoremstyle{definition}
\newtheorem{dfn}{Definition}

\DeclareMathOperator{\dvg}{div}
\DeclareMathOperator{\grad}{grad}
\DeclareMathOperator{\jgrad}{\mathcal{J}grad}

\newcommand{\rp}{\mathrm{Re}}

\newcommand{\re}{{\mathbb R}}
\newcommand{\ce}{{\mathbb C}}
\newcommand{\ze}{{\mathbb Z}}

\theoremstyle{plain}
\newtheorem{theorem}[thm]{Theorem}
\newtheorem{lemma}[thm]{Lemma}
\newtheorem{proposition}[thm]{Proposition}
\newtheorem{corollary}[thm]{Corollary}
\theoremstyle{definition}
\newtheorem{definition}[dfn]{Definition}

\DeclareMathOperator{\ran}{ran}
\DeclareMathOperator{\coker}{coker}
\DeclareMathOperator{\indop}{ind}
\DeclareMathOperator{\Tr}{Tr}
\DeclareMathOperator{\tr}{tr}

\newcommand{\id}{\mathrm{id}}

\begin{document}

\title{Continuation and reduction of steady vortex sheets under metric deformations on a disk}

\author{Yuuki Shimizu} 
\address{Faculty of Science, Academic Assembly, University of Toyama, 3190 Gofuku, Toyama 930-8555, Japan} 
\email{yuuki.shimizu1992@gmail.com}


\date{}


\begin{abstract}
We study the local continuation of a steady vortex sheet under a prescribed
deformation $g_s= e^{2\sigma_s}g_{\mathrm e}$ of a conformal metric on the unit
disk.  Representing the moving sheet by a normal graph, we formulate the
streamline and Bernoulli conditions as a nonlinear residual whose two
components have different Sobolev orders.  If the two one-sided tangential
velocities of the reference sheet do not vanish simultaneously, the
derivative of the unnormalized residual is Fredholm of index two.  Fixing the
mean normal displacement and the total circulation gives an index-zero
problem; when the normalization differential restricted to the unnormalized
kernel is onto $\re^2$, we obtain a parameter-dependent
Lyapunov--Schmidt reduction.  This yields a locally unique normalized branch
in the nondegenerate case and explicit first- and second-order necessary
conditions in the degenerate case.  For a rotationally symmetric reference
metric and a concentric circular sheet, the linearized problem separates
into Fourier blocks and only finitely many modes can be resonant.  In the
constant-curvature disk family, every mode $k\geq2$ is nonresonant, whereas
the first mode produces a two-dimensional reduced problem.  The two
first-mode kernel profiles are normal traces of ambient Killing fields, but
they do not generate fixed-boundary symmetries or, by themselves, a
nonconcentric branch.
\end{abstract}

\maketitle

\section{Introduction}
\label{sec:introduction}

A vortex sheet is a codimension-one concentration of vorticity across which
the tangential component of the velocity has a jump.  In two-dimensional
incompressible flow, its motion can be described either as a singular weak
solution of the Euler equation or through the Birkhoff--Rott equation.  The
relation between these descriptions, as well as the analytical difficulties
caused by the Kelvin--Helmholtz instability, has been studied extensively;
see, for example,
\cite{MarchioroPulvirenti1994,MajdaBertozzi2002,LopesFilhoLopesSchochet2007}.
On a Riemannian manifold, Izosimov and Khesin developed a geometric
description of vortex sheets in terms of discontinuous differential forms,
Hamiltonian structures, and layer potentials \cite{IzosimovKhesin2018}.
These theories treat the ambient geometry as fixed.

The present problem begins with a different question.  Let
$g_0= e^{2\sigma_0}g_{\mathrm e}$ be a conformal metric on the unit disk and
let $(c_0,\gamma_0)$ be a steady vortex sheet for $g_0$.  If the metric is
prescribed to vary along a path

\begin{equation}
  g_s= e^{2\sigma_s}g_{\mathrm e},
  \label{eq:intro-metric-path}
\end{equation}

can the support curve and the sheet strength be varied so that the sheet
remains steady?  Existing constructions of stationary and relative-equilibrium
vortex sheets in planar domains use perturbative arguments around particular
configurations \cite{CaoQinZou2023Stationary,CaoQinZou2023Relative}.  Here the
reference sheet is not required to be circular, the prescribed perturbation
acts on the metric rather than on the vorticity profile or the domain, and the
outer boundary remains fixed.

We write a nearby support as the normal graph
$c_f=c_0+fn$ and perturb the strength by $h$.  The two constants occurring in the streamline and Bernoulli
conditions are denoted by $\mu$ and $\beta$, so the complete unknown is
$u=(f,h,\mu,\beta)$.  Pulling both stationary conditions
back to $c_0$ gives a nonlinear residual

\begin{equation}
 F\colon(-s_0,s_0)\times O\longrightarrow
 H^{m+1}(c_0)\times H^m(c_0),
 \qquad m\geq3.
 \label{eq:intro-residual-type}
\end{equation}

The mixed target is forced by the equations.  The streamline residual is the
trace of a single-layer potential and gains one derivative, whereas the
Bernoulli residual contains an averaged normal derivative and remains at the
order of the sheet strength.  Consequently, the derivative with respect to
$(f,h)$ has the mixed principal form

\begin{equation}
 \begin{pmatrix}
 M_{-w\langle v_0\rangle_{\tau}}&\mathcal S_0M_{w}\\
 M_{-\gamma_0w^{-1}}\mathcal T_0M_{\gamma_0w}
   &M_{\langle v_0\rangle_{\tau}}
 \end{pmatrix},
 \qquad w= e^{\sigma_0\circ c_0},
 \label{eq:intro-principal-matrix}
\end{equation}

where $\mathcal S_0$ and $\mathcal T_0$ have orders $-1$ and $1$,
respectively.  Its effective scalar factor is

\begin{equation}
 -w\left(\langle v_0\rangle_{\tau}^2+\frac{\gamma_0^2}{4}\right)
 =-\frac{w}{2}\left((v_{\tau,0}^+)^2+(v_{\tau,0}^-)^2\right).
 \label{eq:intro-effective-factor}
\end{equation}

Thus the principal matrix is invertible whenever the two one-sided
tangential velocities do not vanish simultaneously.  We derive the complete
linearized transmission problem before extracting
\eqref{eq:intro-principal-matrix}; the remaining terms are then shown to be
compact by the mapping and shape-differentiability properties of the layer
potentials \cite{CostabelLeLouer2012}.  A two-sided parametrix and an
operator-norm continuous Fredholm homotopy prove that

\begin{equation}
 D_uF(0,u_0)\colon H^{m+1}(c_0)\times H^m(c_0)\times\re^2
 \longrightarrow H^{m+1}(c_0)\times H^m(c_0)
 \label{eq:intro-unnormalized-operator}
\end{equation}

is Fredholm of index two.

We append mean normal displacement and total circulation as two
normalization equations; the resulting derivative has index zero.  If the
normalization differential restricted to $\ker D_uF(0,u_0)$ is onto
$\re^2$, its kernel and cokernel have the same finite dimension $d$.
A parameter-dependent Lyapunov--Schmidt reduction then identifies the local
normalized steady sheets with the zeros of a $C^2$ map

\begin{equation}
 B\colon(s,\zeta)\in\re\times\mathcal K\longmapsto
 B(s,\zeta)\in\mathcal C,
 \qquad \dim\mathcal K=\dim\mathcal C=d.
 \label{eq:intro-reduced-map}
\end{equation}

If $d=0$, the normalized sheet continues uniquely along every prescribed
metric path.  If $d>0$, the reduction gives first- and second-order necessary
conditions for a $C^2$ branch, but neither resonance nor the index calculation
alone determines whether a nontrivial branch exists.

For rotationally symmetric reference metrics, the abstract degeneracy can be
computed explicitly.  Let $c_0$ be the circle of radius $\rho$,
take $\gamma_0\equiv\gamma\ne0$, put $x=\rho^2$, and define

\begin{equation}
 A(r)=2\pi\int_0^r  e^{2\sigma(q)}q\, d q,
 \quad U=\frac{A(\rho)}{A(1)},
 \quad p_\sigma=1+\rho\sigma'(\rho),
 \quad Z=\frac{2\pi\rho^2 e^{2\sigma(\rho)}}{A(1)}.
 \label{eq:intro-radial-data}
\end{equation}

The zero Fourier mode is removed by the two normalizations.  For every
$k\geq1$, the cosine and sine components have the same $2\times2$ matrix,
whose determinant, after division by $\gamma^2 e^{\sigma(\rho)}$, is

\begin{equation}
 \mathring D_k^\sigma(\rho)
 =-\left(\frac{1+x^k}{2}-U\right)^2
  -\frac{1-x^{2k}}4
  +\frac{1-x^k}{2k}\,[Z+p_\sigma(1-2U)].
 \label{eq:intro-radial-determinant}
\end{equation}

This formula implies that only finitely many modes can resonate and gives an
explicit finite cutoff for checking them.

The constant-curvature disk family makes the modal conclusion particularly
sharp.  For

\begin{equation}
  e^{\sigma_K(r)}=\frac{2}{1+Kr^2},
 \qquad K>-1,
 \label{eq:intro-constant-curvature-metric}
\end{equation}

the determinant becomes

\begin{equation}
 \mathring D_k^{\sigma_K}(\rho)
 =\frac{1}{(1+Kx)^2}\left[
  (1-x)(x-x^k)
  -\frac{k-1}{2k}(1-x^k)
     \bigl(1+K(K+2)x^2\bigr)
 \right].
 \label{eq:intro-constant-curvature-determinant}
\end{equation}

It vanishes for $k=1$ and is strictly negative for every $k\geq2$.
Accordingly, $d=2$ throughout the admissible constant-curvature family.  The
two kernel profiles are the normal traces of ambient space-form Killing
fields.  Their flows do not preserve the boundary of the unit disk, however,
so this identification does not produce a symmetry-generated nonconcentric
branch.  The unresolved local question is whether the two-dimensional reduced
equation has nonconcentric zeros.

The results are local in the normal-graph chart and concern prescribed
conformal metric paths on the disk.  In the degenerate case we determine the
finite-dimensional equation and its first two necessary conditions, but do
not assert a branch count, a branch scale, or the existence of a nontrivial
branch.  These limitations occur precisely where the Fredholm analysis ends
and the particular reduced equation begins.

Section~\ref{sec:preliminaries} establishes the Euler, bifurcation, and
elliptic ingredients.  Section~\ref{sec:problem-setting} constructs the
steady-sheet residual and the two physical normalizations, and
Section~\ref{sec:main-results} states the Fredholm and continuation results.
Their proofs are given in Section~\ref{sec:proof}.  Section~\ref{sec:applications}
derives the Fourier blocks for rotationally symmetric metrics and then
specializes them to the constant-curvature disk family.  The conclusions and
the remaining finite-dimensional question are collected in
Section~\ref{sec:conclusion}.
\section{Geometric and analytic framework}
\label{sec:preliminaries}

\subsection{Euler--Arnold equation and vortex sheets}
\label{subsec:euler-vortex-sheets}

Let $(M,g)$ be a compact oriented Riemannian surface, possibly with boundary,
let $ d\mu$ be its area form, and let $\mathcal J$ be rotation by $+\pi/2$.
For a smooth incompressible velocity tangent to the boundary, the
Euler--Arnold equation and its instantaneous pressure-free variational
identity are
\begin{equation}
  \partial_t v+\nabla_v^g v=-\grad_g p,
  \qquad \dvg_g v=0,
  \label{eq:euler-arnold}
\end{equation}
and
\begin{equation}
  \int_M g(\partial_t v+\nabla_v^g v,X)\, d\mu=0
  \label{eq:weak-euler-arnold}
\end{equation}
for every smooth $g$-divergence-free vector field $X$ tangent to
$\partial M$.  This intrinsic formulation is the only one used on a general
surface.

The stream-function representation used in
Propositions~\ref{prop:streamline-characterization}
and~\ref{prop:bernoulli-characterization} is restricted to the
unit disk $\mathbb{D}\subset\re^2$ with its impermeable outer boundary.
We fix the standard orientation and write
$g= e^{2\sigma}g_{\mathrm e}$ with
$\sigma\in C^\infty(\overline{\mathbb D})$.  The stream function is
normalized by zero Dirichlet data on $\partial\mathbb{D}$; on the simply
connected disk this fixes the velocity without an additional harmonic
circulation component.  With
$\triangle_g=\dvg_g\grad_g$, the conventions are
\begin{equation}
  v=-\jgrad_g\psi,
  \qquad -\triangle_g\psi=\omega,
  \qquad \psi|_{\partial\mathbb{D}}=0.
  \label{eq:stream-function-convention}
\end{equation}
The Dirichlet condition implies that $v$ is tangent to the outer boundary.

The fixed Euclidean Dirichlet Green kernel is denoted by $G$ and is normalized by
\begin{equation}
  G(z,y)=\frac{1}{2\pi}
  \log\frac{|1-z\overline{y}|}{|z-y|},
  \qquad -\triangle_{\mathrm e,z}G(z,y)=\delta_{y},
  \qquad G(\mathord{\cdot},y)|_{\partial\mathbb{D}}=0.
  \label{eq:disk-green-kernel}
\end{equation}

\begin{definition}[Vortex-sheet data]
\label{def:vortex-sheet-data}
Let $c$ be a connected, separating $C^2$ embedded closed curve in
$\mathbb{D}^\circ$ and let $\gamma\in C^1(c)$ be its sheet strength per unit
$g$-arclength.  Define
\begin{equation}
  \Gamma_{c,\gamma}=\int_c\gamma\, d\ell_g,
  \qquad
  \bar\omega_{c,\gamma}
  =\frac{\Gamma_{c,\gamma}}{|\mathbb{D}|_g},
  \qquad
  |\mathbb{D}|_g=\int_{\mathbb{D}} d\mu.
  \label{eq:circulation-background}
\end{equation}
The weak Dirichlet stream function $\psi_{c,\gamma}\in H^1_0(\mathbb{D})$
is the solution of
\begin{equation}
  \int_{\mathbb{D}}
  g(\grad_g\psi_{c,\gamma},
    \grad_g\phi)\, d\mu
  =\int_c\gamma\phi\, d\ell_g
   -\bar\omega_{c,\gamma}
    \int_{\mathbb{D}}\phi\, d\mu,
  \qquad \phi\in H^1_0(\mathbb{D}),
  \label{eq:weak-sheet-poisson}
\end{equation}
and its velocity is
$v_{c,\gamma}=-\jgrad_g\psi_{c,\gamma}$.
The choice of $\bar\omega_{c,\gamma}$ makes the total vorticity zero;
it is not a compatibility condition for the Dirichlet problem.
\end{definition}

\begin{definition}[Traces, averages, and jumps]
\label{def:traces-jumps}
Write
$\mathbb{D}\setminus c=\Omega^-\mathbin{\dot\cup}\Omega^+$, where
$\Omega^-$ is the bounded component and $\Omega^+$ contains
$\partial\mathbb{D}$.  The $g$-unit normal $n$ points from $\Omega^-$ to
$\Omega^+$ and $\tau=\mathcal J n$.  For any quantity $q$ with one-sided
traces, set
\begin{equation}
  \langle q\rangle=\frac{q^++q^-}{2},
  \qquad [q]=q^+-q^-.
  \label{eq:average-jump-convention}
\end{equation}
In particular,
\begin{equation}
  \langle v_{c,\gamma}\rangle_\tau
  =\frac{g(v_{c,\gamma}^+,\tau)
          +g(v_{c,\gamma}^-,\tau)}{2}.
  \label{eq:average-tangential-velocity}
\end{equation}
In the conformal coordinate, the line and area measures are
\begin{equation}
   d\ell_g= e^\sigma d\ell_{\mathrm e},
  \qquad  d\mu= e^{2\sigma} d A_{\mathrm e}.
  \label{eq:conformal-frame-measures}
\end{equation}
\end{definition}

\begin{proposition}[Weak stream function and fixed Green representation]
\label{prop:weak-stream-function}
The problem in Definition~\ref{def:vortex-sheet-data} has a unique solution.
It satisfies, in Euclidean distributions,
\begin{equation}
  -\triangle_{\mathrm e}\psi_{c,\gamma}
  =\gamma  e^\sigma\delta_c
   -\bar\omega_{c,\gamma} e^{2\sigma}
  \quad\text{in }\mathbb{D},
  \qquad \psi_{c,\gamma}|_{\partial\mathbb{D}}=0,
  \label{eq:sheet-poisson-distribution}
\end{equation}
Here the Euclidean curve delta is fixed by
$\langle\delta_c,\phi\rangle=\int_c\phi\, d\ell_{\mathrm e}$.  The solution has the fixed-kernel
representation
\begin{equation}
  \psi_{c,\gamma}(z)
  =\int_cG(z,y)\gamma(y)\, d\ell_g(y)
   -\bar\omega_{c,\gamma}
    \int_{\mathbb{D}}G(z,y)\, d\mu(y).
  \label{eq:fixed-green-representation}
\end{equation}
The trace of $\psi_{c,\gamma}$ is common to the two sides, and the one-sided
normal derivatives obey
\begin{equation}
  [\partial_n\psi_{c,\gamma}]= -\gamma,
  \qquad
  [g(v_{c,\gamma},\tau)]=\gamma.
  \label{eq:sheet-jump-relations}
\end{equation}
At the geometric $C^2/C^1$ level the restrictions are in
$H^2(\Omega^\pm)$ and have the weak one-sided traces required by
Definition~\ref{def:weak-steady-sheet}.  For a
smooth curve and strength in $H^r(c)$, $r\geq0$, the common trace is in
$H^{r+1}(c)$ and the normal traces are in $H^r(c)$.
\end{proposition}

\begin{proof}
In dimension two the Dirichlet energy is conformally invariant, so the
left-hand side of \eqref{eq:weak-sheet-poisson} equals
$\int_{\mathbb{D}}\nabla_{\mathrm e}\psi_{c,\gamma}\cdot
\nabla_{\mathrm e}\phi\, d A_{\mathrm e}$.  The interior trace map
$H^1(\mathbb{D})\to H^{1/2}(c)\to L^2(c)$ is bounded.  Hence the right-hand
side is a bounded functional on $H^1_0(\mathbb{D})$, while Poincar\'e's
inequality makes the Dirichlet form coercive.  Lax--Milgram gives existence
and uniqueness.

Converting the line and area measures by
\eqref{eq:conformal-frame-measures} gives
\eqref{eq:sheet-poisson-distribution}.  Testing
\eqref{eq:disk-green-kernel} against the finite sheet measure and the smooth
volume density proves that the right-hand side of
\eqref{eq:fixed-green-representation} solves the same weak problem.  The
outer boundary value is zero because $G(\mathord{\cdot},y)$ vanishes there, and
uniqueness proves the representation.

An $H^1$ function has one common interior trace on $c$.  Integrating
\eqref{eq:sheet-poisson-distribution} across a shrinking tubular strip, with
the normal directed from $\Omega^-$ to $\Omega^+$, gives
the Euclidean normal-derivative jump $-\gamma  e^\sigma$.
Since the metric-unit normal derivative is $ e^{-\sigma}$ times the Euclidean
normal derivative and
\[
  g(v_{c,\gamma},\tau)
  =-\partial_n\psi_{c,\gamma},
\]
the other two identities in \eqref{eq:sheet-jump-relations} follow.

The piecewise $H^2$ regularity for a $C^2$ interface and the corresponding
one-sided traces are the standard elliptic transmission estimates; see
\cite[Chapters~4 and~6]{McLean2000}.  To identify the higher-order boundary
maps in the smooth setting, pull the singular part
$-(2\pi)^{-1}\log|z-y|$ of $G$ back to a $C^2$ parametrization of $c$.
Its trace is the one-dimensional logarithmic potential, whose Fourier
multiplier has order $-1$; the remainder kernel and the volume potential
have the stated higher regularity.  The two normal traces are their average
plus and minus one half of the density.  This proves the geometric trace
clause and, on a smooth curve, the asserted $H^r$ mapping statements.
\end{proof}

\begin{definition}[Weak steady vortex sheet]
\label{def:weak-steady-sheet}
The data $(c,\gamma)$ of Definition~\ref{def:vortex-sheet-data} form a weak
steady vortex sheet if the one-sided fields supplied by
Proposition~\ref{prop:weak-stream-function} satisfy
\begin{equation}
  g(v_{c,\gamma}^+,n)=g(v_{c,\gamma}^-,n)=0
  \quad\text{on }c
  \label{eq:sheet-kinematic-condition}
\end{equation}
and
\begin{equation}
  \int_{\Omega^+}g(\nabla^g_{v_{c,\gamma}^+}v_{c,\gamma}^+,X)\, d\mu
  +\int_{\Omega^-}g(\nabla^g_{v_{c,\gamma}^-}v_{c,\gamma}^-,X)\, d\mu=0
  \label{eq:sheet-stationary-weak-euler}
\end{equation}
for every $X\in H^1(\mathbb{D};T\mathbb{D})$ that is weakly
$g$-divergence-free and tangent to $\partial\mathbb{D}$.  The nonlinear
integrals are well defined by the one-sided $H^1$ velocity regularity and
the embedding $H^1\hookrightarrow L^4$ in two dimensions.  Equivalently,
it suffices to test against smooth divergence-free tangent fields, which are
dense in this $H^1$ test space.  The physical Sobolev class in
Definition~\ref{def:bvsmd-perturbation-data} has stronger one-sided regularity, as proved in
Proposition~\ref{prop:nonlinear-residual-regularity}.
\end{definition}

\begin{proposition}[Streamline characterization]
\label{prop:streamline-characterization}
For the connected sheet in Definition~\ref{def:vortex-sheet-data}, the
kinematic condition \eqref{eq:sheet-kinematic-condition} holds if and only if
there is a unique $\mu\in\re$ such that
\begin{equation}
  \psi_{c,\gamma}|_c=\mu.
  \label{eq:streamline-constant}
\end{equation}
\end{proposition}

\begin{proof}
The orientation in Definition~\ref{def:traces-jumps} gives, on either side,
\[
  g(v_{c,\gamma}^\pm,n)
  =g(-\jgrad_g\psi_{c,\gamma}^\pm,n)
  =\partial_\tau\psi_{c,\gamma}^\pm.
\]
Proposition~\ref{prop:weak-stream-function} supplies a common trace.  Thus
both normal velocities vanish exactly when the tangential derivative of that
trace vanishes.  A weakly differentiable function with zero tangential
derivative on the connected closed curve is a single constant, and its value
is unique.
\end{proof}

\begin{lemma}[Normal-trace realization]
\label{lem:normal-trace-realization}
Let $r\geq1$, assume that $c$ is smooth, and let
$q\in H^{r-1/2}(c)$ satisfy
\begin{equation}
  \int_cq\, d\ell_g=0.
  \label{eq:normal-trace-zero-mean}
\end{equation}
There is a bounded linear choice of $X\in H^r(\mathbb{D};T\mathbb{D})$,
tangent to $\partial\mathbb{D}$, such that
\begin{equation}
  \dvg_gX=0,
  \qquad g(X,n)|_c=q,
  \qquad
  \|X\|_{H^r(\mathbb{D})}\leq C\|q\|_{H^{r-1/2}(c)}.
  \label{eq:normal-trace-right-inverse}
\end{equation}
For a $C^2$ curve the assertion holds at $r=1$, and smooth $q$ can be
realized by a smooth field.  Conversely, the normal trace on $c$ of every
divergence-free field has zero mean.
\end{lemma}

\begin{proof}
The divergence theorem on $\Omega^-$ proves the necessary zero-mean
condition.  Conversely, integrate $-q$ in metric arclength around $c$.
Condition \eqref{eq:normal-trace-zero-mean} gives a periodic
$\chi_c\in H^{r+1/2}(c)$ satisfying
$\partial_\tau\chi_c=-q$ and
$\|\chi_c\|_{H^{r+1/2}(c)}\leq C\|q\|_{H^{r-1/2}(c)}$ after its mean is
fixed to zero.  The tubular trace-extension theorem gives a bounded
extension $\chi\in H^{r+1}(\mathbb{D})$ supported in a tubular neighborhood
of $c$ disjoint from $\partial\mathbb{D}$.  Set
$X=\jgrad_g\chi$.  Rotation of a gradient is
divergence-free, $X$ vanishes near the outer boundary, and
\[
  g(X,n)=-\partial_\tau\chi=q
  \quad\text{on }c.
\]
The extension estimate gives
\eqref{eq:normal-trace-right-inverse}.  The same construction in smooth
tubular coordinates preserves smoothness.
\end{proof}

\begin{proposition}[Bernoulli characterization]
\label{prop:bernoulli-characterization}
Assume the kinematic condition \eqref{eq:sheet-kinematic-condition}.  Then
the weak stationary Euler identity
\eqref{eq:sheet-stationary-weak-euler} holds if and only if there is a unique
$\beta\in\re$ such that
\begin{equation}
  \gamma\langle v_{c,\gamma}\rangle_\tau=\beta
  \quad\text{on }c.
  \label{eq:bernoulli-constant}
\end{equation}
\end{proposition}

\begin{proof}
Away from $c$, equation \eqref{eq:sheet-poisson-distribution} says that the
metric vorticity is $-\bar\omega_{c,\gamma}$.  The two-dimensional
acceleration identity and \eqref{eq:stream-function-convention} give
\begin{align*}
  \nabla_v^gv
  &=\grad_g\frac{|v|_g^2}{2}
    -\bar\omega_{c,\gamma}\mathcal Jv
   =\grad_g
    \left(\frac{|v|_g^2}{2}
          -\bar\omega_{c,\gamma}\psi_{c,\gamma}\right).
\end{align*}
On each side of $c$, H\"older's inequality and $H^1\hookrightarrow L^4$
give
\[
 \left|\int g(\nabla_v^gv,X)\, d\mu\right|
 \leq C\|v\|_{H^1}^2\|X\|_{H^1}.
\]
Thus the stationary functional is continuous on the test space of
Definition~\ref{def:weak-steady-sheet}, and the following integration by
parts, first for smooth fields, extends by density.
Integration by parts on $\Omega^-$ and $\Omega^+$, using outer-boundary
tangency of $X$, therefore turns the left-hand side of
\eqref{eq:sheet-stationary-weak-euler} into
\[
  -\int_c
  \left[\frac{|v_{c,\gamma}|_g^2}{2}
  -\bar\omega_{c,\gamma}\psi_{c,\gamma}\right]
  g(X,n)\, d\ell_g.
\]
The stream function has a common trace.  The kinematic condition makes both
velocities tangential, and hence
\begin{align*}
  \left[\frac{|v_{c,\gamma}|_g^2}{2}\right]
  &=\langle v_{c,\gamma}\rangle_\tau
    [g(v_{c,\gamma},\tau)]
   =\gamma\langle v_{c,\gamma}\rangle_\tau
\end{align*}
by \eqref{eq:sheet-jump-relations}.  Every admissible test field has
$\int_cg(X,n)\, d\ell_g=0$.  Thus
\eqref{eq:bernoulli-constant} implies the weak Euler identity.

Conversely, Lemma~\ref{lem:normal-trace-realization} realizes every smooth
zero-mean function as $g(X,n)|_c$.  The weak identity therefore makes
$\gamma\langle v_{c,\gamma}\rangle_\tau$ annihilate all zero-mean test
functions.  It is a constant distribution on the connected curve and, by
the trace regularity in Proposition~\ref{prop:weak-stream-function}, an
almost-everywhere constant function.  Its value $\beta$ is unique.
\end{proof}

The remainder of the paper uses the Euclidean frame fixed in
Definition~\ref{def:reference-state}; metric-unit normals and tangents will
always be indicated explicitly.

\subsection{Bifurcation-theoretic facts}
\label{subsec:bifurcation-facts}

The following statements are formulated on abstract Banach spaces.  The
physical spaces and maps are introduced in
Definition~\ref{def:physical-sobolev-spaces} and
Section~\ref{sec:problem-setting}.

\begin{definition}[Fredholm operator and index]
\label{def:fredholm-index}
For Banach spaces $X,Y$, a bounded operator $A\colon X\to Y$ is Fredholm if
$\ker A$ is finite-dimensional, $\ran A$ is closed, and
$\coker A=Y/\ran A$ is finite-dimensional.  Its
index is
\begin{equation}
  \indop A
  =\dim\ker A-\dim\coker A.
  \label{eq:fredholm-index}
\end{equation}
\end{definition}

\begin{fact}[Finite-dimensional index bookkeeping]
\label{fact:finite-dimensional-index}
Let $A\colon X\to Y$ be Fredholm and let $\mathsf q\geq1$.  Adjoining
$\mathsf q$ scalar domain variables by any finite-rank column gives a
Fredholm operator $X\times\re^{\mathsf q}\to Y$ of index
$\indop A+\mathsf q$.  Appending $\mathsf q$ bounded scalar
equations gives a Fredholm operator $X\to Y\times\re^{\mathsf q}$ of
index $\indop A-\mathsf q$.
\end{fact}

\begin{proof}
The zero-column extension is $A\oplus0\colon X\times\re^{\mathsf q}\to Y$
and has kernel $\ker A\times\re^{\mathsf q}$ and the same cokernel as
$A$.  Any other column is a finite-rank perturbation.  Similarly,
$x\mapsto(Ax,0)$ has the same kernel as $A$ and cokernel
$\coker A\times\re^{\mathsf q}$, and arbitrary appended
bounded scalar equations are finite-rank perturbations.  Fredholmness and
index are invariant under finite-rank perturbations.
\end{proof}

\begin{fact}[Parameter-dependent Banach implicit function theorem]
\label{fact:parameter-implicit-function}
Let $I\subset\re$ be open, let $U\subset X$ be open, and let
$\Phi\colon I\times U\to Y$ be $C^2$.  If
$\Phi(s_0,x_0)=0$ and
$D_x\Phi(s_0,x_0)\colon X\to Y$ is an isomorphism, then there are neighborhoods
of $s_0$ and $x_0$ and a unique $C^2$ map $x=x(s)$ in their product such
that $x(s_0)=x_0$ and $\Phi(s,x(s))=0$.  These are the parameter and
regularity conclusions of the Banach implicit function theorem
\cite[Theorem~I.1.1 and Addendum~I.1.7]{Kielhofer2004}.
\end{fact}

\begin{proposition}[Normalization of a Fredholm operator]
\label{prop:fredholm-normalization}
Let $A\colon X\to Y$ be Fredholm, let $\mathsf q\geq1$ be an integer, and let
$\nu\colon X\to\re^{\mathsf q}$ be bounded.  If
$\nu|_{\ker A}$ is onto, then
\begin{equation}
  \widetilde A=(A,\nu)\colon X\longrightarrow
  Y\times\re^{\mathsf q}
  \label{eq:abstract-normalized-operator}
\end{equation}
is Fredholm,
\begin{equation}
  \ker\widetilde A=\ker A\cap\ker\nu,
  \qquad
  \coker\widetilde A\cong\coker A,
  \qquad
  \indop\widetilde A=\indop A-\mathsf q.
  \label{eq:abstract-normalization-data}
\end{equation}
\end{proposition}

\begin{proof}
Only the cokernel identification requires more than
Fact~\ref{fact:finite-dimensional-index}.  Map the class of $(y,a)$ modulo
$\ran\widetilde A$ to the class of $y$ modulo
$\ran A$.  Surjectivity is immediate.  If $y=Ax$, choose
$k\in\ker A$ with $\nu k=a-\nu x$.  Then
$(y,a)=\widetilde A(x+k)$, so the map is injective.  The kernel identity is
immediate, and the index formula follows.
\end{proof}

\begin{proposition}[Parameter-dependent Lyapunov--Schmidt reduction]
\label{prop:abstract-lyapunov-schmidt}
Let $X$ and $Y$ be Banach spaces, let $I\subset\re$ and $U\subset X$
be open with $0\in I$ and $x_0\in U$, and let
$\Phi\colon I\times U\to Y$ be $C^2$.  Suppose
$\Phi(0,x_0)=0$, and suppose
$A=D_x\Phi(0,x_0)\colon X\to Y$ is Fredholm of index zero.  Choose closed
complements
\begin{equation}
  X=\ker A\oplus X_1,
  \qquad
  Y=\ran A\oplus Y_1.
  \label{eq:abstract-ls-splittings}
\end{equation}
There are neighborhoods of $(0,0)$ in
$\re\times\ker A$ and a unique $C^2$ map
$w$ into $X_1$, with $w(0,0)=0$, such that the range component of
$\Phi(s,x_0+k+w(s,k))$ vanishes.  In these neighborhoods,
\begin{equation}
  \Phi(s,x)=0
  \quad\Longleftrightarrow\quad
  x=x_0+k+w(s,k)
  \ \text{and}\ 
  Q\Phi(s,x_0+k+w(s,k))=0,
  \label{eq:abstract-ls-equivalence}
\end{equation}
where $Q\colon Y\to Y_1$ is the projection along $\ran A$.
The last equation is finite-dimensional and is $C^2$ jointly in the
parameter and kernel coordinate.
\end{proposition}

\begin{proof}
Let $P=1-Q$.  The derivative with respect to $w\in X_1$ of
$P\Phi(s,x_0+k+w)$ at the origin is
$A|_{X_1}\colon X_1\to\ran A$.  It is injective by the first
splitting in \eqref{eq:abstract-ls-splittings}, surjective by definition of
the range, and has a bounded inverse by the open mapping theorem.  Apply
Fact~\ref{fact:parameter-implicit-function} with $(s,k)$ as the parameter.
The resulting unique $w$ solves the range equation.  Splitting the full
equation with $P$ and $Q$ proves \eqref{eq:abstract-ls-equivalence}.  Both
$\ker A$ and $Y_1$ are finite-dimensional because $A$ is Fredholm, and the
index-zero hypothesis gives them equal dimension.
\end{proof}

\begin{corollary}[Nondegenerate abstract reduction]
\label{cor:nondegenerate-abstract-reduction}
Under the hypotheses of Proposition~\ref{prop:abstract-lyapunov-schmidt}, if
$\ker A=0$ and $\coker A=0$, then $A$ is an isomorphism and
$\Phi(s,x)=0$ has a unique local $C^2$ branch through $(0,x_0)$.
\end{corollary}

\begin{proof}
The hypotheses make $A$ a bounded bijection; its inverse is bounded by the
open mapping theorem.  Fact~\ref{fact:parameter-implicit-function} applies.
\end{proof}

\begin{corollary}[Derivative of an implicit branch]
\label{cor:implicit-branch-derivative}
Under the hypotheses of Fact~\ref{fact:parameter-implicit-function}, the
branch derivative is
\begin{equation}
  x'(s_0)
  =-D_x\Phi(s_0,x_0)^{-1}D_s\Phi(s_0,x_0).
  \label{eq:abstract-implicit-branch-derivative}
\end{equation}
\end{corollary}

\begin{proof}
Differentiate $\Phi(s,x(s))=0$ at $s=s_0$ and apply the inverse of
$D_x\Phi(s_0,x_0)$.
\end{proof}

When the abstract kernel is nonzero,
Proposition~\ref{prop:abstract-lyapunov-schmidt} identifies the local
solution set with the zero set of its finite-dimensional reduced equation.
No number, scale, or regularity of branches follows without additional
finite-dimensional transversality hypotheses.

\subsection{Elliptic and boundary-integral inputs}
\label{subsec:elliptic-boundary-inputs}

We now fix the physical standing data before introducing metric-dependent
elliptic objects.

\begin{definition}[Prescribed metric deformation]
\label{def:prescribed-metric-deformation}
Let $(-s_0,s_0)$ be an open interval and let
\begin{equation}
  s\longmapsto\sigma_s
  \quad\text{belong to}\quad
  C^3((-s_0,s_0);C^\infty(\overline{\mathbb D})).
  \label{eq:metric-path-topology}
\end{equation}
Here $C^\infty$ has its standard Fr\'echet topology, so all derivatives in
$s$ through order three are continuous in every spatial $C^k$ seminorm.
Set $\sigma_0=\left.\sigma_s\right|_{s=0}$ and
\begin{equation}
  g_s= e^{2\sigma_s}g_{\mathrm e},
  \qquad  d\mu_{g_s}= e^{2\sigma_s} d A_{\mathrm e},
  \qquad  d\ell_{g_s}= e^{\sigma_s} d\ell_{\mathrm e}.
  \label{eq:prescribed-conformal-metrics}
\end{equation}
The parameter $s$ is prescribed external data.
\end{definition}

\begin{definition}[Reference state]
\label{def:reference-state}
Fix a smooth connected weak steady vortex sheet $(c_0,\gamma_0)$ for $g_0$.
Parametrize $c_0$ by Euclidean arclength
$t\in\re/|c_0|\ze$.  Its Euclidean unit normal $n$ points
from $\Omega^-$ to $\Omega^+$, and $\tau=\mathcal J n=c_0'$.
Define the signed Euclidean curvature $\kappa$ by
\begin{equation}
  n'=\kappa\tau,
  \qquad \tau'=-\kappa n.
  \label{eq:reference-frenet}
\end{equation}
Let $\mu_0$ and $\beta_0$ be the unique constants
in Propositions~\ref{prop:streamline-characterization} and
\ref{prop:bernoulli-characterization}, and set
\begin{equation}
  \psi_0=\psi_{c_0,\gamma_0},
  \qquad
  v_0=-\jgrad_{g_0}\psi_0,
  \qquad
  u_0=(0,0,\mu_0,\beta_0).
  \label{eq:reference-unknown}
\end{equation}
Write
\begin{equation}
  v_{\tau,0}^{\pm}
  =g_0\bigl(v_0^\pm, e^{-\sigma_0\circ c_0}\tau\bigr),
  \qquad
  \langle v_0\rangle_{\tau}
  =\frac{v_{\tau,0}^{+}+v_{\tau,0}^{-}}{2}
  \label{eq:reference-tangential-velocities}
\end{equation}
for the one-sided scalar tangential velocities and their average.
\end{definition}

\begin{definition}[Physical Sobolev spaces]
\label{def:physical-sobolev-spaces}
Fix an integer $m\geq3$ and write
$H^r=H^r(\re/|c_0|\ze;\re)$.  The physical unknown and residual
spaces are
\begin{equation}
  \mathcal X=H^{m+1}\times H^m\times\re^2,
  \qquad
  \mathcal Y=H^{m+1}\times H^m.
  \label{eq:physical-spaces}
\end{equation}
The period is the Euclidean length $|c_0|$ and the Fourier modes are
$e_k(t)= e^{2\pi i kt/|c_0|}$.
\end{definition}

\begin{fact}[Periodic Sobolev calculus]
\label{fact:periodic-sobolev-calculus}
For $r>1/2$, $H^r(\re/|c_0|\ze)$ is a Banach algebra, multiplication
is continuous on $H^r$, and inversion is smooth on the open subset bounded
away from zero.  Smooth Nemytskii maps act smoothly on $H^r$.  In
particular, for $m\geq3$, composition of a smooth ambient function with
$c_0+fn$, the derivative, norm, reciprocal norm, unit tangent, and unit
normal have the Sobolev regularities dictated by
$f\in H^{m+1}$.  These maps are smooth on every open set on which the
denominators stay uniformly nonzero.  The trace maps
  $H^{r+1/2}(\Omega^\pm)\to H^r(c_0)$ are bounded at the orders used here.
\end{fact}

\begin{proof}
The Fourier convolution estimate and Cauchy--Schwarz prove the algebra
property for $r>1/2$.  The Moser estimates, obtained from the chain rule at
integer orders and fractional difference quotients at the remaining orders,
give smooth inversion and Nemytskii regularity on bounded ranges separated
from any singularities.
The composition assertion follows by differentiating the ambient function
and applying the same multiplication estimate; approximation by smooth
functions gives the Sobolev statement.  The trace assertion follows after
flattening the smooth boundary and applying the one-dimensional Fourier
trace estimate.
\end{proof}

\begin{theorem}[Parameter-dependent Dirichlet regularity]
\label{thm:parameter-dirichlet-regularity}
Let $r\geq0$.  For
$q_s\in H^r(\mathbb{D})$ and
$b_s\in H^{r+3/2}(\partial\mathbb{D})$, the problem
\begin{equation}
  -\triangle_{g_s}u_s=q_s\quad\text{in }\mathbb{D},
  \qquad u_s=b_s\quad\text{on }\partial\mathbb{D}
  \label{eq:parameter-dirichlet-problem}
\end{equation}
has a unique solution $u_s\in H^{r+2}(\mathbb{D})$ and, on compact
$s$-intervals,
\begin{equation}
  \|u_s\|_{H^{r+2}(\mathbb{D})}
  \leq C_r
  \bigl(\|q_s\|_{H^r(\mathbb{D})}
       +\|b_s\|_{H^{r+3/2}(\partial\mathbb{D})}\bigr).
  \label{eq:parameter-dirichlet-estimate}
\end{equation}
If $s\mapsto(q_s,b_s)$ is $C^j$, $0\leq j\leq3$, in the displayed
spaces, then $s\mapsto u_s$ is $C^j$ in $H^{r+2}$.
\end{theorem}

\begin{proof}
In two dimensions
$-\triangle_{g_s}u_s=- e^{-2\sigma_s}\triangle_{\mathrm e}u_s$, so
\eqref{eq:parameter-dirichlet-problem} is the fixed Euclidean equation
\[
  -\triangle_{\mathrm e}u_s= e^{2\sigma_s}q_s,
  \qquad u_s=b_s.
\]
Choose a bounded right inverse of the boundary trace and subtract the lift
of $b_s$.  The zero-Dirichlet Euclidean Laplacian is an isomorphism
$H^{r+2}(\mathbb{D})\cap H^1_0(\mathbb{D})\to H^r(\mathbb{D})$; its usual
Fourier or boundary-flattening estimate gives
\eqref{eq:parameter-dirichlet-estimate}.  Multiplication by
$ e^{2\sigma_s}$ is bounded on $H^r$ by
Fact~\ref{fact:periodic-sobolev-calculus} and its fixed-domain analogue.
The solution is the composition of a fixed bounded inverse with the smooth
maps $(\sigma,q)\mapsto  e^{2\sigma}q$ and the boundary lift.  Differentiating
this explicit composition through order $j$ proves the parameter-dependence
assertion.
\end{proof}

\begin{theorem}[Transmission regularity]
\label{thm:transmission-regularity}
Let $c_0$ be the smooth interface in
Definition~\ref{def:reference-state}, let $r\geq0$, and prescribe
\[
  p^\pm\in H^r(\Omega^\pm),\qquad
  a\in H^{r+3/2}(c_0),\qquad
  b\in H^{r+1/2}(c_0),\qquad
  d\in H^{r+3/2}(\partial\mathbb{D}).
\]
There is a unique pair $u^\pm\in H^{r+2}(\Omega^\pm)$ satisfying
\begin{equation}
\begin{aligned}
  -\triangle_{\mathrm e}u^\pm&=p^\pm &&\text{in }\Omega^\pm,\\
  [u]&=a, & [\partial_{n}u]&=b &&\text{on }c_0,\\
  u^+&=d &&\text{on }\partial\mathbb{D},
\end{aligned}
  \label{eq:transmission-problem}
\end{equation}
where the same Euclidean normal $n$ from $\Omega^-$ to $\Omega^+$ is used
in both normal traces.  Moreover,
\begin{equation}
  \|u^+\|_{H^{r+2}}+\|u^-\|_{H^{r+2}}
  \leq C_r\bigl(
  \|p^+\|_{H^r}+\|p^-\|_{H^r}
  +\|a\|_{H^{r+3/2}}+\|b\|_{H^{r+1/2}}
  +\|d\|_{H^{r+3/2}}\bigr).
  \label{eq:transmission-estimate}
\end{equation}
There is no scalar integral compatibility condition; the outer Dirichlet
condition removes the constant kernel.
\end{theorem}

\begin{proof}
The joint trace map consisting of the Dirichlet and normal traces has a
bounded right inverse on each smooth side of $c_0$.  Because $c_0$ and
$\partial\mathbb{D}$ are disjoint, these right inverses may be localized to
choose $\ell^\pm\in H^{r+2}(\Omega^\pm)$ such that
\[
 [\ell]=a,
 \qquad [\partial_{n}\ell]=b,
 \qquad \ell^+=d\quad\hbox{on }\partial\mathbb{D},
\]
with the corresponding norm bounded by the three boundary-data norms in
\eqref{eq:transmission-estimate}.  Define the piecewise function
$q|_{\Omega^\pm}=p^\pm+\triangle_{\mathrm e}\ell^\pm$.  The global
zero-Dirichlet problem
\[
 \int_{\mathbb{D}}\nabla w\cdot\nabla\phi\, d A_{\mathrm e}
 =\int_{\mathbb{D}}q\phi\, d A_{\mathrm e},
 \qquad \phi\in H^1_0(\mathbb{D}),
\]
has a unique solution $w\in H^1_0(\mathbb{D})$.  Its common trace and
weak flux balance give
$[w]=[\partial_{n}w]=0$ on $c_0$.  Hence
$u^\pm=w|_{\Omega^\pm}+\ell^\pm$ solves
\eqref{eq:transmission-problem}.

For completeness of the regularity step, flatten $c_0$ and freeze the
principal coefficients.  At every nonzero tangential frequency $\nu$, the
decaying homogeneous solutions on the two half-planes have interface matrix
\[
  \begin{pmatrix}1&-1\\-|\nu|&-|\nu|\end{pmatrix},
\]
whose determinant is $-2|\nu|$.  The complementing condition, localization,
and the ordinary Dirichlet estimate at the outer boundary therefore give
piecewise $H^{r+2}$ regularity and \eqref{eq:transmission-estimate}, initially
with a lower-order $L^2$ term.  For homogeneous data the two pieces glue to
a function in $H^1_0(\mathbb{D})$; testing its weak harmonic equation by the
function itself proves uniqueness.  The standard contradiction argument
then removes the lower-order term.  The global variational construction
also shows directly that no scalar integral compatibility condition occurs.
\end{proof}

\begin{definition}[Green and torsion functions]
\label{def:green-torsion}
For the fixed kernel in \eqref{eq:disk-green-kernel} and the metric path in
Definition~\ref{def:prescribed-metric-deformation}, define
\begin{equation}
  \Psi_s(z)=\int_{\mathbb{D}}G(z,y)\, d\mu_{g_s}(y)
  =\int_{\mathbb{D}}G(z,y) e^{2\sigma_s(y)}\, d A_{\mathrm e}(y).
  \label{eq:torsion-function}
\end{equation}
\end{definition}

\begin{fact}[Green representation for the torsion function]
\label{fact:torsion-green-representation}
The function in Definition~\ref{def:green-torsion} is the unique solution of
\begin{equation}
  -\triangle_{\mathrm e}\Psi_s= e^{2\sigma_s}
  \quad\text{in }\mathbb{D},
  \qquad \Psi_s|_{\partial\mathbb{D}}=0.
  \label{eq:torsion-equation}
\end{equation}
It is smooth in the spatial variable and $C^3$ in $s$ with values in every
finite Sobolev space.
\end{fact}

\begin{proof}
Apply $-\triangle_{\mathrm e,z}$ to
\eqref{eq:torsion-function} in distributions and use
\eqref{eq:disk-green-kernel}.  The boundary value follows from the same
kernel identity.  Uniqueness and the regularity and parameter dependence are
Theorem~\ref{thm:parameter-dirichlet-regularity} with $q_s=1$
and zero boundary data.
\end{proof}

\begin{definition}[Fixed exact layer potentials and boundary operators]
\label{def:fixed-layer-operators}
For a density $\rho$ per Euclidean arclength on the reference curve, define
\begin{align}
  \mathcal V[\rho](z)
  &=\int_{c_0}G(z,y)\rho(y)\, d\ell_{\mathrm e}(y),
  &
  \mathcal W[\rho](z)
  &=\int_{c_0}\partial_{n_y}G(z,y)\rho(y)\, d\ell_{\mathrm e}(y),
  \label{eq:fixed-layer-potentials}\\
  \mathcal S\rho&=\Tr\mathcal V[\rho],
  &
  \mathcal D\rho&=\langle\Tr\mathcal W[\rho]\rangle,
  \nonumber\\
  \mathcal D'\rho&=
  \langle\partial_{n}\mathcal V[\rho]\rangle,
  &
  \mathcal T\rho&=\partial_{n}\mathcal W[\rho].
  \label{eq:fixed-boundary-operators}
\end{align}
The normal $n_y$ in $\mathcal W$ is the Euclidean source normal,
whereas $n$ in $\mathcal D'$ and $\mathcal T$ is the Euclidean target normal.  All
operators are pulled back by the Euclidean-arclength parametrization of
$c_0$.
\end{definition}

\begin{fact}[Layer-potential jump relations]
\label{fact:layer-jump-relations}
With $+$ denoting the outer side, $-$ the inner side, and
$[q]=q^+-q^-$, the fixed exact potentials satisfy
\begin{equation}
\begin{aligned}
  \partial_{n}\mathcal V[\rho]^\pm
  &=\mathcal D'\rho\mp\frac12\rho,
  & [\partial_{n}\mathcal V[\rho]]&=-\rho,\\
  \mathcal W[\rho]^\pm
  &=\mathcal D\rho\pm\frac12\rho,
  & [\mathcal W[\rho]]&=\rho,\\
  [\partial_{n}\mathcal W[\rho]]&=0.
\end{aligned}
  \label{eq:fixed-layer-jumps}
\end{equation}
\end{fact}

\begin{proof}
The smooth part of $G$ has no jump.  Flattening $c_0$ and integrating the
normal derivative of $-(2\pi)^{-1}\log|z-y|$ across a short normal segment
gives the two half-density terms with the signs in
\eqref{eq:fixed-layer-jumps}.  The double-layer formulas follow by the same
calculation with the source normal.  Averaging defines the principal-value
terms in Definition~\ref{def:fixed-layer-operators} and fixes all signs.
\end{proof}

\begin{fact}[Fixed boundary-operator mapping properties]
\label{fact:fixed-boundary-mappings}
For every real $r$, the fixed exact operators extend boundedly as
\begin{equation}
  \mathcal S\colon H^r\to H^{r+1},
  \qquad
  \mathcal D,\mathcal D'\colon H^r\to H^{r+1},
  \qquad
  \mathcal T\colon H^{r+1}\to H^r.
  \label{eq:fixed-boundary-mappings}
\end{equation}
The first and last maps have orders $-1$ and $1$; the middle maps are
lower-order for the principal analysis.  The one-sided potentials have the
corresponding trace regularity.  These are the one-dimensional specialization
of the pseudo-homogeneous mapping theorem
\cite[Theorem~2.6]{CostabelLeLouer2012}.
\end{fact}

\begin{proof}
Pull the kernel back by Euclidean arclength.  The singular part of
$\mathcal S$ is
\[
 -(2\pi)^{-1}\log\left|2\sin\frac{\pi(t-t')}{|c_0|}\right|,
\]
and its nonzero Fourier multiplier is $|c_0|/(4\pi|k|)$.  The difference
from this kernel is smooth.  The
double-layer and adjoint double-layer quotients have removable diagonal
singularities and are lower-order.  Finally, tangential integration by parts
gives the Maue formula
\[
  \mathcal T=\partial_t\mathcal S\partial_t+R,
\]
where $R$ has a smooth kernel.  The Fourier multiplier estimate proves every
map in \eqref{eq:fixed-boundary-mappings}.
\end{proof}

\begin{definition}[Model principal operators]
\label{def:model-principal-operators}
On the complexified periodic scale, define the model principal operators by
\begin{equation}
\begin{aligned}
  \mathcal S_0e_k&=\frac{|c_0|}{4\pi|k|}e_k,
  &\mathcal T_0e_k&=-\frac{\pi|k|}{|c_0|}e_k,
  &&k\neq0,\\
  \mathcal S_0e_0&=0,
  &\mathcal T_0e_0&=0.
\end{aligned}
  \label{eq:model-operator-multipliers}
\end{equation}
They restrict to real operators and have orders $-1$ and $1$.
\end{definition}

\begin{fact}[Model operator identity]
\label{fact:model-operator-identity}
For every $k\neq0$,
\begin{equation}
  \mathcal S_0\mathcal T_0e_k
  =\mathcal T_0\mathcal S_0e_k=-\frac14e_k,
  \label{eq:model-product-nonzero-modes}
\end{equation}
whereas both compositions vanish on $e_0$.
\end{fact}

\begin{proof}
Multiply the two Fourier multipliers in
\eqref{eq:model-operator-multipliers}; the zero-mode assertion follows from
their separately declared zero-mode values.
\end{proof}

\begin{lemma}[Smooth dependence of moving boundary operators]
\label{lem:moving-boundary-operators}
There is an open neighborhood of $0$ in $H^{m+1}$ such that every
$c_f=c_0+fn$ with $f$ in that neighborhood is embedded, remains in a fixed compact subset of
$\mathbb{D}^\circ$, and has the reference orientation.  Let $\Omega_f^-$ be
the bounded component of $\mathbb{D}\setminus c_f$ and let $\Omega_f^+$
contain $\partial\mathbb{D}$.  For a parameter density
$\rho(t)\, d t$, set
\begin{align}
  \mathcal V_f[\rho](z)
  &=\int_0^{|c_0|}G(z,c_f(t'))\rho(t')\, d t',
  \label{eq:generic-moving-potential}\\
  \mathcal S_f\rho(t)
  &=\mathcal V_f[\rho](c_f(t)),
  \nonumber\\
  \mathcal D_f'\rho(t)
  &=\left\langle
    \partial_{n_f(t)}\mathcal V_f[\rho]
    \right\rangle(c_f(t)),
  \label{eq:generic-moving-boundary-operators}
\end{align}
where $n_f$ is the target Euclidean unit normal.  Then
\begin{equation}
\begin{aligned}
  f&\longmapsto\mathcal S_f
    &&\text{is }C^2\text{ with values in }
      \mathcal L(H^m,H^{m+1}),\\
  f&\longmapsto\mathcal D_f'
    &&\text{is }C^2\text{ with values in }
      \mathcal L(H^m,H^m).
\end{aligned}
  \label{eq:moving-operator-types}
\end{equation}
For each $f$ in that neighborhood, $\mathcal V_f$ is bounded from $H^m$ to
$H^1_0(\mathbb{D})$, and its restrictions to $\Omega_f^\pm$ belong to
$H^{m+3/2}(\Omega_f^\pm)$.  After shrinking the neighborhood, for $j=0,1,2$ there is a
uniform $C$ such that
\begin{align}
 &\|D_f^j\mathcal S_f[p_1,\ldots,p_j]\rho\|_{H^{m+1}}
 +\|D_f^j\mathcal D_f'[p_1,\ldots,p_j]\rho\|_{H^m}
 \nonumber\\
 &\hspace{35mm}\leq
 C\|\rho\|_{H^m}\prod_{i=1}^j\|p_i\|_{H^{m+1}}.
  \label{eq:moving-operator-uniform-estimate}
\end{align}
In addition,
\begin{equation}
 \|\mathcal V_f[\rho]\|_{H^1_0(\mathbb{D})}
 \leq C\|\rho\|_{H^m}.
 \label{eq:moving-global-potential-bound}
\end{equation}
If $f$ and $\rho$ are smooth, the potential is smooth to either side and its one-sided
normal derivatives satisfy
\begin{equation}
  [\partial_{n_f}\mathcal V_f[\rho]]=-\frac{\rho}{|c_f'|}.
  \label{eq:moving-density-jump}
\end{equation}
\end{lemma}

\begin{proof}
The tubular-neighborhood theorem gives a radius on which the normal
projection to $c_0$ is one-to-one.  Since
$H^{m+1}(\re/|c_0|\ze)\hookrightarrow C^2$ for $m\geq3$, a sufficiently
small $H^{m+1}$ ball keeps $c_f$ inside that tube, away from
$\partial\mathbb{D}$, with $|c_f'|$ bounded below and the original
orientation.  The Jordan curve theorem fixes the two components and their
side labels.

Write
$G(x,y)=-(2\pi)^{-1}\log|x-y|+G_{\mathrm{sm}}(x,y)$, where
$G_{\mathrm{sm}}$ is smooth on the fixed compact set containing all curves.
The pulled-back logarithm has the
decomposition
\begin{align*}
  \log|c_f(t)-c_f(t')|
  &=\log\left|2\sin\frac{\pi(t-t')}{|c_0|}\right|\\
  &\quad+
  \log\left|
  \frac{c_f(t)-c_f(t')}
       {2\sin(\pi(t-t')/|c_0|)}
  \right|.
\end{align*}
The quotient extends across the diagonal because
\[
  \frac{c_f(t)-c_f(t')}{t-t'}
  =\int_0^1c_f'(t'+\lambda(t-t'))\, d\lambda.
\]
It is uniformly nonzero near the diagonal by the lower bound for
$|c_f'|$, and away from the diagonal by embeddedness and compactness.
A finite-Sobolev step turns the smooth kernel calculus
into the Fr\'echet assertion in \eqref{eq:moving-operator-types}.  This step
is needed because the cited shape results are stated first for smooth
deformations.  Put
\[
 \delta_p(t,t')=p(t)n(t)-p(t')n(t').
\]
For the singular part of the single-layer kernel, its $j$th directional
kernel, $j\leq3$, is a finite sum of terms of the form
\begin{equation}
 D^j\log|c_f(t)-c_f(t')|
 [\delta_{p_1}(t,t'),\ldots,\delta_{p_j}(t,t')],
 \label{eq:moving-kernel-directional-form}
\end{equation}
with some directions omitted in lower-order summands.  Since
\begin{equation}
 \delta_p(t,t')
 =(t-t')\int_0^1
   \partial_t(pn)\bigl(t'+\lambda(t-t')\bigr)\, d\lambda,
 \label{eq:moving-direction-difference}
\end{equation}
each differentiation of the logarithm is accompanied by a cancelling
factor $t-t'$.  After applying the target normal, use
\[
 n_f(t)
 =\frac{(1+\kappa f)n-(\partial_t f)\tau}
        {\bigl((1+\kappa f)^2+(\partial_t f)^2\bigr)^{1/2}}.
\]
For smooth $f$ and directions, this formula and its first three shape
derivatives have the pseudo-homogeneous classes $-1$ for $\mathcal S_f$ and
$0$ for $\mathcal D_f'$.  The operators in
\cite[Theorem~4.5 and Corollary~4.6]{CostabelLeLouer2012} use the moving
surface measure.  Since
\[
 \rho(t)\, d t=\frac{\rho(t)}{|c_f'(t)|}\, d\ell_{\mathrm e,c_f},
\]
our operators are obtained from those results by precomposition with
$M_{|c_f'|^{-1}}$; this multiplier and its first two derivatives have the
required Sobolev bounds.  The mapping theorem is
\cite[Theorem~2.6]{CostabelLeLouer2012}.  We now verify that the constants
needed here depend only on the finite Sobolev norms.

We extend those smooth operators to the finite Sobolev chart by a fixed-domain
transmission argument.  Put $r_*=m+3/2$ and $\Omega_0^\pm:=\Omega^\pm$, and write
\[
 \tr_0\vartheta
 \coloneqq(\Tr_{c_0(\re/|c_0|\ze)}\vartheta)\circ c_0
\]
for the trace pulled back to $\re/|c_0|\ze$.  Choose
$\chi_{\mathrm{cut}}\in C_c^\infty([0,\infty))$ equal to one near zero.  On a
flat boundary the formula
\[
 \widehat{\mathsf E a}(\nu,r)
 =\chi_{\mathrm{cut}}(r)e^{-r(1+|\nu|^2)^{1/2}}\widehat a(\nu)
\]
and Plancherel's theorem give a bounded right inverse
$H^{m+1}(\re/|c_0|\ze)\to H^{r_*}$ for the trace.  Boundary
charts and a partition of unity therefore give vector-valued right inverses
\[
  \mathsf E_\pm\colon H^{m+1}(\re/|c_0|\ze;\re^2)
 \longrightarrow H^{r_*}(\Omega_0^\pm;\re^2),
\]
chosen so that $\tr_0 \mathsf E_\pm a=a$ and with zero outer
trace on the plus side.  Define
\(\Phi_f^\pm=\id+ \mathsf E_\pm(fn)\).  Because
$H^{r_*}$ embeds into $C^1$ in two dimensions, a smaller neighborhood makes these
uniformly bi-Lipschitz Sobolev diffeomorphisms from $\Omega_0^\pm$ onto
$\Omega_f^\pm$.

Set, only in this proof,
\[
 a_f^\pm
 \coloneqq\det(D\Phi_f^\pm)(D\Phi_f^\pm)^{-1}(D\Phi_f^\pm)^{-\mathsf T}
 \in H^{m+1/2}(\Omega_0^\pm;\re^{2\times2}).
\]
In two space dimensions, $H^{m+1/2}(\Omega_0^\pm)$ is a Banach algebra.
Matrix inversion on the open set of uniformly positive determinants and the
tame product estimate show that $f\mapsto a_f^\pm$ is $C^3$ from $H^{m+1}$
to $H^{m+1/2}$, with each derivative through order three bounded by the
product of the direction norms.

Let
\begin{align*}
 \mathsf H_{\mathrm{dom}}
 \coloneqq\{&(\vartheta^+,\vartheta^-)\in H^{r_*}(\Omega_0^+)\times H^{r_*}(\Omega_0^-):
       \tr_0 \vartheta^+=\tr_0 \vartheta^-,\\
     &\hspace{55mm}
       \Tr_{\partial\mathbb{D}}\vartheta^+=0\},
\end{align*}
and
\[
 \mathsf H_{\mathrm{tar}}
 \coloneqq H^{m-1/2}(\Omega_0^+)\times H^{m-1/2}(\Omega_0^-)
   \times H^m(\re/|c_0|\ze).
\]
On these fixed spaces define
\begin{equation}
 \begin{aligned}
 \mathcal P_f(\vartheta^+,\vartheta^-)
 =\bigl(&-\dvg(a_f^+\nabla\vartheta^+),
          -\dvg(a_f^-\nabla\vartheta^-),\\
        &n\cdot\tr_0(a_f^+\nabla\vartheta^+)
          -n\cdot\tr_0(a_f^-\nabla\vartheta^-)\bigr).
 \end{aligned}
 \label{eq:moving-fixed-transmission-operator}
\end{equation}
The product and trace estimates give
\begin{equation}
 \|\mathcal P_f-\mathcal P_{\widetilde f}\|_{\mathcal L(\mathsf H_{\mathrm{dom}},
                                             \mathsf H_{\mathrm{tar}})}
 \leq C\|f-\widetilde f\|_{H^{m+1}},
 \label{eq:moving-transmission-operator-difference}
\end{equation}
For $f$ and $\widetilde f$ in this neighborhood, these estimates also prove that
$f\mapsto\mathcal P_f$ is $C^2$ in that operator norm.  We verify the required
inverse at $f=0$ independently of
Theorem~\ref{thm:transmission-regularity}.  At a flat interface, tangential Fourier
transformation gives for the two homogeneous modes the jump-symbol matrix
\[
 \begin{pmatrix}1&-1\\-|\nu|&-|\nu|\end{pmatrix},
 \qquad \det=-2|\nu|.
\]
Given flat-model data
$p^\pm\in H^{m-1/2}$ and $\mathsf j\in H^m$, let $q^\pm$ solve the two
zero-interface-trace half-space Dirichlet problems.  Their tangential Fourier
representations, obtained by variation of constants in the normal variable,
give
\[
 \|q^\pm\|_{H^{m+3/2}}
 +\|\left.\partial_{n}q^\pm\right|_{x_2=0}\|_{H^m}
 \leq C\|p^\pm\|_{H^{m-1/2}}.
\]
The remaining harmonic correction has Fourier amplitudes determined by
\[
 \begin{pmatrix}1&-1\\- |\nu|&- |\nu|\end{pmatrix}
 \begin{pmatrix}A^+\\A^-\end{pmatrix}
 =\begin{pmatrix}0\\
 \mathsf j-[\partial_{n}q]\end{pmatrix}.
\]
Thus $A^+=A^-=-(\mathsf j-[\partial_{n}q])/(2|\nu|)$ for $\nu\ne0$, and the
Poisson-extension estimate yields
\begin{align*}
 \|\vartheta^+\|_{H^{m+3/2}}+\|\vartheta^-\|_{H^{m+3/2}}
 \leq C\bigl(&\|p^+\|_{H^{m-1/2}}+\|p^-\|_{H^{m-1/2}}\\
              &+\|\mathsf j\|_{H^m}\bigr)
\end{align*}
at nonzero tangential frequencies.  Localizing along the smooth fixed
interface and freezing its
coefficients preserves this estimate at high frequency; the commutators lose
one order and are absorbed there.  The remaining finite frequency band is
controlled by the $H^1$ energy estimate with the fixed outer Dirichlet
condition.  More precisely, on the common-trace, outer-zero $H^1$ space, the
weak equation for arbitrary $(p^+,p^-,\mathsf j)\in \mathsf H_{\mathrm{tar}}$ is,
for every common-trace, outer-zero test pair $(\widehat\vartheta^+,\widehat\vartheta^-)$,
\[
 \sum_\pm\int_{\Omega_0^\pm}\nabla\vartheta^\pm\cdot\nabla\widehat\vartheta^\pm\, d A_{\mathrm e}
 =\sum_\pm\int_{\Omega_0^\pm}p^\pm\widehat\vartheta^\pm\, d A_{\mathrm e}
  -\langle \mathsf j,\tr_0\widehat\vartheta\rangle.
\]
The right side is bounded by the trace theorem, while the left side is
coercive by the outer-zero Poincar\'e inequality.  Lax--Milgram therefore gives
the common-trace $H^1$ solution.  The localized estimate upgrades it to
$\mathsf H_{\mathrm{dom}}$, and the same energy identity gives uniqueness.
Thus
\(\mathcal P_0\colon \mathsf H_{\mathrm{dom}}\to \mathsf H_{\mathrm{tar}}\)
is an isomorphism.

Shrink the neighborhood once more so that
\(\|\mathcal P_0^{-1}(\mathcal P_f-\mathcal P_0)\|<1/2\).
The Neumann series makes every $\mathcal P_f$ invertible with a uniform
inverse.  For $\rho\in H^m$ define
\[
 (\upsilon_f^+[\rho],\upsilon_f^-[\rho])
 \coloneqq\mathcal P_f^{-1}(0,0,-\rho).
\]
The identities
\begin{align*}
 D(\mathcal P^{-1})[p]
 &=-\mathcal P^{-1}(D\mathcal P[p])\mathcal P^{-1},\\
 D^2(\mathcal P^{-1})[p,q]
 &=\mathcal P^{-1}D\mathcal P[q]\mathcal P^{-1}D\mathcal P[p]\mathcal P^{-1}
   +\mathcal P^{-1}D\mathcal P[p]\mathcal P^{-1}D\mathcal P[q]\mathcal P^{-1}\\
 &\quad-\mathcal P^{-1}D^2\mathcal P[p,q]\mathcal P^{-1}
\end{align*}
show that $f\mapsto\upsilon_f$ is $C^2$ with values in
$\mathcal L(H^m,\mathsf H_{\mathrm{dom}})$ and give, for $j=0,1,2$,
\begin{equation}
 \|D_f^j\upsilon_f[p_1,\ldots,p_j]\rho\|_{\mathsf H_{\mathrm{dom}}}
 \leq C\|\rho\|_{H^m}\prod_{i=1}^j\|p_i\|_{H^{m+1}}.
 \label{eq:moving-kernel-finite-sobolev-bound}
\end{equation}

Before comparing the two constructions, we establish the global energy
class of the Green potential.  The moving-curve trace estimate is uniform on
the chosen neighborhood, and hence
\[
 \left|\int_0^{|c_0|}\rho(t)\varpi(c_f(t))\, d t\right|
 \leq C\|\rho\|_{H^{-1/2}}\|\varpi\|_{H^1_0(\mathbb{D})}
 \leq C\|\rho\|_{H^m}\|\varpi\|_{H^1_0(\mathbb{D})}.
\]
Lax--Milgram gives a unique $U_f[\rho]\in H^1_0(\mathbb{D})$ with this
right-hand side and the corresponding uniform norm bound.  Testing the
Dirichlet Green identity first against smooth densities and then using
density in $H^m$ identifies $U_f[\rho]$ with the integral
\eqref{eq:generic-moving-potential}.  Thus
$\mathcal V_f[\rho]\in H^1_0(\mathbb{D})$ and
\eqref{eq:moving-global-potential-bound} already holds.

Define
\[
 \widetilde\upsilon_f|_{\Omega_f^\pm}
 \coloneqq\upsilon_f^\pm[\rho]\circ(\Phi_f^\pm)^{-1}.
\]
For every $\varpi\in H^1_0(\mathbb{D})$, the change-of-variables formula and
\eqref{eq:moving-fixed-transmission-operator} give
\[
 \int_{\mathbb{D}}\nabla_{\mathrm e}\widetilde\upsilon_f\cdot\nabla_{\mathrm e}\varpi\, d A_{\mathrm e}
 =\int_0^{|c_0|}\rho(t)\varpi(c_f(t))\, d t,
 \qquad \widetilde\upsilon_f|_{\partial\mathbb{D}}=0.
\]
The Green potential $\mathcal V_f[\rho]$ is the $H^1_0$ solution just
constructed, so energy uniqueness identifies it with
$\widetilde\upsilon_f$.  Consequently
\[
 \mathcal S_f\rho=\tr_0\upsilon_f[\rho],
\]
and
\[
 \mathcal D_f'\rho
 =\frac12\sum_{\pm}
   n_f\cdot\tr_0\!\left(
   (D\Phi_f^\pm)^{-\mathsf T}\nabla\upsilon_f^\pm[\rho]\right).
\]
The trace map in the first identity has target $H^{m+1}$; the normal-gradient
trace in the second has target $H^m$.  Their coefficients depend $C^2$ on
$f$.  Together with \eqref{eq:moving-kernel-finite-sobolev-bound}, these
identities prove \eqref{eq:moving-operator-types} and
\eqref{eq:moving-operator-uniform-estimate}.

The same trace-extension construction and the inverse-function identity give
uniform $H^{r_*}$ bounds for $(\Phi_f^\pm)^{-1}$ on the smaller neighborhood.
For every $\vartheta_0\in H^{r_*}(\Omega_0^\pm)$, the integer chain rule through
order $m+1$, the $H^{m+1/2}$ algebra estimate, and the bi-Lipschitz change of
variables in the remaining one-half Slobodeckij seminorm give
\[
 \|\vartheta_0\circ(\Phi_f^\pm)^{-1}\|_{H^{r_*}(\Omega_f^\pm)}
 \leq C\|\vartheta_0\|_{H^{r_*}(\Omega_0^\pm)}.
\]
This proves the asserted one-sided $H^{m+3/2}$ bound.  The common trace glues
the two pieces consistently with the $H^1_0$ solution constructed above.
Finally, the Piola identity gives
\[
 n\cdot\tr_0(a_f^\pm\nabla\upsilon_f^\pm)
 =|c_f'|\,(\partial_{n_f}\widetilde\upsilon_f^\pm)\circ c_f,
\]
so the third component of \eqref{eq:moving-fixed-transmission-operator}
proves \eqref{eq:moving-density-jump}.

If $f_N$ and $\rho_N$ are Fourier truncations, then
$a_{f_N}^\pm\to a_f^\pm$ in $H^{m+1/2}$ and
$\mathcal P_{f_N}\to\mathcal P_f$ in operator norm.  The $C^3$ coefficient
dependence also gives
\[
 D_f^j\mathcal P_{f_N}\longrightarrow D_f^j\mathcal P_f
\]
in the corresponding multilinear operator norm for $j=1,2$.  The resolvent
identity
\[
 \mathcal P_{f_N}^{-1}-\mathcal P_f^{-1}
 =\mathcal P_{f_N}^{-1}(\mathcal P_f-\mathcal P_{f_N})\mathcal P_f^{-1}
\]
and the two inverse-derivative formulas show operator-norm convergence of
the maps and their first two shape derivatives.  This supplies the exact
smooth-approximation and limit step and identifies the finite-Sobolev maps
with the smooth Green-kernel operators described by
\eqref{eq:moving-kernel-directional-form}--%
\eqref{eq:moving-direction-difference}.  No differentiability of
$f\mapsto\mathcal V_f$ into the global space $H^1_0(\mathbb{D})$ is claimed:
its first fixed-point shape variation is a double-layer potential and can
have a jump.  Differentiation of the residual uses only the pulled-back
boundary maps $\mathcal S_f$ and $\mathcal D_f'$.
\end{proof}

\begin{lemma}[Compact remainders and commutators]
\label{lem:compact-remainders}
On the fixed smooth reference curve, finite-mode operators and operators
with smooth kernels are compact between the physical source and target
spaces.  Moreover,
\begin{equation}
\begin{aligned}
  \mathcal S-\mathcal S_0&\colon H^r\to H^{r+2},\\
  \mathcal T-\mathcal T_0&\colon H^{r+1}\to H^{r+1},\\
  [\mathcal S_0,M_b]&\colon H^r\to H^{r+2},\\
  [\mathcal T_0,M_b]&\colon H^{r+1}\to H^{r+1}
\end{aligned}
  \label{eq:compact-operator-gains}
\end{equation}
for every smooth coefficient $b$, where $M_b$ denotes multiplication by
$b$.  Each map becomes compact when followed by the one-derivative-lower
embedding appropriate to $\mathcal Y$.  The exact operators
$\mathcal D,\mathcal D'$ and the zero-mode discrepancies are lower-order or
finite-rank contributions.
\end{lemma}

\begin{proof}
The kernel decomposition in the proof of
Fact~\ref{fact:fixed-boundary-mappings} shows that exact-minus-model kernels
are smooth; the Maue formula gives the stated gain for
$\mathcal T-\mathcal T_0$.  For a Fourier multiplier $a(k)$ of order
$-1$ or $1$, the matrix of $[a(D),M_b]$ contains
$(a(k)-a(\ell))\widehat b(k-\ell)$.  The mean-value estimate for the
multiplier and rapid decay of $\widehat b$ lower the order by one, giving
the last two maps in \eqref{eq:compact-operator-gains}.  Rellich's theorem
then gives compactness.  Finite modes are finite rank, and a smooth kernel
maps bounded sets into every higher Sobolev space.
\end{proof}
\section{Problem setting}
\label{sec:problem-setting}

\subsection{Deformations}
\label{subsec:deformations}

The disk, the fixed Green kernel, the prescribed metric path, the reference
sheet, and the physical spaces are those of
Definitions~\ref{def:prescribed-metric-deformation}--
\ref{def:physical-sobolev-spaces}.  In particular,
$t\in\re/|c_0|\ze$ is Euclidean arclength,
$n$ points from $\Omega^-$ to $\Omega^+$, and
$\tau=\mathcal J n=c_0'$; the Frenet convention is fixed in
\eqref{eq:reference-frenet}.

\begin{definition}[BVSMD perturbation data]
\label{def:bvsmd-perturbation-data}
Lemma~\ref{lem:moving-boundary-operators} supplies an open neighborhood
of $0$ in $H^{m+1}$.  For every $f$ in that neighborhood, define
\begin{equation}
  c_f(t)=c_0(t)+f(t)n(t).
  \label{eq:normal-graph}
\end{equation}
The curve is embedded and separating in $\mathbb D^\circ$, has the same
orientation and side convention as $c_0$, and stays a uniform positive
distance from $\partial\mathbb D$.  Its Euclidean Jacobian and unit frame
are
\begin{equation}
\begin{aligned}
  c_f'&=(1+\kappa f)\tau+f'n,\\
  |c_f'|&=\bigl((1+\kappa f)^2+(f')^2\bigr)^{1/2},\\
  \tau_f&=|c_f'|^{-1}\bigl((1+\kappa f)\tau+f'n\bigr),\\
  n_f&=|c_f'|^{-1}\bigl((1+\kappa f)n-f'\tau\bigr).
\end{aligned}
  \label{eq:normal-graph-frame}
\end{equation}
The metric line element is
\begin{equation}
  d\ell_{g_s}=e^{\sigma_s\circ c_f}|c_f'|\,d t.
  \label{eq:moving-conformal-frame}
\end{equation}
For such $f$ and $h\in H^m$, define the scalar sheet strength on the moving
support by
\begin{equation}
  \gamma_{f,h}(c_f(t))=\gamma_0(t)+h(t).
  \label{eq:moving-sheet-strength}
\end{equation}
The physical unknown is
\begin{equation}
  u=(f,h,\mu,\beta)\in\mathcal X.
  \label{eq:physical-unknown}
\end{equation}
Choose an open neighborhood $O\subset\mathcal X$ of
$u_0=(0,0,\mu_0,\beta_0)$ whose $f$-projection is contained in that neighborhood and on
which the lower bounds from Lemma~\ref{lem:moving-boundary-operators} and
\eqref{eq:normal-graph-frame} are uniform.  Thus $O$ encodes embeddedness,
separation, interiority,
nonvanishing tangent, the positive conformal factor, and the existence of
all pulled-back traces and moving operators.  The constants $\mu$ and
$\beta$ are unknown because the stationary conditions determine their
constancy rather than prescribed values.
\end{definition}

The tubular-neighborhood map $(t,r)\mapsto c_0(t)+rn(t)$ is a
diffeomorphism for $|r|$ smaller than a fixed radius.  The embedding
$H^{m+1}\hookrightarrow C^2$ is continuous because $m\geq3$.  The
neighborhood is chosen so that $\|f\|_{C^1}$ is small, $|f|$ stays
inside the tubular radius, $c_f$ stays away from the outer boundary, and
$1+\kappa f>0$.  The normal projection is then one-to-one, $|c_f'|$ is
bounded below, and $c_f'\cdot\tau=1+\kappa f>0$ preserves the orientation.
The Jordan curve theorem gives the same inner and outer side designation.
Differentiating \eqref{eq:normal-graph} and using
\eqref{eq:reference-frenet} gives \eqref{eq:normal-graph-frame}; conformal
rescaling gives \eqref{eq:moving-conformal-frame}.

The normal graph fixes the local tangential reparametrization freedom.  The
strength $\gamma_{f,h}$ is a scalar per unit metric arclength; it is not the
parameter density used in the boundary integral.

\subsection{Residual map}
\label{subsec:residual-map}

\begin{definition}[Sheet density and background vorticity]
\label{def:sheet-density-background}
For $(s,u)\in(-s_0,s_0)\times O$, set
\begin{equation}
\begin{aligned}
  w_{s,f}(t)&=e^{\sigma_s(c_f(t))},\\
  \Theta(s,f,h)(t)
  &=(\gamma_0(t)+h(t))w_{s,f}(t)|c_f'(t)|,\\
  |\mathbb D|_{g_s}
  &=\int_{\mathbb D}e^{2\sigma_s}\,d A_{\mathrm e},\\
  \bar\omega(s,f,h)
  &=\frac{\displaystyle\int_0^{|c_0|}\Theta(s,f,h)(t)\,d t}
          {|\mathbb D|_{g_s}}.
\end{aligned}
  \label{eq:moving-density-background}
\end{equation}
Thus
$\Theta\,d t=\gamma_{f,h}\,d\ell_{g_s}$ is the pulled-back sheet measure.
The regular bulk vorticity is
$-\bar\omega(s,f,h)$.  Every occurrence of
$\bar\omega$ retains its dependence on $s,f,h$.
\end{definition}

\begin{definition}[Moving boundary operators]
\label{def:moving-boundary-operators}
We use the potential $\mathcal V_f$ and the pulled-back boundary operators
$\mathcal S_f$ and $\mathcal D_f'$ defined in
\eqref{eq:generic-moving-potential}--\eqref{eq:generic-moving-boundary-operators}.
The derivative in $\mathcal D_f'$ uses the target Euclidean normal.  No
source Jacobian is inserted because $\rho(t)\,d t$ is already the parameter
measure; equivalently, its density relative to Euclidean arclength on $c_f$
is $\rho/|c_f'|$.
\end{definition}

\begin{proposition}[Boundary-integral representation]
\label{prop:boundary-integral-representation}
The canonical weak stream function $\psi_{c_f,\gamma_{f,h}}$ of
$(c_f,\gamma_{f,h})$ for $g_s$ is
\begin{equation}
  \psi_{c_f,\gamma_{f,h}}
  =\mathcal V_f[\Theta(s,f,h)]
   -\bar\omega(s,f,h)\Psi_s.
  \label{eq:moving-stream-representation}
\end{equation}
It has zero outer Dirichlet value and satisfies
\begin{equation}
  -\triangle_{\mathrm e}\psi_{c_f,\gamma_{f,h}}
  =(\gamma_0+h)w_{s,f}\,\delta_{c_f}
   -\bar\omega(s,f,h)e^{2\sigma_s}
  \label{eq:moving-stream-distribution}
\end{equation}
when $\delta_{c_f}$ is measured with Euclidean arclength.  Its pulled-back trace and averaged Euclidean normal derivative are
\begin{equation}
\begin{aligned}
  \psi_{c_f,\gamma_{f,h}}\circ c_f
  &=\mathcal S_f\Theta
    -\bar\omega\,\Psi_s\circ c_f,\\
  \left\langle\partial_{n_f}\psi_{c_f,\gamma_{f,h}}\right\rangle\circ c_f
  &=\mathcal D_f'\Theta
    -\bar\omega\,\partial_{n_f}\Psi_s\circ c_f.
\end{aligned}
  \label{eq:moving-stream-boundary-values}
\end{equation}
Its canonical velocity for $g_s$ is
$v_{c_f,\gamma_{f,h}}
\coloneqq-\jgrad_{g_s}\psi_{c_f,\gamma_{f,h}}$.
\end{proposition}

\begin{proof}
For a test function $\phi$,
\[
  \left\langle-\triangle_{\mathrm e}\mathcal V_f[\Theta],\phi\right\rangle
  =\int_0^{|c_0|}\Theta(t)\phi(c_f(t))\,d t
  =\int_{c_f}(\gamma_0+h)w_{s,f}\phi\,d\ell_{\mathrm e}.
\]
Fact~\ref{fact:torsion-green-representation} supplies the volume term in
\eqref{eq:moving-stream-distribution}.  Both terms in
\eqref{eq:moving-stream-representation} vanish on
$\partial\mathbb D$ by the Dirichlet property of $G$.  Thus the displayed
function solves exactly the variational problem in
Definition~\ref{def:vortex-sheet-data}, with
$d\ell_{g_s}=w_{s,f}d\ell_{\mathrm e}$.  Uniqueness in
Proposition~\ref{prop:weak-stream-function} proves the representation.
Taking the common trace and the average of the two target-normal traces gives
\eqref{eq:moving-stream-boundary-values}.  The jump itself is
$[\partial_{n_f}\psi_{c_f,\gamma_{f,h}}]=-(\gamma_0+h)w_{s,f}$ by
\eqref{eq:moving-density-jump}, since
$\Theta/|c_f'|=(\gamma_0+h)w_{s,f}$.
\end{proof}

The $g_s$-unit tangent along $c_f$ is
$e^{-\sigma_s\circ c_f}\tau_f=w_{s,f}^{-1}\tau_f$.  Conformal conversion
therefore gives
\begin{equation}
  g_s\left(v_{c_f,\gamma_{f,h}}^{\pm},
            e^{-\sigma_s\circ c_f}\tau_f\right)
  =-w_{s,f}^{-1}\partial_{n_f}
       \psi_{c_f,\gamma_{f,h}}^{\pm}.
  \label{eq:metric-tangential-normal-conversion}
\end{equation}
This identity accounts for the sign and the single conformal factor in the
second residual.

\begin{definition}[Steady vortex-sheet residual]
\label{def:steady-sheet-residual}
For $u=(f,h,\mu,\beta)\in O$, define
\begin{align}
  E_I(s,u)
  &=\mathcal S_f\Theta(s,f,h)
    -\bar\omega(s,f,h)\Psi_s\circ c_f-\mu,
  \label{eq:streamline-residual}\\
  E_J(s,u)
  &=-(\gamma_0+h)w_{s,f}^{-1}
    \bigl(\mathcal D_f'\Theta(s,f,h)
      -\bar\omega(s,f,h)\partial_{n_f}\Psi_s\circ c_f\bigr)-\beta,
  \label{eq:bernoulli-residual}\\
  F(s,u)&=(E_I(s,u),E_J(s,u)).
  \label{eq:physical-residual}
\end{align}
The residual has type
\begin{equation}
  F\colon(-s_0,s_0)\times O\longrightarrow\mathcal Y.
  \label{eq:physical-residual-type}
\end{equation}
its well-definedness and differentiability are proved in
Proposition~\ref{prop:nonlinear-residual-regularity}.  Only the moving
nonlinear operators of Definition~\ref{def:moving-boundary-operators} occur
in \eqref{eq:streamline-residual}--\eqref{eq:bernoulli-residual}; the fixed
exact and model principal operators have no role in the nonlinear
definition.
\end{definition}

\begin{definition}[Normalized residual]
\label{def:normalized-residual}
Define the mean-displacement and circulation normalization by
\begin{equation}
  \mathcal N(s,u)=\left(
    \frac1{|c_0|}\int_0^{|c_0|}f(t)\,d t,
    \int_0^{|c_0|}\Theta(s,f,h)(t)\,d t
  \right)\in\re^2,
  \label{eq:physical-normalization}
\end{equation}
and define
\begin{equation}
  \mathcal F(s,u)=
  \bigl(F(s,u),\mathcal N(s,u)-\mathcal N(0,u_0)\bigr)
  \in\mathcal Y\times\re^2.
  \label{eq:normalized-physical-residual}
\end{equation}
The second component of $\mathcal N$ depends on $s$ through $w_{s,f}$, a
dependence that enters the parameter and mixed derivatives below.
\end{definition}

\begin{proposition}[Residual characterization of steady sheets]
\label{prop:residual-characterization}
For $(s,u)\in(-s_0,s_0)\times O$, the equation $F(s,u)=0$ is equivalent to
$(c_f,\gamma_{f,h})$ being a weak steady vortex sheet for $g_s$ with
streamline constant $\mu$ and Bernoulli constant $\beta$.  The equation
$\mathcal F(s,u)=0$ is equivalent to those conditions together with
\begin{equation}
  \mathcal N(s,u)=\mathcal N(0,u_0),
  \label{eq:normalization-values-fixed}
\end{equation}
that is, fixed mean normal displacement and fixed total circulation.
\end{proposition}

\begin{proof}
By Proposition~\ref{prop:boundary-integral-representation}, $E_I=0$ is
exactly constancy of the common stream-function trace.  Proposition
\ref{prop:streamline-characterization} makes it equivalent to the
kinematic condition.  Under that condition,
\eqref{eq:metric-tangential-normal-conversion} identifies $E_J=0$ with
\[
 \gamma_{f,h}\left\langle
 g_s\left(v_{c_f,\gamma_{f,h}},
          e^{-\sigma_s\circ c_f}\tau_f\right)
 \right\rangle=\beta.
\]
Proposition~\ref{prop:bernoulli-characterization} makes this Bernoulli condition
equivalent to the weak stationary Euler identity.  This proves the first
claim.  The second follows from the definition
\eqref{eq:normalized-physical-residual}; the second normalization entry is
$\int_{c_f}\gamma_{f,h}\,d\ell_{g_s}$ by
\eqref{eq:moving-density-background}.
\end{proof}

\subsection{Regularity}
\label{subsec:nonlinear-regularity}

The asymmetric spaces in \eqref{eq:physical-spaces} reflect the
one-derivative gain of the single-layer trace and the order-zero target used
for the averaged normal derivative.

\begin{lemma}[Regularity of sheet data]
\label{lem:sheet-data-regularity}
On a sufficiently small choice of $O$, these maps have the indicated
values and are jointly $C^2$ in all displayed variables:
\begin{equation}
\begin{aligned}
  c_f&\in H^{m+1},
  &c_f',|c_f'|,\tau_f,n_f&\in H^m,\\
  w_{s,f}&\in H^{m+1},
  &\Theta(s,f,h)&\in H^m,\\
  \bar\omega(s,f,h)&\in\re,
  &\Psi_s\circ c_f&\in H^{m+1},\\
  \partial_{n_f}\Psi_s\circ c_f&\in H^m.
\end{aligned}
  \label{eq:sheet-data-regularities}
\end{equation}
The reciprocal $|c_f'|^{-1}$ and $w_{s,f}^{-1}$ have the same regularity as
$|c_f'|$ and $w_{s,f}$, respectively.  All bounds are uniform after $O$ and
the parameter interval are replaced by smaller neighborhoods.
\end{lemma}

\begin{proof}
Equations \eqref{eq:normal-graph-frame} and
Fact~\ref{fact:periodic-sobolev-calculus} give the first row and the smooth
dependence on $f$; the lower bounds in
equations \eqref{eq:normal-graph-frame} give the reciprocal maps.
The Fr\'echet $C^3$ metric path and smooth composition give
$w_{s,f}\in H^{m+1}$.  Since $H^m$ is an algebra,
$(\gamma_0+h)w_{s,f}|c_f'|\in H^m$, proving the density assertion.  Integration
and division by the strictly positive metric area preserve $C^2$ dependence
and give the scalar assertion.

Fact~\ref{fact:torsion-green-representation} gives $C^3$ dependence of
$\Psi_s$ in every finite spatial Sobolev space.  Smooth trace composition
therefore puts $\Psi_s\circ c_f$ in $H^{m+1}$.  Finally,
\[
  \partial_{n_f}\Psi_s\circ c_f
  =n_f\cdot(\nabla_{\mathrm e}\Psi_s)\circ c_f,
\]
is a product of an $H^m$ field and an $H^{m+1}$ field, hence lies in $H^m$.
The same calculations, differentiated twice, give the asserted joint
dependence and uniform bounds.  The third parameter derivative in
\eqref{eq:metric-path-topology} supplies more than the two derivatives used
here and remains available for the second-order reduction.
\end{proof}

\begin{proposition}[Regularity of the nonlinear residual]
\label{prop:nonlinear-residual-regularity}
For $m\geq3$, after shrinking $O$ and the parameter interval if necessary,
\begin{equation}
\begin{aligned}
  E_I(s,u)&\in H^{m+1},
  &E_J(s,u)&\in H^m,\\
  F&\in C^2((-s_0,s_0)\times O;\mathcal Y),\\
  \mathcal N&\in C^2((-s_0,s_0)\times O;\re^2),\\
  \mathcal F&\in C^2((-s_0,s_0)\times O;
    \mathcal Y\times\re^2).
\end{aligned}
  \label{eq:nonlinear-map-regularity}
\end{equation}
Every first, second, and mixed derivative is the derivative of the full map
in \eqref{eq:physical-residual} or
\eqref{eq:normalized-physical-residual}; in particular, derivatives of the
second component of $\mathcal N$ include all derivatives of
$\Theta(s,f,h)$.  The corresponding stream functions have one-sided
$H^{m+3/2}$ regularity, so Definition~\ref{def:weak-steady-sheet} is
realized in this physical class.
\end{proposition}

\begin{proof}
Lemma~\ref{lem:moving-boundary-operators} and
Lemma~\ref{lem:sheet-data-regularity} give
\[
  \mathcal S_f\Theta\in H^{m+1},
  \qquad
  \mathcal D_f'\Theta\in H^m,
\]
with joint $C^2$ dependence after composing the operator-valued and
density-valued maps.  The same sheet-data lemma gives
$\Psi_s\circ c_f\in H^{m+1}$ and
$\partial_{n_f}\Psi_s\circ c_f\in H^m$.  Multiplication by
$\gamma_0+h$ and $w_{s,f}^{-1}$ preserves $H^m$.  Equations
\eqref{eq:streamline-residual} and \eqref{eq:bernoulli-residual} therefore
give the first row of \eqref{eq:nonlinear-map-regularity} and the joint
$C^2$ assertion for $F$.

Both normalization entries are bounded integrals of the $C^2$ maps $f$ and
$\Theta(s,f,h)$, so $\mathcal N$ is jointly $C^2$, with its genuine
$s$-dependence intact.  The assertion for $\mathcal F$ follows by taking the
product map.  Finally, Lemma~\ref{lem:moving-boundary-operators} gives
$H^{m+3/2}$ regularity on each side from the $H^m$ density.  Thus the
one-sided velocity and nonlinear Euler integrals have more regularity than
required in Definition~\ref{def:weak-steady-sheet}.
\end{proof}

\begin{lemma}[Base-state identities]
\label{lem:base-state-identities}
The reference point satisfies
\begin{equation}
  F(0,u_0)=0,
  \qquad
  \mathcal F(0,u_0)=0.
  \label{eq:base-state-zero}
\end{equation}
For a direction
$\dot u=(\dot f,\dot h,\dot\mu,\dot\beta)\in\mathcal X$, the complete
unknown derivative of the normalization at fixed $s=0$ is
\begin{equation}
  D_u\mathcal N(0,u_0)[\dot u]
  =\left(
    \frac1{|c_0|}\int_0^{|c_0|}\dot f\,d t,
    \int_0^{|c_0|}e^{\sigma_0\circ c_0}
    \left[
      \dot h+\gamma_0
      (\kappa+\partial_{n}\sigma_0\circ c_0)\dot f
    \right]d t
  \right).
  \label{eq:physical-normalization-derivative}
\end{equation}
If $\dot\sigma=\left.\partial_s\sigma_s\right|_{s=0}$ while $u=u_0$ is fixed, then
\begin{equation}
  D_s\mathcal N(0,u_0)
  =\left(
    0,
    \int_0^{|c_0|}
    \gamma_0e^{\sigma_0\circ c_0}
    (\dot\sigma\circ c_0)\,d t
  \right).
  \label{eq:parameter-normalization-derivative}
\end{equation}
Thus the normalization derivative entering
\eqref{eq:normalization-compatibility} is the restriction of
\eqref{eq:physical-normalization-derivative}, and the
parameter derivative of $\mathcal F$ includes
\eqref{eq:parameter-normalization-derivative}.
\end{lemma}

\begin{proof}
The reference sheet is steady with constants $\mu_0,\beta_0$ by
Definition~\ref{def:reference-state}.  Proposition
\ref{prop:residual-characterization} gives $F(0,u_0)=0$, and the definition
of $\mathcal F$ gives its base zero.

At $f=0$, equation \eqref{eq:normal-graph-frame} gives
$\left.D_f|c_f'|\right|_{f=0}[\dot f]=\kappa\dot f$, while the chain rule gives
\[
  D_fw_{0,f}|_0[\dot f]
  =e^{\sigma_0\circ c_0}(\partial_{n}\sigma_0\circ c_0)\dot f.
\]
Differentiating all three factors in
$\Theta=(\gamma_0+h)w_{s,f}|c_f'|$ proves
\eqref{eq:physical-normalization-derivative}; the scalar directions
$\dot\mu,\dot\beta$ do not enter $\mathcal N$.  At fixed $u_0$, only the
factor $w_{s,0}$ depends on $s$, and its derivative is
$e^{\sigma_0\circ c_0}(\dot\sigma\circ c_0)$.  This proves
\eqref{eq:parameter-normalization-derivative}.
\end{proof}
\section{Main results}
\label{sec:main-results}

We state the results on the physical spaces of
Definition~\ref{def:physical-sobolev-spaces}.  Thus, for a fixed integer
$m\geq3$,
\[
 \mathcal X
 =H^{m+1}(\re/|c_0|\ze)\times
   H^m(\re/|c_0|\ze)
   \times\re^2,
 \qquad
 \mathcal Y
 =H^{m+1}(\re/|c_0|\ze)\times
   H^m(\re/|c_0|\ze).
\]
The unknown is $u=(f,h,\mu,\beta)$ and
$u_0=(0,0,\mu_0,\beta_0)$.  The residuals from
Definitions~\ref{def:steady-sheet-residual}
and~\ref{def:normalized-residual} have types
\[
 F\colon(-s_0,s_0)\times O\longrightarrow\mathcal Y,
 \qquad
 \mathcal F\colon(-s_0,s_0)\times O
       \longrightarrow\mathcal Y\times\re^2.
\]
In particular,
\[
 \mathcal F(s,u)
 =\bigl(F(s,u),\mathcal N(s,u)-\mathcal N(0,u_0)\bigr),
\]
so the parameter dependence of the physical normalization is part of the
map.

The hypotheses for this section are as follows.  The prescribed
map $s\mapsto\sigma_s$ is $C^3$ with values in
$C^\infty(\overline{\mathbb D})$ and
$g_s=e^{2\sigma_s}g_{\mathrm e}$.  The curve $c_0$ is a smooth embedded
separating curve parametrized by Euclidean arclength, the strength
$\gamma_0$ is smooth, and $(c_0,\gamma_0)$ is the
reference weak steady sheet.  The normal-graph neighborhood $O\subset
\mathcal X$ is chosen so that all curves remain embedded and all objects in
Section~\ref{sec:problem-setting} are defined.  Finally, the one-sided scalar
tangential velocities obey the pointwise stagnation-free condition
\begin{equation}
 |v_{\tau,0}^{+}(t)|+|v_{\tau,0}^{-}(t)|>0
 \quad\text{for every }t\in\re/|c_0|\ze.
 \label{eq:standing-stagnation-free}
\end{equation}
By Proposition~\ref{prop:nonlinear-residual-regularity}, $F$, $\mathcal N$,
and $\mathcal F$ are jointly $C^2$ on a possibly smaller product neighborhood of
$(0,u_0)$.

\begin{theorem}[Fredholm property of the unnormalized linearization]
\label{thm:unnormalized-fredholm}
Let $m\geq3$.  Let $s\mapsto\sigma_s$ be $C^3$ with values in
$C^\infty(\overline{\mathbb D})$, set $g_s=e^{2\sigma_s}g_{\mathrm e}$,
and let $(c_0,\gamma_0)$ be a smooth embedded separating reference weak
steady sheet, with $c_0$ parametrized by Euclidean arclength and
$\gamma_0\in C^\infty(c_0)$.  Let $O$ be the admissible
normal-graph neighborhood from
Definition~\ref{def:bvsmd-perturbation-data}, and assume
\eqref{eq:standing-stagnation-free}.  The full derivative is
\[
 L\coloneqq D_uF(0,u_0)\colon\mathcal X\longrightarrow\mathcal Y.
\]
It is Fredholm and
\[
 \indop L=2.
\]
Its kernel is finite dimensional and consists of smooth data.  Under the
periodic $L^2(d t)$ dual pairing, its cokernel is represented by a
finite-dimensional space of smooth formal-adjoint null vectors.  These
spaces and their dimensions are independent of the admissible Sobolev
exponent used to realize the same physical linearization.
\end{theorem}

The exact derivative used here is Proposition~\ref{prop:complete-linearization};
the Fredholm and index arguments are
Proposition~\ref{prop:function-block-fredholm} and the finite-dimensional
bookkeeping of Fact~\ref{fact:finite-dimensional-index}.  Sobolev-order
independence is proved in Lemma~\ref{lem:kernel-cokernel-regularity}.

\begin{theorem}[Fredholm property of the normalized linearization]
\label{thm:normalized-fredholm}
Assume all hypotheses of Theorem~\ref{thm:unnormalized-fredholm}.  Then
\[
 D_u\mathcal F(0,u_0)\colon\mathcal X
       \longrightarrow\mathcal Y\times\re^2
\]
is Fredholm of index zero.

Suppose, in addition, that the physical normalization is compatible with
$L$ in the sense that
\begin{equation}
 D_u\mathcal N(0,u_0)\big|_{\ker L}\colon
 \ker L\longrightarrow\re^2
 \quad\text{is onto}.
 \label{eq:normalization-compatibility}
\end{equation}
Then
\begin{align}
 \ker D_u\mathcal F(0,u_0)
 &=\ker L\cap\ker D_u\mathcal N(0,u_0),
 \label{eq:normalized-kernel}\\
 \ran D_u\mathcal F(0,u_0)
 &=\ran L\times\re^2.
 \label{eq:normalized-range}
\end{align}
Projection to the $\mathcal Y$ factor identifies
$\coker D_u\mathcal F(0,u_0)$ canonically with
$\coker L$.  In
particular, after \eqref{eq:normalization-compatibility} has been verified,
the degeneracy index
\begin{equation}
 d\coloneqq\dim\coker L
   =\dim\coker D_u\mathcal F(0,u_0)
   =\dim\ker D_u\mathcal F(0,u_0)
 \label{eq:degeneracy-index}
\end{equation}
is well defined.
\end{theorem}

The algebra behind \eqref{eq:normalized-kernel}--\eqref{eq:normalized-range}
is Proposition~\ref{prop:fredholm-normalization}; its physical realization is
\eqref{eq:physical-normalization-derivative}--%
\eqref{eq:normalized-finite-dimensional-counts}.

\begin{theorem}[Lyapunov--Schmidt reduction for normalized steady sheets]
\label{thm:physical-lyapunov-schmidt}
Assume the hypotheses of Theorem~\ref{thm:unnormalized-fredholm} and the
compatibility condition \eqref{eq:normalization-compatibility}.  Put
$\mathcal K=\ker D_u\mathcal F(0,u_0)$ and choose a closed complement $\mathcal X_1$
so that
$\mathcal X=\mathcal K\oplus\mathcal X_1$.  Choose a finite-dimensional complement
$\mathcal C$ of $\ran D_u\mathcal F(0,u_0)$ in
$\mathcal Y\times\re^2$, and let $P$ and $Q$ be the bounded
projections onto $\ran D_u\mathcal F(0,u_0)$ and $\mathcal C$,
respectively.  Thus $P+Q$ is the identity and
$\ker Q=\ran D_u\mathcal F(0,u_0)$.

There are neighborhoods of $0$ in $\re\times\mathcal K$ and a unique
$C^2$ map
\[
 \xi\colon(s,\zeta)\longmapsto\xi(s,\zeta)\in\mathcal X_1,
 \qquad
 \xi(0,0)=0,
 \qquad
 D_\zeta\xi(0,0)=0,
\]
which solves the complementary equation
\begin{equation}
 P\mathcal F\bigl(s,u_0+\zeta+\xi(s,\zeta)\bigr)=0.
 \label{eq:main-range-equation}
\end{equation}
The reduced map
\begin{equation}
 B(s,\zeta)
 \coloneqq Q\mathcal F\bigl(s,u_0+\zeta+\xi(s,\zeta)\bigr)
 \in\mathcal C
 \label{eq:main-reduced-map}
\end{equation}
has domain in $\re\times\mathcal K$ and target
$\mathcal C\simeq\coker D_u\mathcal F(0,u_0)$, with
$\dim\mathcal K=\dim\mathcal C=d$.  In a product neighborhood of $(0,u_0)$, define
\[
 (s,\zeta)\longmapsto
 \bigl(s,u_0+\zeta+\xi(s,\zeta)\bigr).
\]
This assignment is a bijection from the zeros of $B$ to the zeros of
$\mathcal F$.
Consequently it is also a bijection to the normalized weak steady sheets in
the chosen normal chart.
\end{theorem}

The physical construction, including the $s$-dependent normalization, is
carried out in Subsection~\ref{subsec:physical-ls-proof}; see
Definition~\ref{def:physical-reduced-equation}.

\begin{corollary}[Nondegenerate continuation from the reference state]
\label{cor:nondegenerate-continuation}
Assume all hypotheses of Theorem~\ref{thm:unnormalized-fredholm}, suppose
\eqref{eq:normalization-compatibility} holds, and let $d=0$.  Then there are
$0<\varepsilon<s_0$, a sufficiently small neighborhood of $u_0$ contained
in $O$, and a unique $C^2$ map $u\colon(-\varepsilon,\varepsilon)\longrightarrow O$
whose values lie in that neighborhood and such that
\[
 u(0)=u_0,
 \qquad
 \mathcal F(s,u(s))=0.
\]
The uniqueness is among solutions with $|s|<\varepsilon$ and $u$ in that
neighborhood, and hence is relative to the chosen
normal chart and normalization.  The corresponding sheets have the fixed
mean normal displacement and fixed total circulation prescribed by
$\mathcal N(0,u_0)$.  Their initial derivative is
\begin{equation}
 u'(0)=-\bigl[D_u\mathcal F(0,u_0)\bigr]^{-1}
          D_s\mathcal F(0,u_0),
 \qquad
 D_s\mathcal F(0,u_0)
 =\bigl(D_sF(0,u_0),D_s\mathcal N(0,u_0)\bigr).
 \label{eq:implicit-branch-derivative}
\end{equation}
\end{corollary}

The proof is the $\mathcal K=\{0\}$ case of
Theorem~\ref{thm:physical-lyapunov-schmidt} and is given at the end of
Subsection~\ref{subsec:physical-ls-proof}.

\begin{theorem}[First- and second-order continuation obstructions]
\label{thm:reduced-obstructions}
Assume the hypotheses of Theorem~\ref{thm:physical-lyapunov-schmidt}, and
let $B$ be \eqref{eq:main-reduced-map}.  For $\eta\in\mathcal K$, the quadratic
reduced expression is
\begin{equation}
 \frac12B_{\zeta\zeta}(0,0)[\eta,\eta]
   +B_{s\zeta}(0,0)[\eta]
   +\frac12B_{ss}(0,0)
 \in\mathcal C.
 \label{eq:quadratic-reduced-expression}
\end{equation}
If $s\mapsto u(s)$ is a $C^2$ curve of normalized solutions with
$u(0)=u_0$, and $\eta$ is the derivative at zero of its $\mathcal K$-coordinate,
then necessarily
\begin{equation}
 B_s(0,0)=0,
 \qquad
 \frac12B_{\zeta\zeta}(0,0)[\eta,\eta]
   +B_{s\zeta}(0,0)[\eta]
   +\frac12B_{ss}(0,0)=0.
 \label{eq:first-second-obstructions}
\end{equation}
Thus failure of either condition excludes such a $C^2$ curve.

These are necessary conditions only.  No converse, branch existence, branch
number, branch scale, or branch regularity is asserted.
\end{theorem}

The parameter-dependent normalization is included in every derivative of
$\mathcal F$.  The identities are stated in
Lemma~\ref{lem:physical-reduced-derivatives}; its chain-rule proof and the
proof of the stated necessary conditions are in
Subsection~\ref{subsec:physical-ls-proof}.
\section{Proofs of the Fredholm and reduction theorems}
\label{sec:proof}

We proceed in five steps.  First, fixed-metric shape differentiation yields
the Eulerian transmission problem and the derivative of the two residuals.
Second, the boundary traces split this derivative into an elliptic
mixed-order matrix and compact terms.  Third, a two-sided parametrix and a
fixed-space Fredholm homotopy determine the index.  Fourth, the two physical
normalizations remove the excess index.  Finally, the normalized equation is
reduced to its finite-dimensional kernel and cokernel, with the metric
dependence retained in the reduced derivatives.

\subsection{Fixed-metric first variations}
\label{subsec:fixed-metric-first-variations}

In this subsection the metric parameter is fixed at \(s=0\).  We use the
reference shorthand
\[
 \sigma=\sigma_0,
 \qquad
 w=e^{\sigma\circ c_0},
 \qquad
 \Theta_0=\gamma_0w,
 \qquad
 \Psi=\Psi_0 .
\]
All normals, tangents, curvatures, and normal derivatives in the calculation
are Euclidean.  Thus \(n\) points from \(\Omega^-\) to \(\Omega^+\),
\(\tau=\mathcal J n\), and \([q]=q^+-q^-\).  The parameter \(t\) is
Euclidean arclength on
\(\re/|c_0|\ze\), and our Frenet convention is
\[
 c_0'=\tau,
 \qquad
 \partial_t n=\kappa\tau,
 \qquad
 \partial_t\tau=-\kappa n.
\]
Every ambient scalar and directional derivative in this subsection is
understood after pullback by $c_0$ unless a composition is displayed.  For a
smooth ambient scalar $q$, we also fix the Hessian-contraction convention
\begin{equation}
 \partial_n^2q=D^2q[n,n].
 \label{eq:base-hessian-convention}
\end{equation}

\begin{definition}[Eulerian variation]
\label{def:eulerian-variation}
For \((f,h)\in H^{m+1}\times H^m\), set
\[
 c_\varepsilon=c_0+\varepsilon fn,
 \qquad
 \gamma_\varepsilon(c_\varepsilon(t))
   =\gamma_0(t)+\varepsilon h(t),
\]
and let \(\psi_\varepsilon=\psi_{c_\varepsilon,\gamma_\varepsilon}\) for the fixed
metric $g_0$.
The one-sided Eulerian variation is the fixed-ambient-point derivative
\[
 \varphi^\pm(x)
 =\left.\frac{d}{d\varepsilon}\right|_{\varepsilon=0}
     \psi_\varepsilon^\pm(x),
 \qquad x\in\Omega^\pm.
\]
It is distinct from a moving-trace derivative: for either side,
\begin{equation}
\label{eq:moving-trace-chain-rule}
 \left.\frac{d}{d\varepsilon}\right|_{0}
 \psi_\varepsilon^\pm(c_\varepsilon(t))
 =\varphi^\pm(c_0(t))+f(t)\partial_{n}\psi_0^\pm(c_0(t)).
\end{equation}
Throughout the proof, with \(s\) fixed, abbreviate
\[
 (D\Theta)[f,h]
 \coloneqq D_{(f,h)}\Theta(0,0,0)[f,h],
 \qquad
 (D\bar\omega)[f,h]
 \coloneqq D_{(f,h)}\bar\omega(0,0,0)[f,h].
\]
\end{definition}

\begin{lemma}[Geometric first variations]
\label{lem:geometric-first-variations}
For the perturbation in Definition~\ref{def:eulerian-variation},
\begin{align}
 \left.\frac{d}{d\varepsilon}\right|_0c_\varepsilon'
   &=\kappa f\tau+f_tn,
 &
 \left.\frac{d}{d\varepsilon}\right|_0|c_\varepsilon'|
   &=\kappa f,                                             \label{eq:geometry-first-two}\\
 \left.\frac{d}{d\varepsilon}\right|_0\tau_\varepsilon
   &=f_tn,
 &
 \left.\frac{d}{d\varepsilon}\right|_0n_\varepsilon
   &=-f_t\tau,                                             \label{eq:geometry-frame}
\end{align}
where \(\tau_\varepsilon=c_\varepsilon'/|c_\varepsilon'|\) and
\(n_\varepsilon=-\mathcal J\tau_\varepsilon\) with the orientation fixed at
the start of Subsection~\ref{subsec:fixed-metric-first-variations}.  If
\(q_\varepsilon^\pm\) has Eulerian derivative
\(\dot q^\pm\), then
\begin{equation}
\label{eq:normal-derivative-variation}
 \left.\frac{d}{d\varepsilon}\right|_0
 \bigl(\partial_{n_\varepsilon}q_\varepsilon^\pm\bigr)
       (c_\varepsilon(t))
 =\partial_{n}\dot q^\pm
   +f\partial_n^2q_0^\pm-f_t\partial_{\tau} q_0^\pm .
\end{equation}
All quantities on the right are evaluated at \(c_0(t)\).  In particular,
the variations in \eqref{eq:geometry-first-two}--\eqref{eq:geometry-frame}
belong to \(H^m\), and \eqref{eq:normal-derivative-variation} has the trace
regularity furnished by the transmission theorem.
\end{lemma}

\begin{proof}
Differentiating \(c_\varepsilon=c_0+\varepsilon fn\) in \(t\) and using
\(n_t=\kappa\tau\) gives
\[
 c_\varepsilon'
 =(1+\varepsilon\kappa f)\tau+\varepsilon f_tn.
\]
Taking its norm and then normalizing proves
\eqref{eq:geometry-first-two}--\eqref{eq:geometry-frame}.  Finally,
\[
 \partial_{n_\varepsilon}q_\varepsilon^\pm(c_\varepsilon)
 =n_\varepsilon\mathbin{\cdot}
   \nabla q_\varepsilon^\pm(c_\varepsilon).
\]
The product and composition rules give
\(-f_t\partial_{\tau} q_0^\pm+\partial_{n}\dot q^\pm
 +f\partial_n^2q_0^\pm\), including both the normal-vector and
evaluation-point contributions.  Periodic Sobolev multiplication for
\(m\geq3\) gives the asserted types.
\end{proof}

\begin{lemma}[Density and background-vorticity variations]
\label{lem:density-background-variations}
The fixed-metric derivatives are
\begin{equation}
\label{eq:density-first-variation}
 (D\Theta)[f,h]
 =wh+\gamma_0w(\partial_{n}\sigma+\kappa)f
 \in H^m
\end{equation}
and
\begin{equation}
\label{eq:background-first-variation}
 (D\bar\omega)[f,h]
 =\frac{1}{|\mathbb D|_{g_0}}
   \int_0^{|c_0|}
   \bigl[wh+\gamma_0w(\partial_{n}\sigma+\kappa)f\bigr]\,d t
 \in\re .
\end{equation}
Thus \((f,h)\mapsto(D\Theta)[f,h]\) is bounded
\(H^{m+1}\times H^m\to H^m\), while
\((f,h)\mapsto(D\bar\omega)[f,h]\) is a bounded scalar functional.
\end{lemma}

\begin{proof}
At fixed \(s=0\),
\[
 \Theta(0,\varepsilon f,\varepsilon h)
 =(\gamma_0+\varepsilon h)
   e^{\sigma\circ c_\varepsilon}|c_\varepsilon'|.
\]
Lemma~\ref{lem:geometric-first-variations} and the chain rule yield
\[
 \left.\frac{d}{d\varepsilon}\right|_0
 e^{\sigma\circ c_\varepsilon}=wf\partial_{n}\sigma,
 \qquad
 \left.\frac{d}{d\varepsilon}\right|_0|c_\varepsilon'|=\kappa f,
\]
which proves \eqref{eq:density-first-variation}.  Since the metric is fixed,
the denominator in
\(\bar\omega=|\mathbb D|_{g_0}^{-1}\int\Theta\,d t\) is fixed;
differentiating
only the numerator gives \eqref{eq:background-first-variation}.  The claimed
mapping properties follow from the smoothness of the reference coefficients
and periodic Sobolev multiplication.
\end{proof}

\begin{lemma}[Jump data for the Eulerian variation]
\label{lem:eulerian-jump-data}
The Eulerian variation in Definition~\ref{def:eulerian-variation} satisfies
\begin{align}
 [\varphi]&=\Theta_0f=\gamma_0wf,                           \label{eq:variation-value-jump}\\
 [\partial_{n}\varphi]
 &=-wh-\gamma_0w(\partial_{n}\sigma)f
   -f[\partial_n^2\psi_0].                                \label{eq:variation-normal-jump}
\end{align}
The value jump belongs to \(H^{m+1}\), and the normal-derivative jump belongs
to \(H^m\).  No curvature Jacobian occurs in
\eqref{eq:variation-normal-jump}; that Jacobian belongs to the pulled-back
measure density \(\Theta\), not to the physical normal-derivative jump.
\end{lemma}

\begin{proof}
Continuity of each \(\psi_\varepsilon\) on its moving support, followed by
\eqref{eq:moving-trace-chain-rule}, gives
\[
 [\varphi]+f[\partial_{n}\psi_0]=0.
\]
The reference jump is
\([\partial_{n}\psi_0]=-\gamma_0w=-\Theta_0\), proving
\eqref{eq:variation-value-jump}.  On the other hand, the physical jump on
the perturbed curve is
\[
 [\partial_{n_\varepsilon}\psi_\varepsilon]\circ c_\varepsilon
 =-(\gamma_0+\varepsilon h)e^{\sigma\circ c_\varepsilon}.
\]
Its right-hand derivative is
\(-wh-\gamma_0wf\partial_{n}\sigma\).  Applying
\eqref{eq:normal-derivative-variation} to both sides gives on the left
\[
 [\partial_{n}\varphi]+f[\partial_n^2\psi_0]
 -f_t[\partial_{\tau}\psi_0].
\]
The reference streamline condition says
\(\psi_0^\pm\circ c_0=\mu_0\), hence
\(\partial_{\tau}\psi_0^\pm=0\).  Equating the derivatives proves
\eqref{eq:variation-normal-jump}.  The trace types follow from the stated
regularity of the reference solution and the Sobolev product theorem.
\end{proof}

\subsection{The linearized transmission problem and derivative}
\label{subsec:exact-transmission-linearization}

\begin{proposition}[Linearized transmission representation]
\label{prop:linearized-transmission}
For every \((f,h)\in H^{m+1}\times H^m\), the Eulerian variation is the
unique piecewise \(H^{m+3/2}\) solution of
\begin{equation}
\label{eq:physical-linearized-transmission}
\begin{aligned}
 -\triangle_{\mathrm e}\varphi^\pm
   &=-(D\bar\omega)[f,h]e^{2\sigma}
       &&\text{in }\Omega^\pm,\\
 \varphi^+&=0 &&\text{on }\partial\mathbb D,\\
 [\varphi]&=\gamma_0wf &&\text{on }c_0,\\
 [\partial_{n}\varphi]
   &=-wh-\bigl(\gamma_0w\partial_{n}\sigma
                    +[\partial_n^2\psi_0]\bigr)f
       &&\text{on }c_0.
\end{aligned}
\end{equation}
It is represented, using only the fixed exact layer potentials, by
\begin{equation}
\label{eq:exact-variation-representation}
\begin{split}
 \varphi={}&\mathcal W[\gamma_0wf]
 +\mathcal V\!\left[
      wh+\bigl(\gamma_0w\partial_{n}\sigma
                 +[\partial_n^2\psi_0]\bigr)f\right]\\
 &\quad -(D\bar\omega)[f,h]\Psi .
\end{split}
\end{equation}
Consequently its exact average traces are
\begin{equation}
\label{eq:exact-average-value-trace}
\begin{split}
 \langle\varphi\rangle={}&\mathcal D(\gamma_0wf)
 +\mathcal S\!\left[
      wh+\bigl(\gamma_0w\partial_{n}\sigma
                 +[\partial_n^2\psi_0]\bigr)f\right]\\
 &\quad -(D\bar\omega)[f,h]\,\Psi\circ c_0
\end{split}
\end{equation}
and
\begin{equation}
\label{eq:exact-average-normal-trace}
\begin{split}
 \langle\partial_{n}\varphi\rangle={}&\mathcal T(\gamma_0wf)
 +\mathcal D'\!\left[
      wh+\bigl(\gamma_0w\partial_{n}\sigma
                 +[\partial_n^2\psi_0]\bigr)f\right]\\
 &\quad -(D\bar\omega)[f,h]\,\partial_{n}\Psi\circ c_0.
\end{split}
\end{equation}
\end{proposition}

\begin{proof}
The fixed Green representation and the kernel differentiation established in
Lemma~\ref{lem:moving-boundary-operators} first prove that the fixed-point
derivative exists away from $c_0$.  Differentiating there gives
\begin{equation}
 \varphi
 =\mathcal W[\gamma_0wf]
  +\mathcal V\!\left[wh+\gamma_0w(\partial_{n}\sigma+\kappa)f\right]
  -(D\bar\omega)[f,h]\Psi.
 \label{eq:direct-green-variation}
\end{equation}
The streamline identity also yields the relation needed to put this formula
in transmission form.  Indeed, differentiating
$\psi_0^\pm\circ c_0=\mu_0$ twice in Euclidean arclength gives
\[
 D^2\psi_0^\pm[\tau,\tau]=\kappa\partial_{n}\psi_0^\pm.
\]
The volume vorticity is the same on the two sides, so
$[\triangle_{\mathrm e}\psi_0]=0$.  Since
$[\partial_{n}\psi_0]=-\gamma_0w$, it follows that
\begin{equation}
 [\partial_n^2\psi_0]
 =-[D^2\psi_0[\tau,\tau]]
 =-\kappa[\partial_{n}\psi_0]
 =\kappa\gamma_0w.
 \label{eq:reference-second-normal-jump}
\end{equation}
Thus \eqref{eq:direct-green-variation} is precisely
\eqref{eq:exact-variation-representation}.

Differentiating the piecewise equation for \(\psi_\varepsilon\) at fixed
metric gives the first line of \eqref{eq:physical-linearized-transmission};
the fixed outer boundary gives the second.  The remaining two lines are
Lemma~\ref{lem:eulerian-jump-data}.  Their data have exactly the trace types
in Theorem~\ref{thm:transmission-regularity}: the volume term is smooth on
each side, the value jump is in \(H^{m+1}\), and the normal jump is in
\(H^m\).  That theorem therefore gives a unique piecewise \(H^{m+3/2}\)
solution.  In particular, because an outer Dirichlet condition is present,
there is no scalar integral compatibility condition to check.

The jump relations also provide a direct verification: the double layer in
\eqref{eq:exact-variation-representation}
is harmonic off \(c_0\), has value jump \(\gamma_0wf\), and has continuous
normal derivative.  The single layer is harmonic off \(c_0\), has continuous
value, and its normal-derivative jump is the negative of its displayed
density.  Finally, \(-\triangle_{\mathrm e}\Psi=e^{2\sigma}\), so the last
term has the required volume source.  The Dirichlet Green kernel, its source
normal derivative, and \(\Psi\) all vanish when the target lies on
\(\partial\mathbb D\); hence every term has zero outer trace.  Thus the
right-hand side of \eqref{eq:exact-variation-representation} solves every
line of \eqref{eq:physical-linearized-transmission}, and uniqueness in
Theorem~\ref{thm:transmission-regularity} identifies it with \(\varphi\).
The averaged jump relations then give
\eqref{eq:exact-average-value-trace} and
\eqref{eq:exact-average-normal-trace}.
\end{proof}

The reference kinematic and velocity identities used in
Proposition~\ref{prop:complete-linearization} are
\begin{equation}
\label{eq:reference-average-normal-velocity}
 \langle\partial_{n}\psi_0\rangle
 =-w\langle v_0\rangle_{\tau} .
\end{equation}
Here \(\langle v_0\rangle_{\tau}\) is the averaged scalar tangential component
defined in \eqref{eq:reference-tangential-velocities}.

\begin{proposition}[Fixed-metric linearization]
\label{prop:complete-linearization}
At \((s,u)=(0,u_0)\), with every variable not displayed held fixed, the four
function-variable derivatives are
\begin{align}
 D_fE_I(0,u_0)[f]
 ={}&-w\langle v_0\rangle_{\tau} f+\mathcal D(\gamma_0wf)
 \notag\\
 &+\mathcal S\!\left[
   \bigl(\gamma_0w\partial_{n}\sigma
      +[\partial_n^2\psi_0]\bigr)f\right]
 -(D\bar\omega)[f,0]\,\Psi\circ c_0,
                                                               \label{eq:exact-df-f1}\\
 D_hE_I(0,u_0)[h]
 ={}&\mathcal S(wh)
 -(D\bar\omega)[0,h]\,\Psi\circ c_0,             \label{eq:exact-dh-f1}\\
 D_fE_J(0,u_0)[f]
 ={}&-\gamma_0w^{-1}\mathcal T(\gamma_0wf)
 -\gamma_0(\partial_{n}\sigma)\langle v_0\rangle_{\tau} f
 \notag\\
 &-\gamma_0w^{-1}\mathcal D'\!\left[
   \bigl(\gamma_0w\partial_{n}\sigma
      +[\partial_n^2\psi_0]\bigr)f\right]
 \notag\\
 &+\gamma_0w^{-1}(D\bar\omega)[f,0]
          \,\partial_{n}\Psi\circ c_0
 -\gamma_0w^{-1}f\langle\partial_n^2\psi_0\rangle,
                                                               \label{eq:exact-df-f2}\\
 D_hE_J(0,u_0)[h]
 ={}&\langle v_0\rangle_{\tau} h
 -\gamma_0w^{-1}\mathcal D'(wh)
 \notag\\
 &+\gamma_0w^{-1}(D\bar\omega)[0,h]
          \,\partial_{n}\Psi\circ c_0.             \label{eq:exact-dh-f2}
\end{align}
Their types are, in the displayed order,
\[
 H^{m+1}\to H^{m+1},\qquad
 H^m\to H^{m+1},\qquad
 H^{m+1}\to H^m,\qquad
 H^m\to H^m.
\]
The scalar derivatives are the four columns
\begin{equation}
\label{eq:exact-scalar-columns}
 D_\mu E_I[\dot\mu]=-\dot\mu\,1,
 \quad D_\beta E_I[\dot\beta]=0,
 \quad D_\mu E_J[\dot\mu]=0,
 \quad D_\beta E_J[\dot\beta]=-\dot\beta\,1.
\end{equation}
In particular, if
\(L=D_uF(0,u_0)\colon\mathcal X\to\mathcal Y\), then
\begin{equation}
\label{eq:full-unnormalized-derivative-formula}
 L(f,h,\dot\mu,\dot\beta)
 =\begin{pmatrix}
 D_fE_I[f]+D_hE_I[h]-\dot\mu\,1\\
 D_fE_J[f]+D_hE_J[h]-\dot\beta\,1
 \end{pmatrix},
\end{equation}
where the four function terms are exactly
\eqref{eq:exact-df-f1}--\eqref{eq:exact-dh-f2}.
\end{proposition}

\begin{proof}
Because \(E_I\) is the stream-function trace on the moving curve,
Definition~\ref{def:eulerian-variation} gives
\[
 D_{(f,h)}E_I[f,h]
 =\langle\varphi\rangle+f\langle\partial_{n}\psi_0\rangle.
\]
Substitution of \eqref{eq:exact-average-value-trace} and
\eqref{eq:reference-average-normal-velocity} gives
\eqref{eq:exact-df-f1} and \eqref{eq:exact-dh-f1}.

For the second residual write, only in this proof,
\[
 \mathfrak n_{\mathrm{av}}(f,h)
 =\Bigl\langle\partial_{n_f}\psi_{c_f,\gamma_{f,h}}\Bigr\rangle\circ c_f,
 \qquad
 E_J=-(\gamma_0+h)w_{0,f}^{-1}\mathfrak n_{\mathrm{av}}(f,h)-\beta.
\]
Lemma~\ref{lem:geometric-first-variations}, the streamline condition, and
\eqref{eq:exact-average-normal-trace} give
\[
 D_{(f,h)}\mathfrak n_{\mathrm{av}}(0,0)[f,h]
 =\langle\partial_{n}\varphi\rangle
   +f\langle\partial_n^2\psi_0\rangle;
\]
the term \(-f_t\langle\partial_{\tau}\psi_0\rangle\) vanishes by the reference
streamline condition.  Moreover,
\[
 D_fw_{0,f}^{-1}\big|_{f=0}[f]
 =-w^{-1}f\partial_{n}\sigma,
 \qquad
 \mathfrak n_{\mathrm{av}}(0,0)=-w\langle v_0\rangle_{\tau}.
\]
Applying the product rule and separating \(f\) and \(h\) now gives exactly
\eqref{eq:exact-df-f2} and \eqref{eq:exact-dh-f2}.  Finally,
\eqref{eq:exact-scalar-columns} follows directly from the terms
\(-\mu\) and \(-\beta\) in the two residuals.  Fixed-operator mapping
properties, smooth multiplication, and the fact that
\((D\bar\omega)[\cdot,\cdot]\) is scalar-valued prove all four stated
types.  This accounts explicitly for the geometric evaluation terms, the
conformal-factor term, the second-normal-derivative terms, and both rank-one
background terms.
\end{proof}

\subsection{The mixed-order Fredholm analysis}
\label{subsec:mixed-order-fredholm-analysis}

\begin{proposition}[Principal mixed-order matrix and compact remainder]
\label{prop:principal-mixed-matrix}
As an operator
\(H^{m+1}\times H^m\to H^{m+1}\times H^m\),
the principal mixed-order part is
\begin{equation}
\label{eq:principal-mixed-matrix}
 \begin{pmatrix}
 M_{-w\langle v_0\rangle_{\tau}}&\mathcal S_0M_{w}\\
 M_{-\gamma_0w^{-1}}\mathcal T_0M_{\gamma_0w}
   &M_{\langle v_0\rangle_{\tau}}
 \end{pmatrix}.
\end{equation}
The differences between the four derivative entries and the corresponding
entries of
\eqref{eq:principal-mixed-matrix} are
\begin{equation}
\begin{aligned}
 D_fE_I(0,u_0)[f]-M_{-w\langle v_0\rangle_{\tau}}f
 ={}&\mathcal D(\gamma_0wf)\\
 &+\mathcal S\!\left[
   \bigl(\gamma_0w\partial_{n}\sigma
        +[\partial_n^2\psi_0]\bigr)f\right]\\
 &-(D\bar\omega)[f,0]\,\Psi\circ c_0.
\end{aligned}
\label{eq:r11}
\end{equation}
\begin{equation}
 D_hE_I(0,u_0)[h]-\mathcal S_0(wh)
 =(\mathcal S-\mathcal S_0)(wh)
 -(D\bar\omega)[0,h]\,\Psi\circ c_0.
\label{eq:r12}
\end{equation}
\begin{equation}
\begin{aligned}
 D_fE_J(0,u_0)[f]
 +\gamma_0w^{-1}\mathcal T_0(\gamma_0wf)
 ={}&-\gamma_0w^{-1}
       (\mathcal T-\mathcal T_0)(\gamma_0wf)\\
 &-\gamma_0(\partial_{n}\sigma)\langle v_0\rangle_{\tau} f\\
 &-\gamma_0w^{-1}\mathcal D'\!\left[
   \bigl(\gamma_0w\partial_{n}\sigma
        +[\partial_n^2\psi_0]\bigr)f\right]\\
 &+\gamma_0w^{-1}(D\bar\omega)[f,0]
       \,\partial_{n}\Psi\circ c_0\\
 &-\gamma_0w^{-1}f\langle\partial_n^2\psi_0\rangle.
\end{aligned}
\label{eq:r21}
\end{equation}
\begin{equation}
\begin{aligned}
 D_hE_J(0,u_0)[h]-\langle v_0\rangle_{\tau} h
 ={}&-\gamma_0w^{-1}\mathcal D'(wh)\\
 &+\gamma_0w^{-1}(D\bar\omega)[0,h]
       \,\partial_{n}\Psi\circ c_0.
\end{aligned}
\label{eq:r22}
\end{equation}
Every left-hand difference in \eqref{eq:r11}--\eqref{eq:r22} is compact
between its corresponding mixed-order spaces.
\end{proposition}

\begin{proof}
Subtracting the four entries of \eqref{eq:principal-mixed-matrix} from
\eqref{eq:exact-df-f1}--\eqref{eq:exact-dh-f2} gives
\eqref{eq:r11}--\eqref{eq:r22}, so the list is exhaustive.  We record the
compactness mechanism term by term, using
Fact~\ref{fact:fixed-boundary-mappings} and
Lemma~\ref{lem:compact-remainders}.

For the difference in \eqref{eq:r11}, from $H^{m+1}$ to $H^{m+1}$, the
lower-order double layer
maps \(\mathcal D(\gamma_0w\,\cdot)\) into \(H^{m+2}\), and the single
layer applied to the smooth coefficient times \(f\) also maps into
\(H^{m+2}\).  Each is followed by the compact embedding
\(H^{m+2}\hookrightarrow H^{m+1}\).  The last term factors through the
one-dimensional range spanned by \(\Psi\circ c_0\).

For the difference in \eqref{eq:r12}, from $H^m$ to $H^{m+1}$, the exact/model difference
\(\mathcal S-\mathcal S_0\) has a smooth kernel and maps \(H^m\) into
\(H^{m+2}\); the second term again has rank one.  Thus both are compact into
\(H^{m+1}\).

For the difference in \eqref{eq:r21}, from $H^{m+1}$ to $H^m$, the exact/model difference
\(\mathcal T-\mathcal T_0\) is one order below the hypersingular principal
part and maps into \(H^{m+1}\).  The two multiplication terms containing
\(f\) also map into \(H^{m+1}\).  The \(\mathcal D'\) term gains at least
one derivative, and the background term has rank one.  Every non-finite-rank
term therefore factors through \(H^{m+1}\hookrightarrow H^m\).

For the difference in \eqref{eq:r22}, from $H^m$ to $H^m$, the
$\mathcal D'$ term maps into
\(H^{m+1}\), and the background term has rank one.  Rellich compactness
finishes this block.  The zero modes omitted by the nonzero-mode model
identities are finite-dimensional and hence compact.  These mechanisms cover
all terms in \eqref{eq:r11}--\eqref{eq:r22}.
\end{proof}

We also record the two commutator gains, valid for every
smooth periodic coefficient \(b\) and every real \(r\):
\begin{equation}
\label{eq:model-commutator-gains}
 [\mathcal S_0,M_b]\colon H^r\longrightarrow H^{r+2},
 \qquad
 [\mathcal T_0,M_b]\colon H^{r+1}\longrightarrow H^{r+1}.
\end{equation}
As in Lemma~\ref{lem:compact-remainders}, they follow directly from the
multiplier differences
\(|k|^{-1}-|j|^{-1}\) and \(|k|-|j|\), with the zero frequencies separated
as finite-rank terms.  Thus each commutator is one order below its model
operator in the mixed scale.

\begin{lemma}[Effective scalar factor]
\label{lem:effective-scalar-factor}
The scalar factor associated with \eqref{eq:principal-mixed-matrix} is
\begin{equation}
\label{eq:effective-scalar-factor}
 -w\left(\langle v_0\rangle_{\tau}^2+\frac{\gamma_0^2}{4}\right)
 =-\frac{w}{2}\left((v_{\tau,0}^+)^2+(v_{\tau,0}^-)^2\right).
\end{equation}
It is nonzero at a point exactly when the two one-sided scalar tangential
velocities do not vanish simultaneously there.  Hence the stagnation-free
hypothesis implies
\[
 \left|w\left(\langle v_0\rangle_{\tau}^2
                  +\frac{\gamma_0^2}{4}\right)\right|\geq c>0
 \quad\text{on }c_0.
\]
\end{lemma}

\begin{proof}
The tangential velocity jump gives
\[
 v_{\tau,0}^+=\langle v_0\rangle_{\tau}+\frac{\gamma_0}{2},
 \qquad
 v_{\tau,0}^-=\langle v_0\rangle_{\tau}-\frac{\gamma_0}{2}.
\]
Squaring and adding proves the second equality in
\eqref{eq:effective-scalar-factor}.  Since \(w>0\), the remaining claims
follow, including the uniform lower bound on the compact parameter circle.
Notice that one-sided stagnation and zeros of \(\gamma_0\) are allowed; only
simultaneous stagnation is excluded.
\end{proof}

\begin{proposition}[Two-sided parametrix]
\label{prop:two-sided-parametrix}
Under the stagnation-free hypothesis, the principal operator in
\eqref{eq:principal-mixed-matrix} and hence \(D_{(f,h)}F(0,u_0)\) are Fredholm on
\(H^{m+1}\times H^m\).  More precisely, both possess bounded left and right
inverses modulo compact operators on these fixed domain and target spaces.
\end{proposition}

\begin{proof}
Only within this proof, denote the operator in
\eqref{eq:principal-mixed-matrix} by $\mathfrak l$ and put
\[
 \mathsf a=-w\langle v_0\rangle_{\tau},\qquad
 \mathsf d=\langle v_0\rangle_{\tau},\qquad
 \mathsf p=-\gamma_0w^{-1},\qquad
 \mathsf r=\gamma_0w,\qquad
 \mathsf e=-w\left(\langle v_0\rangle_{\tau}^2+\frac{\gamma_0^2}{4}\right),
\]
and set
\[
 \mathcal B_1=\mathcal S_0M_{w},\qquad
 \mathcal B_2=M_{\mathsf p}\mathcal T_0M_{\mathsf r}.
\]
Each entry has the required mixed order.  Define the bounded operator
\begin{equation}
\label{eq:adjugate-parametrix-matrix}
 \mathcal Q=
 \begin{pmatrix}
 M_{\mathsf d}&-\mathcal B_1\\
 -\mathcal B_2&M_{\mathsf a}
 \end{pmatrix}
 \colon H^{m+1}\times H^m\longrightarrow H^{m+1}\times H^m.
\end{equation}
Direct multiplication gives
\begin{align}
 \mathfrak l\mathcal Q
 &=\begin{pmatrix}
 M_{\mathsf a\mathsf d}-\mathcal B_1\mathcal B_2
   &\mathcal B_1M_{\mathsf a}
      -M_{\mathsf a}\mathcal B_1\\
 \mathcal B_2M_{\mathsf d}
      -M_{\mathsf d}\mathcal B_2
   &M_{\mathsf a\mathsf d}-\mathcal B_2\mathcal B_1
 \end{pmatrix},                                      \label{eq:lq-product}\\
 \mathcal Q\mathfrak l
 &=\begin{pmatrix}
 M_{\mathsf a\mathsf d}-\mathcal B_1\mathcal B_2
   &M_{\mathsf d}\mathcal B_1
      -\mathcal B_1M_{\mathsf d}\\
 M_{\mathsf a}\mathcal B_2
      -\mathcal B_2M_{\mathsf a}
   &M_{\mathsf a\mathsf d}-\mathcal B_2\mathcal B_1
 \end{pmatrix}.                                      \label{eq:ql-product}
\end{align}
The off-diagonal entries are compact by
\eqref{eq:model-commutator-gains}.  For example,
\[
 \mathcal B_1\mathcal B_2
 =M_{w\mathsf p}\mathcal S_0\mathcal T_0
      M_{\mathsf r}
   +[\mathcal S_0,M_{w\mathsf p}]
       \mathcal T_0M_{\mathsf r}.
\]
The second term gains one derivative in the mixed scale.  For every
$k\ne0$, Fact~\ref{fact:model-operator-identity} gives
\(\mathcal S_0\mathcal T_0e_k=-\tfrac14e_k\), while the zero-mode discrepancy
is rank one.  Hence
\[
 \mathcal B_1\mathcal B_2=-\frac14M_{w\mathsf p\mathsf r}+C_{11}.
\]
The same commutator calculation, now using
\(\mathcal T_0\mathcal S_0e_k=-\tfrac14e_k\) for $k\ne0$, gives
\[
 \mathcal B_2\mathcal B_1=-\frac14M_{w\mathsf p\mathsf r}+C_{22}.
\]
Here \(C_{11}\) and \(C_{22}\) are compact, with their only zero-mode parts
finite rank.  Since
\[
 \mathsf a\mathsf d+\frac14w\mathsf p\mathsf r
 =-w\left(\langle v_0\rangle_{\tau}^2+\frac{\gamma_0^2}{4}\right)
 =\mathsf e,
\]
equations \eqref{eq:lq-product}--\eqref{eq:ql-product} yield
\begin{equation}
\label{eq:two-sided-adjugate-products}
 \mathfrak l\mathcal Q=M_{\mathsf e}+C_1,
 \qquad
 \mathcal Q\mathfrak l=M_{\mathsf e}+C_2,
\end{equation}
with compact matrix operators \(C_1,C_2\).  Multiplication by
\(\mathsf e^{-1}\) is bounded by
Lemma~\ref{lem:effective-scalar-factor}.  Thus
the two products
$\mathfrak l(\mathcal QM_{\mathsf e^{-1}})$ and
$(M_{\mathsf e^{-1}}\mathcal Q)\mathfrak l$ differ from the corresponding
identity maps by compact operators.  The two-sided Atkinson argument gives
finite kernel, closed range, and finite cokernel, hence Fredholmness.  Finally,
Proposition~\ref{prop:principal-mixed-matrix} shows that the exact function
block is a compact perturbation of $\mathfrak l$, so it has the same
property.
\end{proof}

\begin{proposition}[Index of the function-variable block]
\label{prop:function-block-fredholm}
Under the stagnation-free hypothesis,
\[
 D_{(f,h)}F(0,u_0)\colon H^{m+1}\times H^m
 \longrightarrow H^{m+1}\times H^m,
\]
is Fredholm of index zero.
\end{proposition}

\begin{proof}
We give the operator-norm homotopy required for the index calculation on one
fixed domain and target.  Let \(J_m\) denote, only in this proof, the fixed
Fourier isomorphism
\[
 J_m(f,h)=
 \left(
   \sum_{k\in\ze}
   \sqrt{1+\left(\frac{2\pi k}{|c_0|}\right)^2}\,
       \widehat f(k) e^{2\pi i kt/|c_0|},
   h\right),
\]
from \(H^{m+1}\times H^m\) onto \(H^m\times H^m\).  The first multiplier
commutes with \(\mathcal S_0\) and \(\mathcal T_0\).  For \(k\ne0\),
\[
 \sqrt{1+\left(\frac{2\pi k}{|c_0|}\right)^2}
 \frac{|c_0|}{4\pi|k|}\longrightarrow\frac12,
 \qquad
 -\frac{\pi|k|}{|c_0|}
 \left(1+\left(\frac{2\pi k}{|c_0|}\right)^2\right)^{-1/2}
 \longrightarrow-\frac12.
\]
The differences are compact Fourier multipliers; their zero-mode
differences are finite rank.  Every commutator of the first multiplier or
its inverse with a smooth multiplication coefficient is one order below the
corresponding diagonal term.
Together with Proposition~\ref{prop:principal-mixed-matrix}, this proves,
with $\mathfrak P_\circ(t)$ and $C_\circ$ denoting the resulting operators,
\begin{equation}
\label{eq:equal-order-conjugation}
 J_mD_{(f,h)}F(0,u_0)J_m^{-1}=M_{\mathfrak P_\circ}+C_\circ,
 \qquad C_\circ\ \hbox{compact},
 \quad\text{on }H^m\times H^m,
\end{equation}
where
\begin{equation}
\label{eq:pointwise-equal-order-matrix}
 \mathfrak P_\circ(t)=\begin{pmatrix}
 -w\langle v_0\rangle_{\tau}&w/2\\[1mm]
 \gamma_0^2/2&\langle v_0\rangle_{\tau}
 \end{pmatrix}.
\end{equation}
Its determinant is
$-w(\langle v_0\rangle_{\tau}^2+\gamma_0^2/4)$.  It follows from
Lemma~\ref{lem:effective-scalar-factor} that \(\mathfrak P_\circ(t)^{-1}\) is smooth and
bounded; hence \(M_{\mathfrak P_\circ}\) is invertible on \(H^m\times H^m\).

On this one fixed space, use the path
\[
 \lambda\longmapsto M_{\mathfrak P_\circ}+(1-\lambda)C_\circ,
 \qquad 0\leq\lambda\leq1.
\]
It is continuous in operator norm, and every member is a compact perturbation
of the invertible operator \(M_{\mathfrak P_\circ}\), hence Fredholm.  Equivalently, its
principal matrix is the same uniformly invertible $\mathfrak P_\circ$ for every
$\lambda$.  The endpoint $M_{\mathfrak P_\circ}$ has index zero, so constancy of the
Fredholm index along the path gives index zero at $\lambda=0$.  Conjugation by
the fixed isomorphism \(J_m\)
preserves the index, proving the assertion.  This proof includes the zero
and low modes in the compact part of the norm-continuous path and uses no
symbol-winding argument.
\end{proof}

We now pass from the function block to the full unnormalized derivative.
Consider the domain-extension operator
\[
 (x,r)\longmapsto D_{(f,h)}F(0,u_0)x
 \quad\text{from}\quad
 (H^{m+1}\times H^m)\times\re^2
 \quad\text{to}\quad\mathcal Y.
\]
Its kernel and range are, respectively,
\[
 \ker D_{(f,h)}F(0,u_0)\times\re^2,
 \qquad
 \ran D_{(f,h)}F(0,u_0).
\]
Its cokernel is that of \(D_{(f,h)}F(0,u_0)\).  Consequently this
domain-extension operator has index two.  By
\eqref{eq:exact-scalar-columns}, the full derivative \(L\) differs from this
domain-extension operator by a map with range in the span of
\((1,0)\) and \((0,1)\), so it is finite rank.  Therefore
\begin{equation}
\label{eq:full-l-index-two}
 L=D_uF(0,u_0)\colon\mathcal X\longrightarrow\mathcal Y
 \quad\text{is Fredholm with}\quad \operatorname{ind} L=2.
\end{equation}
This calculation adjoins domain variables before any normalization equation
is appended.

The derivative of the $s$-dependent normalization is the bounded map
\eqref{eq:physical-normalization-derivative}.
The scalar directions \(\dot\mu,\dot\beta\) do not occur.  The first
component is the derivative of mean normal displacement; the second is the
integral of \eqref{eq:density-first-variation}, not the derivative of the
background-vorticity quotient.  The compatibility assumption is
\eqref{eq:normalization-compatibility}.
It is an explicit hypothesis, to be verified separately in a special
reference configuration; no canonical pair of kernel directions is assumed.

Finally, append the two normalization equations.  Since
\[
 D_u\mathcal F(0,u_0)x
 =\bigl(Lx,D_u\mathcal N(0,u_0)x\bigr),
\]
compare it first with \(x\mapsto(Lx,0)\).  This comparison operator has index
\(\operatorname{ind} L-2=0\), and the difference between the two operators
has finite-dimensional range.  Thus, without using compatibility,
\begin{equation}
\label{eq:proof-normalized-index-zero}
 D_u\mathcal F(0,u_0)\colon\mathcal X
 \longrightarrow\mathcal Y\times\re^2
 \quad\text{is Fredholm of index zero}.
\end{equation}
If \eqref{eq:normalization-compatibility} holds, then
\begin{align}
 \ker D_u\mathcal F(0,u_0)
   &=\ker L\cap\ker D_u\mathcal N(0,u_0),       \label{eq:proof-normalized-kernel}\\
 \ran D_u\mathcal F(0,u_0)
   &=\ran L\times\re^2,      \label{eq:proof-normalized-range}\\
 \coker D_u\mathcal F(0,u_0)
   &\cong\coker L                  \label{eq:proof-normalized-cokernel}
\end{align}
under the canonical quotient map induced by projection onto the first
factor.  Indeed, for \(y=Lx_0\) and \(r\in\re^2\), surjectivity on
\(\ker L\) supplies \(k\in\ker L\) such that
\(D_u\mathcal N(0,u_0)k=r-D_u\mathcal N(0,u_0)x_0\); then
\(D_u\mathcal F(0,u_0)(x_0+k)=(y,r)\).  This proves
\eqref{eq:proof-normalized-range}, while the other two identities follow
directly.
In particular, if
\[
 d=\dim\coker L,
\]
then \(\dim\ker L=d+2\), and compatibility gives
\begin{equation}
\label{eq:normalized-finite-dimensional-counts}
 \dim\ker D_u\mathcal F(0,u_0)
 =\dim\coker D_u\mathcal F(0,u_0)=d.
\end{equation}
Index zero is not being identified with invertibility: invertibility follows
only when \(d=0\).

\begin{lemma}[Kernel and cokernel regularity]
\label{lem:kernel-cokernel-regularity}
For every admissible integer \(m\geq3\), the kernels of the function block,
\(L\), and, under \eqref{eq:normalization-compatibility},
\(D_u\mathcal F(0,u_0)\), consist of smooth functions.  Under the periodic
$L^2( d t)$ dual pairing, their cokernels are represented by smooth
formal-adjoint null vectors, with the scalar-column orthogonality conditions
for $L$.  Under these identifications, the kernel and cokernel spaces are the
same at every admissible Sobolev exponent.  In particular, each of their
finite dimensions is independent of that exponent.
\end{lemma}

\begin{proof}
We make both the kernel and adjoint arguments explicit.  Only in this
proof, put
\[
 \mathfrak P_*(t)=\begin{pmatrix}
 -w\langle v_0\rangle_{\tau}&w/2\\
 \gamma_0^2/2&\langle v_0\rangle_{\tau}
 \end{pmatrix}.
\]
The estimates used
in Proposition~\ref{prop:principal-mixed-matrix}, including the commutator
estimates in \eqref{eq:model-commutator-gains}, hold on every periodic
Sobolev level.  For each real $r$, define the fixed Fourier isomorphism
\[
 J_r(f,h)=
 \left(
   \sum_{k\in\ze}
   \sqrt{1+\left(\frac{2\pi k}{|c_0|}\right)^2}\,
     \widehat f(k) e^{2\pi i kt/|c_0|},
   h
 \right)\colon
 H^{r+1}\times H^r\longrightarrow H^r\times H^r.
\]
Consequently the equal-order realization in
\eqref{eq:equal-order-conjugation} can be sharpened to
\begin{equation}
\label{eq:equal-order-one-gain}
 J_rD_{(f,h)}F(0,u_0)J_r^{-1}=M_{\mathfrak P_*}+C_r,
 \qquad
 C_r\colon H^r\times H^r\longrightarrow H^{r+1}\times H^{r+1},
\end{equation}
for every real \(r\); the finitely many zero modes are smooth.  If
\(\vartheta\) lies in the function-block kernel, then
\[
 \vartheta=-M_{\mathfrak P_*^{-1}}C_r\vartheta.
\]
Thus \(\vartheta\in H^r\times H^r\) implies
\(\vartheta\in H^{r+1}\times H^{r+1}\), and iteration gives smoothness.  If
\((f,h,\dot\mu,\dot\beta)\in\ker L\), the same equation has on its
right-hand side only the one-order-gaining term $C_r\vartheta$ from
\eqref{eq:equal-order-one-gain} and the two smooth
constant functions from \eqref{eq:exact-scalar-columns}; the same bootstrap
applies.  The normalized kernel is a subspace of \(\ker L\).

For the cokernel, use the real periodic distributional \(L^2( d t)\) pairing
componentwise, and let a dagger denote the corresponding formal transpose.
The symmetry of the Dirichlet Green kernel gives
\(\mathcal S^\dagger=\mathcal S\),
\(\mathcal T^\dagger=\mathcal T\), and
\(\mathcal D^\dagger=\mathcal D'\).  Therefore the formal adjoint boundary
transmission system has principal part
\begin{equation}
\label{eq:formal-adjoint-mixed-system}
 \begin{pmatrix}
 M_{-w\langle v_0\rangle_{\tau}}
   &M_{\gamma_0w}\mathcal T_0M_{-\gamma_0w^{-1}}\\
 M_{w}\mathcal S_0&M_{\langle v_0\rangle_{\tau}}
 \end{pmatrix},
\end{equation}
and its lower-order part is the transpose of the four expressions
\eqref{eq:r11}--\eqref{eq:r22}.  After transport by the transpose of the fixed
equal-order isomorphism, the adjoint system is
\begin{equation}
\label{eq:equal-order-adjoint-system}
 M_{\mathfrak P_*^{\mathsf T}}\vartheta_{\mathrm{ad}}
   +C_r^\dagger\vartheta_{\mathrm{ad}}=0,
 \qquad
 C_r^\dagger\colon H^r\times H^r\longrightarrow
 H^{r+1}\times H^{r+1}.
\end{equation}
This is also the boundary system obtained by transposing the two-sided
outer-Dirichlet transmission formulation: single and hypersingular layers
remain self-transpose, while the two double-layer averages are exchanged.
The transposed jump data have the same trace orders, so
Theorem~\ref{thm:transmission-regularity} supplies the corresponding
piecewise adjoint field with the same gain and no scalar compatibility
condition.

Since \(\mathfrak P_*^{\mathsf T}\) is pointwise invertible,
\eqref{eq:equal-order-adjoint-system} gives
\(\vartheta_{\mathrm{ad}}
=-M_{\mathfrak P_*^{-\mathsf T}}C_r^\dagger\vartheta_{\mathrm{ad}}\).  Starting from any
distributional order and iterating shows that every adjoint null vector is
smooth.  A representative of \(\coker L\) is such an adjoint
null vector subject only to the two scalar-column orthogonality conditions
\[
 \int_0^{|c_0|}(\vartheta_{\mathrm{ad}})_1\, d t=0,
 \qquad
 \int_0^{|c_0|}(\vartheta_{\mathrm{ad}})_2\, d t=0,
\]
so it is smooth as well.  Under compatibility,
\eqref{eq:proof-normalized-range} identifies the normalized cokernel with the
same smooth representatives, with zero component in the dual of the appended
\(\re^2\).

Finally, the realizations on different Sobolev levels agree on smooth
functions and as distributional operators.  A kernel or adjoint-kernel
element at one admissible level is smooth by the bootstrap based on
\eqref{eq:equal-order-one-gain} and
hence belongs to every other admissible level; the reverse inclusions are
the same statement with the levels interchanged.  The kernel and cokernel
spaces, and therefore their dimensions, are independent of \(m\).
\end{proof}

The unnormalized Fredholm theorem now follows:
Proposition~\ref{prop:function-block-fredholm} gives index zero for the
function block.  The exact scalar columns
\eqref{eq:exact-scalar-columns} and the domain-extension calculation
\eqref{eq:full-l-index-two} then prove that
\(L=D_uF(0,u_0)\) is Fredholm of index two.  Lemma~
\ref{lem:kernel-cokernel-regularity} proves the asserted Sobolev-order
independence.  This proves Theorem~\ref{thm:unnormalized-fredholm}.

For the normalized Fredholm theorem,
formula \eqref{eq:physical-normalization-derivative} gives the normalization
derivative, and its restriction in
\eqref{eq:normalization-compatibility} is to the kernel of the full
unnormalized derivative \(L\).  The unconditional target-extension argument
gives \eqref{eq:proof-normalized-index-zero}; under the stated surjectivity
hypothesis,
\eqref{eq:proof-normalized-kernel}--\eqref{eq:proof-normalized-cokernel}
give the claimed kernel, range, and canonical cokernel identification.
Lemma~\ref{lem:kernel-cokernel-regularity} supplies Sobolev-order
independence.  This proves Theorem~\ref{thm:normalized-fredholm}.
\subsection{Physical Lyapunov--Schmidt reduction}
\label{subsec:physical-ls-proof}

We now work with the physical parameter-dependent normalized residual
\[
 \mathcal F(s,u)
 =\bigl(F(s,u),\mathcal N(s,u)-\mathcal N(0,u_0)\bigr)
 \colon(-s_0,s_0)\times O\longrightarrow
 \mathcal Y\times\re^2.
\]
In particular, none of the differentiations in this subsection freezes the
normalization
at $s=0$.  Let
\[
 \mathcal K=\ker D_u\mathcal F(0,u_0),
 \qquad
 d=\dim\mathcal K=\dim\coker D_u\mathcal F(0,u_0),
\]
where the last equality is Theorem~\ref{thm:normalized-fredholm}.  Choose
closed complements
\begin{equation}
 \mathcal X=\mathcal K\oplus\mathcal X_1,
 \qquad
 \mathcal Y\times\re^2
 =\ran D_u\mathcal F(0,u_0)\oplus\mathcal C,
 \label{eq:physical-fixed-complements}
\end{equation}
with $\dim\mathcal C=d$.  Because $D_u\mathcal F(0,u_0)$ is Fredholm, its
range is closed and both summands in
\eqref{eq:physical-fixed-complements} admit bounded projections.  Denote by
$P$ the projection onto $\ran D_u\mathcal F(0,u_0)$ and by $Q$
the projection onto $\mathcal C$ along that range.  Their sum is the identity,
and
\begin{equation}
 \ker Q=\ran D_u\mathcal F(0,u_0),
 \qquad
 QD_u\mathcal F(0,u_0)=0.
 \label{eq:physical-projections}
\end{equation}
The restriction
\begin{equation}
 D_u\mathcal F(0,u_0)\big|_{\mathcal X_1}\colon\mathcal X_1
       \longrightarrow\ran D_u\mathcal F(0,u_0)
 \label{eq:complement-isomorphism}
\end{equation}
is a bounded isomorphism.

Consider the $C^2$ map
\[
 (s,\zeta,\xi)\longmapsto P\mathcal F(s,u_0+\zeta+\xi)
\]
on the open subset of
$(-s_0,s_0)\times\mathcal K\times\mathcal X_1$ on which
$u_0+\zeta+\xi\in O$.
At $(0,0,0)$ its derivative in $\xi$ is
$D_u\mathcal F(0,u_0)|_{\mathcal X_1}$.  The parameter-dependent Banach
implicit function
theorem therefore gives product neighborhoods and a unique $C^2$ map
\begin{equation}
 \xi=\xi(s,\zeta)\in\mathcal X_1,
 \qquad
 \xi(0,0)=0,
 \label{eq:physical-reconstruction}
\end{equation}
such that
\begin{equation}
 P\mathcal F\bigl(s,u_0+\zeta+\xi(s,\zeta)\bigr)=0.
 \label{eq:physical-range-equation}
\end{equation}
Differentiation in a direction $\eta\in\mathcal K$ at the origin gives
\[
 D_u\mathcal F(0,u_0)
 \bigl(\eta+D_\zeta\xi(0,0)[\eta]\bigr)=0.
\]
Since $D_u\mathcal F(0,u_0)\eta=0$ and
$D_\zeta\xi(0,0)[\eta]\in\mathcal X_1$,
\eqref{eq:complement-isomorphism} implies
\begin{equation}
 D_\zeta\xi(0,0)=0.
 \label{eq:reconstruction-kernel-derivative}
\end{equation}

\begin{definition}[Physical reduced equation]
\label{def:physical-reduced-equation}
For the fixed complements and projections in
\eqref{eq:physical-fixed-complements}--\eqref{eq:physical-projections}, define
\begin{equation}
 B(s,\zeta)
 \coloneqq Q\mathcal F\bigl(s,u_0+\zeta+\xi(s,\zeta)\bigr).
 \label{eq:physical-reduced-equation}
\end{equation}
This is a $C^2$ map from the product neighborhood of $(0,0)$ in
$(-s_0,s_0)\times\mathcal K$ furnished by the implicit function theorem into
$\mathcal C$.
\end{definition}

\begin{proof}[Proof of Theorem~\ref{thm:physical-lyapunov-schmidt}]
Equations \eqref{eq:physical-range-equation} and
\eqref{eq:physical-reduced-equation}, together with
the fact that the two projections sum to the identity, show that
\[
 \mathcal F\bigl(s,u_0+\zeta+\xi(s,\zeta)\bigr)=0
 \quad\Longleftrightarrow\quad B(s,\zeta)=0.
\]
Conversely, every sufficiently small $u-u_0$ has a unique decomposition
$\zeta+\xi$ in $\mathcal K\oplus\mathcal X_1$.  If $\mathcal F(s,u)=0$, then its
range component is zero, so uniqueness in the implicit function theorem
forces $\xi=\xi(s,\zeta)$.  This proves the equivalence and uniqueness of
the reconstruction.  Proposition~\ref{prop:residual-characterization} then
identifies the zeros of $\mathcal F$ with the weak steady sheets satisfying
the two physical normalizations.  The domain and target in
\eqref{eq:physical-reduced-equation} have the asserted finite dimensions by
Theorem~\ref{thm:normalized-fredholm}.
\end{proof}

\begin{lemma}[Derivatives of the physical reduced equation]
\label{lem:physical-reduced-derivatives}
The range correction is
\begin{equation}
 \xi_s(0,0)
 =-\left(PD_u\mathcal F(0,u_0)\big|_{\mathcal X_1}\right)^{-1}
     P\,D_s\mathcal F(0,u_0)
 \in\mathcal X_1.
 \label{eq:parameter-complement-correction}
\end{equation}
For $\eta,\eta_1,\eta_2\in\mathcal K$, all derivatives in
\eqref{eq:reduced-first-parameter}--\eqref{eq:full-map-uu-ss} being evaluated at
$(0,u_0)$, one has
\begin{align}
 B_s(0,0)
 &=Q\mathcal F_s,
 \label{eq:reduced-first-parameter}\\
 B_{\zeta\zeta}(0,0)[\eta_1,\eta_2]
 &=Q\mathcal F_{uu}[\eta_1,\eta_2],
 \label{eq:reduced-zz}\\
 B_{s\zeta}(0,0)[\eta]
 &=Q\left(
   \mathcal F_{su}[\eta]
   +\mathcal F_{uu}[\xi_s(0,0),\eta]
  \right),
 \label{eq:reduced-sz}\\
 B_{ss}(0,0)
 &=Q\left(
   \mathcal F_{ss}
   +2\mathcal F_{su}[\xi_s(0,0)]
   +\mathcal F_{uu}[\xi_s(0,0),\xi_s(0,0)]
  \right).
 \label{eq:reduced-ss}
\end{align}
For $q,q_1,q_2\in\mathcal X$, the derivatives of the full physical map are
\begin{align}
 \mathcal F_s
 &=\bigl(F_s,D_s\mathcal N\bigr),
 &\mathcal F_{su}[q]
 &=\bigl(F_{su}[q],D_sD_u\mathcal N[q]\bigr),
 \label{eq:full-map-s-su}\\
 \mathcal F_{uu}[q_1,q_2]
 &=\bigl(F_{uu}[q_1,q_2],D_u^2\mathcal N[q_1,q_2]\bigr),
 &\mathcal F_{ss}
 &=\bigl(F_{ss},D_s^2\mathcal N\bigr),
 \label{eq:full-map-uu-ss}
\end{align}
with the normalization components given explicitly in
\eqref{eq:normalization-s-derivative}--\eqref{eq:normalization-ss-derivative}.
\end{lemma}

\begin{proof}
We first record how the metric dependence of the normalization enters the
reduction.  Define
\begin{equation}
 \mathfrak m(s,f)(t)
 \coloneqq e^{\sigma_s(c_f(t))}|c_f'(t)|,
 \qquad
 \Theta(s,f,h)=(\gamma_0+h)\mathfrak m(s,f).
 \label{eq:normalization-jacobian-factor}
\end{equation}
Then
\begin{equation}
 \mathcal N(s,u)
 =\left(
   \frac1{|c_0|}\int_0^{|c_0|} f(t)\, d t,
   \int_0^{|c_0|}(\gamma_0(t)+h(t))\mathfrak m(s,f)(t)\, d t
  \right).
 \label{eq:physical-normalization-recalled}
\end{equation}
For the parameter derivatives
\eqref{eq:jacobian-first-derivatives}--\eqref{eq:normalization-ss-derivative}, write
\[
 \dot\sigma=\partial_s\sigma_s\big|_{s=0},
 \qquad
 \ddot\sigma=\partial_s^2\sigma_s\big|_{s=0},
 \qquad \mathfrak m_0=e^{\sigma_0\circ c_0}.
\]
All occurrences of these functions and their normal derivatives are
evaluated on $c_0$.  For any smooth ambient scalar $q$, subscripts of order
two and three denote contractions of the Euclidean derivative tensors:
\begin{equation}
 \partial_{nn}q=D^2q[n,n],\qquad
 \partial_{\tau n}q=D^2q[\tau,n],\qquad
 \partial_{nnn}q=D^3q[n,n,n].
 \label{eq:ls-tensor-contractions}
\end{equation}
They are not iterated derivatives of the moving frame.  With the convention
$\partial_t\tau=-\kappa n$ and $\partial_t n=\kappa\tau$ for the Euclidean
curvature, set
\begin{equation}
 \mathfrak a_{\mathrm{sh}}[f]\coloneqq(\partial_{n}\sigma_0+\kappa)f.
 \label{eq:normalization-first-shape-factor}
\end{equation}
The normal graph identity
\[
 |c_f'|^2=(1+\kappa f)^2+(\partial_t f)^2,
\]
gives these derivatives at $(s,f)=(0,0)$:
\begin{align}
 \mathfrak m_s&=\mathfrak m_0\dot\sigma,
 \qquad D_f\mathfrak m[f]=\mathfrak m_0\mathfrak a_{\mathrm{sh}}[f],
 \label{eq:jacobian-first-derivatives}\\
 \mathfrak m_{ss}&=\mathfrak m_0(\ddot\sigma+\dot\sigma^2),
 \notag\\
 \mathfrak m_{sf}[f]
 &=\mathfrak m_0\bigl(f\,\partial_{n}\dot\sigma+\dot\sigma\,\mathfrak a_{\mathrm{sh}}[f]\bigr),
 \label{eq:jacobian-parameter-second-derivatives}\\
 D_f^2\mathfrak m[f_1,f_2]
 &=\mathfrak m_0\Bigl(
   (\partial_t f_1)(\partial_t f_2)
   +(\partial_{nn}\sigma_0-\kappa^2)f_1f_2
   +\mathfrak a_{\mathrm{sh}}[f_1]\mathfrak a_{\mathrm{sh}}[f_2]
  \Bigr).
 \label{eq:jacobian-shape-second-derivative}
\end{align}
Indeed, these are obtained by differentiating
$\log \mathfrak m=\sigma_s(c_f)+\log|c_f'|$; in particular,
\[
 D_f^2\log|c_f'|[f_1,f_2]
 =(\partial_t f_1)(\partial_t f_2)-\kappa^2f_1f_2.
\]

Let $q=(f,h,\dot\mu,\dot\beta)$ and
$q_j=(f_j,h_j,\dot\mu_j,\dot\beta_j)$ be directions in $\mathcal X$.
The first derivative is \eqref{eq:physical-normalization-derivative}.
Direct differentiation of \eqref{eq:physical-normalization-recalled} yields
the remaining derivatives used in
Lemma~\ref{lem:physical-reduced-derivatives}:
\begin{align}
 D_s\mathcal N(0,u_0)
 &=\left(0,
   \int_0^{|c_0|}\gamma_0\mathfrak m_0\dot\sigma\, d t
  \right),
 \label{eq:normalization-s-derivative}\\
 D_sD_u\mathcal N(0,u_0)[q]
 &=\left(0,
   \int_0^{|c_0|} \mathfrak m_0\left[
    \gamma_0\bigl(\partial_{n}\dot\sigma
                    +\dot\sigma(\partial_{n}\sigma_0+\kappa)\bigr)f
    +h\dot\sigma
   \right] d t
  \right),
 \label{eq:normalization-su-derivative}\\
 D_u^2\mathcal N(0,u_0)[q_1,q_2]
 &=\Biggl(0,\int_0^{|c_0|} \mathfrak m_0\Bigl\{
   \gamma_0(\partial_t f_1)(\partial_t f_2)
   \notag\\
 &\qquad
   +\gamma_0\bigl(\partial_{nn}\sigma_0
                  +(\partial_{n}\sigma_0)^2
                  +2\kappa\partial_{n}\sigma_0\bigr)f_1f_2
   \notag\\
 &\qquad
   +(\partial_{n}\sigma_0+\kappa)(h_1f_2+h_2f_1)
  \Bigr\}\, d t\Biggr),
 \label{eq:normalization-uu-derivative}\\
 D_s^2\mathcal N(0,u_0)
 &=\left(0,
   \int_0^{|c_0|}\gamma_0\mathfrak m_0
       (\ddot\sigma+\dot\sigma^2)\, d t
  \right).
 \label{eq:normalization-ss-derivative}
\end{align}
The scalar directions $\dot\mu$ and $\dot\beta$ do not occur because the
normalization is independent of $\mu$ and $\beta$.  Equations
\eqref{eq:normalization-s-derivative}--\eqref{eq:normalization-ss-derivative}
show in particular why replacing $\mathcal N(s,u)$ by $\mathcal N(0,u)$
would change the physical reduced derivatives.

We next compute the pure-parameter derivatives of the residual while holding
the state at $u_0$.
Write
\begin{equation}
 \Theta_s(t)\coloneqq\Theta(s,0,0)(t)
 =\gamma_0(t) e^{\sigma_s(c_0(t))},
 \qquad
 \Theta_0=\gamma_0\mathfrak m_0.
 \label{eq:frozen-parameter-density}
\end{equation}
Then
\begin{equation}
 \dot\Theta=\Theta_0\dot\sigma,
 \qquad
 \ddot\Theta=\Theta_0(\ddot\sigma+\dot\sigma^2),
 \label{eq:frozen-density-derivatives}
\end{equation}
where the coefficients on the right are restricted to $c_0$.  Let
\[
 \mathcal A_s\coloneqq|\mathbb D|_{g_s}
 =\int_{\mathbb D} e^{2\sigma_s}\, d A_{\mathrm e},
 \qquad
 \Gamma_s\coloneqq\int_0^{|c_0|}\Theta_s(t)\, d t,
 \qquad
 \bar\omega_s\coloneqq\frac{\Gamma_s}{\mathcal A_s}.
\]
At $s=0$,
\begin{align}
 \dot{\mathcal A}
 &=2\int_{\mathbb D} e^{2\sigma_0}\dot\sigma\, d A_{\mathrm e},
 &
 \ddot{\mathcal A}
 &=2\int_{\mathbb D} e^{2\sigma_0}
       (\ddot\sigma+2\dot\sigma^2)\, d A_{\mathrm e},
 \label{eq:frozen-area-derivatives}\\
 \dot\Gamma
 &=\int_0^{|c_0|}\Theta_0\dot\sigma\, d t,
 &
 \ddot\Gamma
 &=\int_0^{|c_0|}\Theta_0(\ddot\sigma+\dot\sigma^2)\, d t,
 \label{eq:frozen-circulation-derivatives}\\
 \dot{\bar\omega}
 &=\frac{\dot\Gamma}{\mathcal A_0}
   -\frac{\Gamma_0\dot{\mathcal A}}{\mathcal A_0^2},
 &
 \ddot{\bar\omega}
 &=\frac{\ddot\Gamma}{\mathcal A_0}
   -\frac{\Gamma_0\ddot{\mathcal A}}{\mathcal A_0^2}
   -\frac{2\dot\Gamma\dot{\mathcal A}}{\mathcal A_0^2}
   +\frac{2\Gamma_0(\dot{\mathcal A})^2}{\mathcal A_0^3}.
 \label{eq:frozen-background-derivatives}
\end{align}

The torsion function and its first two parameter derivatives are
\begin{align}
 \Psi_s(z)
 &=\int_{\mathbb D}G(z,y) e^{2\sigma_s(y)}\, d A_{\mathrm e}(y),
 \label{eq:frozen-torsion}\\
 \dot\Psi(z)
 &=2\int_{\mathbb D}G(z,y) e^{2\sigma_0(y)}
        \dot\sigma(y)\, d A_{\mathrm e}(y),
 \label{eq:frozen-torsion-first}\\
 \ddot\Psi(z)
 &=2\int_{\mathbb D}G(z,y) e^{2\sigma_0(y)}
        \bigl(\ddot\sigma(y)+2\dot\sigma(y)^2\bigr)
        \, d A_{\mathrm e}(y).
 \label{eq:frozen-torsion-second}
\end{align}
Equivalently, with zero Dirichlet data on $\partial\mathbb D$,
\begin{align}
 -\triangle_{\mathrm e}\dot\Psi
 &=2 e^{2\sigma_0}\dot\sigma,
 &
 -\triangle_{\mathrm e}\ddot\Psi
 &=2 e^{2\sigma_0}(\ddot\sigma+2\dot\sigma^2).
 \label{eq:frozen-torsion-pde}
\end{align}

Let $\mathcal S$ and $\mathcal D'$ be the fixed trace and adjoint
double-layer operators on $c_0$, and write $\Psi=\Psi_0$.  In
\eqref{eq:frozen-f1-first}--\eqref{eq:frozen-f1-second}, all ambient
functions are restricted to $c_0$ and the pullback symbol $\circ c_0$ is
omitted.  The pure parameter
derivatives of the
streamline residual are
\begin{align}
 \dot E_I
 &=\mathcal S\dot\Theta
   -\dot{\bar\omega}\Psi-\bar\omega_0\dot\Psi,
 \label{eq:frozen-f1-first}\\
 \ddot E_I
 &=\mathcal S\ddot\Theta
   -\ddot{\bar\omega}\Psi
   -2\dot{\bar\omega}\dot\Psi
   -\bar\omega_0\ddot\Psi.
 \label{eq:frozen-f1-second}
\end{align}
For the Bernoulli residual set
\begin{align}
 a_s&\coloneqq\gamma_0 e^{-\sigma_s\circ c_0},
 &
 \mathfrak j_s&\coloneqq\mathcal D'\Theta_s
       -\bar\omega_s\,\partial_{n}\Psi_s\circ c_0,
 \label{eq:frozen-a-b}\\
 a_0&=\gamma_0\mathfrak m_0^{-1},
 &
 \dot a&=-a_0\dot\sigma,
 \qquad
 \ddot a=a_0(\dot\sigma^2-\ddot\sigma),
 \label{eq:frozen-a-derivatives}\\
 \dot{\mathfrak j}
 &=\mathcal D'\dot\Theta
   -\dot{\bar\omega}\,\partial_{n}\Psi
   -\bar\omega_0\,\partial_{n}\dot\Psi,
 \label{eq:frozen-b-first}\\
 \ddot{\mathfrak j}
 &=\mathcal D'\ddot\Theta
   -\ddot{\bar\omega}\,\partial_{n}\Psi
   -2\dot{\bar\omega}\,\partial_{n}\dot\Psi
   -\bar\omega_0\,\partial_{n}\ddot\Psi.
 \label{eq:frozen-b-second}
\end{align}
Since $E_J(s,u_0)=-a_s\mathfrak j_s-\beta_0$, it follows that
\begin{align}
 \dot E_J
 &=-(\dot a\,\mathfrak j_0+a_0\dot{\mathfrak j}),
 \label{eq:frozen-f2-first}\\
 \ddot E_J
 &=-(\ddot a\,\mathfrak j_0
     +2\dot a\,\dot{\mathfrak j}+a_0\ddot{\mathfrak j}).
 \label{eq:frozen-f2-second}
\end{align}
Thus
\begin{equation}
 F_s(0,u_0)=(\dot E_I,\dot E_J),
 \qquad
 F_{ss}(0,u_0)=(\ddot E_I,\ddot E_J)
 \label{eq:frozen-f-derivatives}
\end{equation}
contains every density, area, background-vorticity, torsion, conformal
prefactor, and normal-derivative contribution.

We now display the mixed and state derivatives.  For directions
$(s_i,q_i)=(s_i,f_i,h_i,\mu_i,\beta_i)
\in\re\times\mathcal X$, $i=1,2$, write $\delta_i$ and $\delta_{12}$
for the first and bilinear second variations at $(0,u_0)$ in the directions
$(s_i,q_i)$ and $((s_1,q_1),(s_2,q_2))$, respectively.  Put
\begin{align}
 \phi_i
 &=s_i\dot\sigma+f_i\partial_{n}\sigma_0,
 \notag\\
 \phi_{12}
 &=s_1s_2\ddot\sigma
   +s_1f_2\partial_{n}\dot\sigma
   +s_2f_1\partial_{n}\dot\sigma
   +f_1f_2\partial_{nn}\sigma_0.
 \label{eq:full-second-conformal-data}
\end{align}
The density factor in \eqref{eq:normalization-jacobian-factor} has variations
\begin{align}
 \delta_i\mathfrak m
 &=s_i\mathfrak m_s+D_f\mathfrak m[f_i],
 \notag\\
 \delta_{12}\mathfrak m
 &=s_1s_2\mathfrak m_{ss}
   +s_1\mathfrak m_{sf}[f_2]+s_2\mathfrak m_{sf}[f_1]
   +D_f^2\mathfrak m[f_1,f_2],
 \label{eq:full-second-jacobian-data}\\
 \delta_i\Theta
 &=h_i\mathfrak m_0+\gamma_0\delta_i\mathfrak m,
 \notag\\
 \delta_{12}\Theta
 &=h_1\delta_2\mathfrak m+h_2\delta_1\mathfrak m+\gamma_0\delta_{12}\mathfrak m,
 \label{eq:full-second-density-data}
\end{align}
where every entry on the right is the explicit expression in
\eqref{eq:jacobian-first-derivatives}--\eqref{eq:jacobian-shape-second-derivative}.
For the area, circulation, and background vorticity, set
\begin{align}
 \delta_i\mathcal A&=s_i\dot{\mathcal A},
 &
 \delta_{12}\mathcal A&=s_1s_2\ddot{\mathcal A},
 \notag\\
 \delta_i\Gamma&=\int_0^{|c_0|}\delta_i\Theta\, d t,
 &
 \delta_{12}\Gamma&=\int_0^{|c_0|}\delta_{12}\Theta\, d t.
 \label{eq:full-second-area-circulation}
\end{align}
Direct differentiation of $\bar\omega=\Gamma/\mathcal A$ gives
\begin{align}
 \delta_i\bar\omega
 &=\frac{\delta_i\Gamma}{\mathcal A_0}
   -\frac{\Gamma_0\delta_i\mathcal A}{\mathcal A_0^2},
 \notag\\
 \delta_{12}\bar\omega
 &=\frac{\delta_{12}\Gamma}{\mathcal A_0}
   -\frac{\delta_1\Gamma\,\delta_2\mathcal A
          +\delta_2\Gamma\,\delta_1\mathcal A
          +\Gamma_0\delta_{12}\mathcal A}{\mathcal A_0^2}
   +\frac{2\Gamma_0\delta_1\mathcal A\,\delta_2\mathcal A}
          {\mathcal A_0^3}.
 \label{eq:full-second-background-data}
\end{align}

The shape derivatives of the two moving boundary operators are fixed
exactly by their kernels.  If
$f_{\varepsilon}=\varepsilon_1f_1+\varepsilon_2f_2$, then for
$j=1,2$ and the corresponding indices $i_1,\ldots,i_j$,
\begin{align}
 \bigl(\delta_{i_1\cdots i_j}\mathcal S\bigr)\rho(t)
 &=
 \left.
 \partial_{\varepsilon_{i_1}}\cdots
 \partial_{\varepsilon_{i_j}}
 \int_0^{|c_0|}
 G(c_{f_\varepsilon}(t),c_{f_\varepsilon}(t'))\rho(t')\, d t'
 \right|_{\varepsilon=0},
 \notag\\
 \bigl(\delta_{i_1\cdots i_j}\mathcal D'\bigr)\rho(t)
 &=
 \left.
 \partial_{\varepsilon_{i_1}}\cdots
 \partial_{\varepsilon_{i_j}}
 \int_0^{|c_0|}
 n_{f_\varepsilon}(t)\cdot
 \nabla_zG(c_{f_\varepsilon}(t),c_{f_\varepsilon}(t'))
 \rho(t')\, d t'
 \right|_{\varepsilon=0}.
 \label{eq:exact-first-second-moving-operator-data}
\end{align}
Only the $f_i$ components enter these derivatives.  Lemma
\ref{lem:moving-boundary-operators}, in particular
\eqref{eq:moving-kernel-directional-form}--%
\eqref{eq:moving-kernel-finite-sobolev-bound}, proves that these kernel
derivatives are the bounded Fr\'echet derivatives at the displayed orders.

Set
\[
 T(s,f)=\Psi_s\circ c_f,\qquad
 R(s,f)=\partial_{n_f}\Psi_s\circ c_f.
\]
The evaluation-point and target-normal contributions are
\begin{align}
 \delta_iT
 &=s_i\dot\Psi+f_i\partial_{n}\Psi,
 \notag\\
 \delta_{12}T
 &=s_1s_2\ddot\Psi
   +(s_1f_2+s_2f_1)\partial_{n}\dot\Psi
   +f_1f_2\partial_{nn}\Psi,
 \label{eq:full-second-torsion-trace}\\
 \delta_iR
 &=s_i\partial_{n}\dot\Psi
   +f_i\partial_{nn}\Psi
   -(\partial_t f_i)\partial_{\tau}\Psi,
 \notag\\
 \delta_{12}R
 &=s_1s_2\partial_{n}\ddot\Psi
 \notag\\
 &\quad
   +s_1\bigl(f_2\partial_{nn}\dot\Psi
             -(\partial_t f_2)\partial_{\tau}\dot\Psi\bigr)
   +s_2\bigl(f_1\partial_{nn}\dot\Psi
             -(\partial_t f_1)\partial_{\tau}\dot\Psi\bigr)
 \notag\\
 &\quad
   +f_1f_2\partial_{nnn}\Psi
   -\bigl((\partial_t f_1)f_2+(\partial_t f_2)f_1\bigr)
       \partial_{\tau n}\Psi
 \notag\\
 &\quad
   +\kappa\bigl(f_1\partial_t f_2+f_2\partial_t f_1\bigr)
       \partial_{\tau}\Psi
   -(\partial_t f_1)(\partial_t f_2)\partial_{n}\Psi.
 \label{eq:full-second-torsion-normal}
\end{align}
All torsion terms in these formulas are restricted to $c_0$.  To verify the
last identity, differentiating \eqref{eq:normal-graph-frame} twice gives
\[
 \delta_i n_f=-(\partial_t f_i)\tau,\qquad
 \delta_{12}n_f
 =\kappa(f_1\partial_t f_2+f_2\partial_t f_1)\tau
  -(\partial_t f_1)(\partial_t f_2)n,
\]
and the ordinary second-order chain rule gives
\eqref{eq:full-second-torsion-normal}.

For the conformal prefactor in the Bernoulli residual, put
$W(s,f)=e^{-\sigma_s(c_f)}$ and
$a(s,f,h)=(\gamma_0+h)W(s,f)$.  Then
\begin{align}
 \delta_i W
 &=-\mathfrak m_0^{-1}\phi_i,
 &
 \delta_{12}W
 &=\mathfrak m_0^{-1}(\phi_1\phi_2-\phi_{12}),
 \notag\\
 \delta_i a&=h_i\mathfrak m_0^{-1}+\gamma_0\delta_i W,
 &
 \delta_{12}a
 &=h_1\delta_2W+h_2\delta_1W
   +\gamma_0\delta_{12}W.
 \label{eq:full-second-bernoulli-prefactor}
\end{align}
Finally define, only for this calculation,
\[
 C(s,f,h)\coloneqq\mathcal D'_f\Theta(s,f,h)
       -\bar\omega(s,f,h)R(s,f).
\]
Throughout this calculation, a subscript $0$ on $T$, $R$, $a$, or $C$
denotes evaluation at $(0,u_0)$.  Its variations are
\begin{align}
 \delta_i C
 &=(\delta_i\mathcal D')\Theta_0
   +\mathcal D'\delta_i\Theta
   -(\delta_i\bar\omega)R_0
   -\bar\omega_0\delta_iR,
 \notag\\
 \delta_{12}C
 &=(\delta_{12}\mathcal D')\Theta_0
   +(\delta_1\mathcal D')\delta_2\Theta
   +(\delta_2\mathcal D')\delta_1\Theta
   +\mathcal D'\delta_{12}\Theta
 \notag\\
 &\quad
   -(\delta_{12}\bar\omega)R_0
   -(\delta_1\bar\omega)\delta_2R
   -(\delta_2\bar\omega)\delta_1R
   -\bar\omega_0\delta_{12}R.
 \label{eq:full-second-bernoulli-bracket}
\end{align}
Consequently the residual derivatives are
\begin{align}
 \delta_iE_I
 &=(\delta_i\mathcal S)\Theta_0+\mathcal S\delta_i\Theta
   -(\delta_i\bar\omega)T_0
   -\bar\omega_0\delta_iT-\mu_i,
 \notag\\
 \delta_{12}E_I
 &=(\delta_{12}\mathcal S)\Theta_0
   +(\delta_1\mathcal S)\delta_2\Theta
   +(\delta_2\mathcal S)\delta_1\Theta
   +\mathcal S\delta_{12}\Theta
 \notag\\
 &\quad
   -(\delta_{12}\bar\omega)T_0
   -(\delta_1\bar\omega)\delta_2T
   -(\delta_2\bar\omega)\delta_1T
   -\bar\omega_0\delta_{12}T,
 \label{eq:full-first-second-streamline}\\
 \delta_iE_J
 &=-(\delta_i a)C_0-a_0\delta_i C-\beta_i,
 \notag\\
 \delta_{12}E_J
 &=-(\delta_{12}a)C_0
   -(\delta_1a)\delta_2C-(\delta_2a)\delta_1C
   -a_0\delta_{12}C.
 \label{eq:full-first-second-bernoulli}
\end{align}
Thus, writing $q=(f,h,\dot\mu,\dot\beta)$,
\begin{equation}
 F_{su}(0,u_0)[q]
 =D^2F(0,u_0)[(1,0),(0,q)],
 \qquad
 F_{uu}(0,u_0)[q_1,q_2]
 =D^2F(0,u_0)[(0,q_1),(0,q_2)],
 \label{eq:explicit-mixed-state-residual-data}
\end{equation}
where the right-hand sides are given by
\eqref{eq:full-second-conformal-data}--%
\eqref{eq:full-first-second-bernoulli}.  Choosing both directions equal to
$(1,0)$ recovers \eqref{eq:frozen-f-derivatives}.  Equations
\eqref{eq:full-second-conformal-data}--\eqref{eq:full-first-second-bernoulli}
account for the moving operators, evaluation maps, coefficients, background
vorticity, and scalar columns; the normalization terms were computed above.
It remains to differentiate the range and reduced equations.

Differentiate the range equation \eqref{eq:physical-range-equation} in $s$.
At the origin this gives
\[
 P\bigl(\mathcal F_s
   +D_u\mathcal F(0,u_0)\xi_s(0,0)\bigr)=0,
\]
and \eqref{eq:complement-isomorphism} gives
\eqref{eq:parameter-complement-correction}.  Differentiate
\eqref{eq:physical-reduced-equation}.  Since
$QD_u\mathcal F(0,u_0)=0$ and
$D_\zeta\xi(0,0)=0$, the terms containing
$D_u\mathcal F(0,u_0)\xi_{\zeta\zeta}$,
$D_u\mathcal F(0,u_0)\xi_{s\zeta}$, and
$D_u\mathcal F(0,u_0)\xi_{ss}$ vanish after applying $Q$.  The remaining
chain-rule terms are exactly
\eqref{eq:reduced-first-parameter}--\eqref{eq:reduced-ss}.  Finally,
\eqref{eq:full-map-s-su}--\eqref{eq:full-map-uu-ss} follow by differentiating
the definition of $\mathcal F$; their second components are
\eqref{eq:normalization-s-derivative}--\eqref{eq:normalization-ss-derivative}.
\end{proof}

\begin{proof}[Proof of Theorem~\ref{thm:reduced-obstructions}]
Let $u(s)$ be a $C^2$ curve of normalized solutions through $u_0$.
Unique reconstruction gives
\[
 u(s)=u_0+\zeta(s)+\xi(s,\zeta(s)),
 \qquad
 B(s,\zeta(s))=0,
\]
where $\zeta(s)\in\mathcal K$ is $C^2$, $\zeta(0)=0$, and
$\eta=\zeta'(0)$.  Because $B(0,0)=0$ and $B_\zeta(0,0)=0$, the first
derivative at zero is
\[
 0=B_s(0,0).
\]
The second derivative is
\[
 0=B_{ss}(0,0)+2B_{s\zeta}(0,0)[\eta]
   +B_{\zeta\zeta}(0,0)[\eta,\eta].
\]
Dividing this identity by two gives the second condition in
\eqref{eq:first-second-obstructions}.
Lemma~\ref{lem:physical-reduced-derivatives} identifies these quantities with
the derivatives of the parameter-dependent normalized map.
This proves only the stated necessary conditions.  The available joint
$C^2$ regularity does not furnish a differentiable rescaled map, and no
converse branch assertion is made.
\end{proof}

\begin{proof}[Proof of Corollary~\ref{cor:nondegenerate-continuation}]
If $d=0$, Theorem~\ref{thm:normalized-fredholm} makes
$D_u\mathcal F(0,u_0)$ an isomorphism.
The parameter-dependent Banach implicit function theorem applied directly to
$\mathcal F$ gives the unique normalized $C^2$ branch in a product
neighborhood.  Differentiation of $\mathcal F(s,u(s))=0$ gives
\eqref{eq:implicit-branch-derivative}; its normalization component is
\eqref{eq:normalization-s-derivative}, not zero in general.
Proposition~\ref{prop:residual-characterization} converts every zero to a
weak steady sheet, and the definition of $\mathcal N$ gives the two stated
physical constraints.
\end{proof}
\section{Applications}
\label{sec:applications}

\subsection{Rotationally symmetric reference metrics}
\label{subsec:rotational-application}

Radial metrics and circular sheets are the basic setting in which the
reference state is explicit and the normalized derivative diagonalizes in
Fourier modes.  They therefore turn the abstract degeneracy index into a
finite, computable resonance problem.  In this section, $\rho$ denotes the
circle radius and the sheet density remains denoted by $\Theta$.

\begin{definition}[Rotationally symmetric reference data]
\label{def:rotational-setting}
Let
\[
 g_0= e^{2\sigma(r)}g_{\mathrm e},
\]
on the unit disk, where the radial function $z\mapsto\sigma(|z|)$ belongs to
$C^\infty(\overline{\mathbb D})$.  Fix
$0<\rho<1$ and $\gamma\in\re\setminus\{0\}$.  The reference support
and strength are
\begin{equation}
 c_0(t)=\rho
 \left(\cos\frac{t}{\rho},\sin\frac{t}{\rho}\right),
 \qquad
 \gamma_0(t)=\gamma,
 \qquad
 t\in\re/(2\pi\rho)\ze.
 \label{eq:radial-arclength-parametrization}
\end{equation}
Thus, if $\theta$ is the angular coordinate, then
\begin{equation}
 t=\rho\theta,\qquad |c_0|=2\pi\rho,
 \qquad  e^{ i k\theta}= e^{2\pi i kt/|c_0|}.
 \label{eq:radial-angle-arclength}
\end{equation}
The unit normal is chosen radially outward.  A prescribed $C^3$ conformal
path $s\mapsto\sigma_s$ satisfies $\sigma_0(z)=\sigma(|z|)$; the metrics for
$s\ne0$ need not be rotationally symmetric.

Define
\begin{align}
 A(r)
 &\coloneqq2\pi\int_0^r  e^{2\sigma(q)}q\, d q,
 &
 U&\coloneqq\frac{A(\rho)}{A(1)},
 \label{eq:radial-area-u}\\
 x&\coloneqq\rho^2,
 \label{eq:radial-x}\\
 \gamma_\sigma&\coloneqq\gamma  e^{\sigma(\rho)},
 &
 p_\sigma&\coloneqq1+\rho\sigma'(\rho),
 \label{eq:radial-gamma-p}\\
 Z&\coloneqq\frac{2\pi\rho^2 e^{2\sigma(\rho)}}{A(1)}.
 \label{eq:radial-z}
\end{align}
Here $A(r)$ is the metric area of the disk of radius $r$.
\end{definition}

\begin{proposition}[Circular reference sheet]
\label{prop:circular-reference-sheet}
For the data in Definition~\ref{def:rotational-setting}, set
\begin{align}
 \bar\omega_0
 &=\frac{2\pi\rho\gamma  e^{\sigma(\rho)}}{A(1)},
 \label{eq:radial-background-vorticity}\\
 \psi_0'(r)
 &=\frac{\bar\omega_0A(r)}{2\pi r}
   -\frac{\rho\gamma  e^{\sigma(\rho)}}{r}\,
       \mathbf 1_{\{r>\rho\}},
 \label{eq:radial-stream-derivative}\\
 \psi_0(r)&=-\int_r^1\psi_0'(q)\, d q,
 &
 \mu_0&=\psi_0(\rho)
 =\rho\gamma  e^{\sigma(\rho)}
   \int_\rho^1\left(1-\frac{A(q)}{A(1)}\right)\frac{ d q}{q},
 \label{eq:radial-stream-mu}\\
 \langle v_0\rangle_{\tau}
 &=\gamma\left(\frac12-U\right),
 &
 \beta_0&=\gamma^2\left(\frac12-U\right).
 \label{eq:radial-velocity-beta}
\end{align}
The value of \eqref{eq:radial-stream-derivative} at the origin is its regular
limit, and its two one-sided values are used at $r=\rho$.  Then
$u_0=(0,0,\mu_0,\beta_0)$ is a zero of $F(0,\cdot)$, and
$(c_0,\gamma)$ is a smooth embedded weak steady vortex sheet.  It satisfies
the stagnation-free condition \eqref{eq:standing-stagnation-free}.
\end{proposition}

\begin{proof}
The circulation is $2\pi\rho\gamma  e^{\sigma(\rho)}$, so the
zero-total-vorticity convention in \eqref{eq:circulation-background} gives
\eqref{eq:radial-background-vorticity}.  Integrating the radial Poisson
equation over the disk of radius $r$ gives
\[
 2\pi r\psi_0'(r)
 =\bar\omega_0A(r)
  -2\pi\rho\gamma  e^{\sigma(\rho)}\mathbf 1_{\{r>\rho\}},
\]
which proves \eqref{eq:radial-stream-derivative}.  The circle is a level set
of $\psi_0$, and integration from $\rho$ to the outer Dirichlet boundary
gives \eqref{eq:radial-stream-mu}.  Averaging the two radial derivatives at
the sheet and converting the stream derivative to the scalar tangential
velocity gives the first identity in \eqref{eq:radial-velocity-beta}; its
product with the constant strength gives $\beta_0$.  The streamline and
Bernoulli characterizations therefore imply $F(0,u_0)=0$ and weak
steadiness.  Finally, the two one-sided scalar tangential velocities differ
by the nonzero sheet strength $\gamma$, so they cannot vanish
simultaneously.  Smoothness, embeddedness, and separation are immediate
from \eqref{eq:radial-arclength-parametrization}.
\end{proof}

\begin{proposition}[Exact Fourier mode matrices]
\label{prop:exact-radial-mode-matrices}
Use the real Fourier planes
\[
 \mathcal E_k
 =\operatorname{span}_{\re}\{\cos(k\theta),\sin(k\theta)\},
 \qquad k\geq1,
\]
with the conversion \eqref{eq:radial-angle-arclength}.  The cosine and sine
invariant subspaces of $\mathcal E_k\times \mathcal E_k$ carry the same real
matrix for the shape--strength part of $D_uF(0,u_0)$:
\begin{equation}
 \begin{pmatrix}I_f&I_h\\J_f&J_h\end{pmatrix},
 \label{eq:radial-mode-matrix}
\end{equation}
where
\begin{align}
 I_f
 &=\frac{\gamma_\sigma}{2}
   \left(1+\frac{p_\sigma}{k}\right)(1-x^k)
   -\gamma_\sigma+\gamma_\sigma U,
 \label{eq:radial-l1f}\\
 I_h
 &= e^{\sigma(\rho)}\frac{\rho}{2k}(1-x^k),
 \label{eq:radial-l1h}\\
 J_f
 &=\gamma  e^{-\sigma(\rho)}\left[
   \frac{\gamma_\sigma p_\sigma}{\rho}
      \left(U-\frac12\right)
   -\bar\omega_0 e^{2\sigma(\rho)}
   +\frac{\gamma_\sigma}{2\rho}
      \bigl((k+p_\sigma)x^k+k\bigr)
  \right],
 \label{eq:radial-l2f}\\
 J_h
 &=\gamma\left(\frac{1+x^k}{2}-U\right).
 \label{eq:radial-l2h}
\end{align}
The two normalization differentials vanish on every nonzero Fourier mode.

For the constant mode, use the input coordinates
$(f_0,h_0,\dot\mu,\dot\beta)$ and the output order
$(E_I,E_J,\mathcal N_1,\mathcal N_2)$.  The exact normalized block is
\begin{equation}
 \begin{pmatrix}
 \dfrac{p_\sigma\mu_0}{\rho}-\gamma_\sigma(1-U)
   &\dfrac{\mu_0}{\gamma}&-1&0\\
 -\dfrac{2\pi\gamma^2\rho  e^{2\sigma(\rho)}}{A(1)}
   &2\gamma(\frac12-U)&0&-1\\
 1&0&0&0\\
 2\pi\gamma_\sigma p_\sigma&2\pi\rho  e^{\sigma(\rho)}&0&0
 \end{pmatrix}.
 \label{eq:radial-zero-mode-matrix}
\end{equation}
In particular, its determinant is
\begin{equation}
 2\pi\rho  e^{\sigma(\rho)}>0.
 \label{eq:radial-zero-mode-determinant}
\end{equation}

Define the determinant directly from the four entries by
\[
 D_k^\sigma(\rho)\coloneqq I_fJ_h-I_hJ_f,
 \qquad
 \mathring D_k^\sigma(\rho)
 \coloneqq\frac{ D_k^\sigma(\rho)}{\gamma^2 e^{\sigma(\rho)}}.
\]
Then the exact determinant identity is
\begin{equation}
 \mathring D_k^\sigma(\rho)
 =-\left(\frac{1+x^k}{2}-U\right)^2
  -\frac{1-x^{2k}}4
  +\frac{1-x^k}{2k}\,[Z+p_\sigma(1-2U)].
 \label{eq:radial-determinant-identity}
\end{equation}
\end{proposition}

\begin{proof}
To compute these blocks, write $z=re^{i\theta}$ and
$y=r'e^{i\theta'}$.  The Euclidean Dirichlet Green kernel of the unit disk,
with $-\triangle_{\mathrm e}G=\delta$, has the Fourier expansion
\begin{equation}
 G(z,y)
 =-\frac1{2\pi}\log r_>
  +\sum_{k\geq1}\frac1{2\pi k}
   \left[\left(\frac{r_<}{r_>}\right)^k-(rr')^k\right]
   \cos k(\theta-\theta'),
 \label{eq:disk-green-fourier}
\end{equation}
where $r_<=\min(r,r')$ and $r_>=\max(r,r')$.  For the density $e_k(\theta')= e^{ i k\theta'}$ per Euclidean arclength
on $r'=\rho$, direct
integration gives
\[
 \mathcal V[e_k](r,\theta)
 =\frac{\rho}{2k}
 \begin{cases}
  ((r/\rho)^k-(r\rho)^k) e^{ i k\theta},&r<\rho,\\
  ((\rho/r)^k-(r\rho)^k) e^{ i k\theta},&r>\rho.
 \end{cases}
 \, .
\]
Consequently the fixed exact circle operators satisfy
\begin{align}
 \mathcal S e_k&=\frac{\rho}{2k}(1-x^k)e_k,
 &
 \mathcal D'e_k&=-\frac{x^k}{2}e_k,
 \label{eq:circle-s-dprime}\\
 \mathcal D e_k&=-\frac{x^k}{2}e_k,
 &
 \mathcal T e_k&=-\frac{k}{2\rho}(1+x^k)e_k.
 \label{eq:circle-d-t}
\end{align}
For a nonzero mode the variation of $\bar\omega$ is zero.  The transmission
jumps generated by $f=e_k$ are
\[
 [\varphi]=\gamma_\sigma e_k,
 \qquad
 [\partial_r\varphi]
 =-\frac{\gamma_\sigma p_\sigma}{\rho}e_k,
\]
whereas those generated by $h=e_k$ are
\[
 [\varphi]=0,
 \qquad
 [\partial_r\varphi]=- e^{\sigma(\rho)}e_k.
\]
The transmission representation and
\eqref{eq:circle-s-dprime}--\eqref{eq:circle-d-t} therefore give
\[
 \varphi_f
 =\mathcal W[\gamma_\sigma e_k]
  +\mathcal V[(\gamma_\sigma p_\sigma/\rho)e_k],
 \qquad
 \varphi_h=\mathcal V[ e^{\sigma(\rho)}e_k].
\]
At the radial base,
\begin{align*}
 \langle\psi_0'\rangle
 &=\gamma_\sigma\left(U-\frac12\right),
 &
 [\psi_0'']&=\frac{\gamma_\sigma}{\rho},\\
 \langle\psi_0''\rangle
 &=\bar\omega_0 e^{2\sigma(\rho)}
  -\frac{\gamma_\sigma U}{\rho}
  +\frac{\gamma_\sigma}{2\rho}.
\end{align*}
Put $w= e^{\sigma(\rho)}$ and let $\mathcal D'_k=-x^k/2$ and
$\mathcal T_k=-k(1+x^k)/(2\rho)$ denote the two multipliers above.  Substitution
in the four entries of Proposition~\ref{prop:complete-linearization} gives
the intermediate assembly
\begin{align}
 I_f
 &=-w\langle v_0\rangle_{\tau}-\frac{\gamma_\sigma x^k}{2}
   +\frac{\gamma_\sigma p_\sigma}{2k}(1-x^k),
 \label{eq:radial-l1f-assembly}\\
 I_h
 &=w\frac{\rho}{2k}(1-x^k),
 \label{eq:radial-l1h-assembly}\\
 J_f
 &=-\gamma w^{-1}\mathcal T_k\gamma_\sigma
   -\gamma\sigma'(\rho)\langle v_0\rangle_{\tau}
   -\gamma w^{-1}\mathcal D'_k
      \frac{\gamma_\sigma p_\sigma}{\rho}
   -\gamma w^{-1}\langle\psi_0''\rangle,
 \label{eq:radial-l2f-assembly}\\
 J_h
 &=\langle v_0\rangle_{\tau}-\gamma\mathcal D'_k.
 \label{eq:radial-l2h-assembly}
\end{align}
Using the displayed base derivatives and
$\langle v_0\rangle_{\tau}=\gamma(1/2-U)$ reduces these four lines to
\eqref{eq:radial-l1f}--\eqref{eq:radial-l2h}.

For the zero mode, differentiate the formula for $\mu_0$ in
\eqref{eq:radial-stream-mu} and
$\beta_0=\gamma^2(1/2-U)$ with respect to the circle radius and constant
strength.  Since
\[
 \frac{ d}{ d\rho}\int_\rho^1
 \left(1-\frac{A(q)}{A(1)}\right)\frac{ d q}{q}
 =-\frac{1-U}{\rho},
 \qquad
 U'(\rho)=\frac{2\pi\rho  e^{2\sigma(\rho)}}{A(1)},
\]
the first two rows of \eqref{eq:radial-zero-mode-matrix} follow.  The two
normalization derivatives are
\[
 D\mathcal N_1(f_0,h_0,\dot\mu,\dot\beta)=f_0,
\]
\[
 D\mathcal N_2(f_0,h_0,\dot\mu,\dot\beta)
 =2\pi\gamma_\sigma p_\sigma f_0
  +2\pi\rho  e^{\sigma(\rho)}h_0,
\]
which give its final two rows.  Expansion along the third row gives
\eqref{eq:radial-zero-mode-determinant}.

Finally, put
\[
 \mathsf b_k=\frac{1+x^k}{2}-U,\qquad
 \mathsf c_k=p_\sigma\left(U-\frac12\right)-Z
       +\frac{(k+p_\sigma)x^k+k}{2}.
\]
The four entries already proved obey
\[
 \frac{I_f}{\gamma_\sigma}
 =-\mathsf b_k+\frac{p_\sigma}{2k}(1-x^k),
 \qquad
 \frac{J_h}{\gamma}=\mathsf b_k,
 \qquad
 J_f=\frac{\gamma^2}{\rho}\mathsf c_k.
\]
Substitution in $I_fJ_h-I_hJ_f$, followed by
$2\mathsf b_k-x^k=1-2U$, gives
\eqref{eq:radial-determinant-identity}.
\end{proof}

\begin{proposition}[Modal degeneracy and normalization compatibility]
\label{prop:radial-modal-degeneracy}
The set $\{k\in\ze_{\geq1}:D_k^\sigma(\rho)=0\}$ is finite.  More precisely, if
\[
 C_\rho\coloneqq Z+p_\sigma(1-2U),
\]
then
\begin{equation}
 \mathring D_k^\sigma(\rho)
 \leq-\frac{1-x^2}{4}+\frac{|C_\rho|}{2k},
 \label{eq:radial-tail-bound}
\end{equation}
and every integer
\begin{equation}
 k>\frac{2|C_\rho|}{1-x^2}
 \label{eq:radial-tail-cutoff}
\end{equation}
is nonresonant.  Thus it is enough to inspect the modes with
$1\leq k<K_{\mathrm{tail}}$, where
\[
 K_{\mathrm{tail}}
 \coloneqq\left\lceil\frac{2|C_\rho|}{1-x^2}\right\rceil+1.
\]

Every resonant matrix has rank one.  For $k\geq1$ satisfying
$ D_k^\sigma(\rho)=0$,
right and left null vectors for one real copy satisfy
\begin{equation}
 \begin{pmatrix}I_f&I_h\\J_f&J_h\end{pmatrix}
 \begin{pmatrix}I_h\\-I_f\end{pmatrix}=0,
 \qquad
 \begin{pmatrix}I_f&I_h\\J_f&J_h\end{pmatrix}^{\mathsf T}
 \begin{pmatrix}J_h\\-I_h\end{pmatrix}=0.
 \label{eq:radial-null-vectors}
\end{equation}
The normalized kernel is the direct sum, over resonant $k$, of
\begin{align*}
 &(I_h\cos k\theta,-I_f\cos k\theta,0,0),\qquad
   (I_h\sin k\theta,-I_f\sin k\theta,0,0),
\end{align*}
and the normalized cokernel has the corresponding $L^2$ representatives
\begin{align*}
 &(J_h\cos k\theta,-I_h\cos k\theta,0,0),\qquad
   (J_h\sin k\theta,-I_h\sin k\theta,0,0).
\end{align*}
In particular,
\begin{equation}
 d=2\,\bigl|\{k\in\ze_{\geq1}:D_k^\sigma(\rho)=0\}\bigr|.
 \label{eq:radial-degeneracy-count}
\end{equation}
The mean-displacement and circulation normalization satisfies
\eqref{eq:normalization-compatibility} at every radial reference state.
\end{proposition}

\begin{proof}
Since $0<x<1$, one has $1-x^{2k}\geq1-x^2$.  Dropping the first
nonpositive square in \eqref{eq:radial-determinant-identity} and bounding
$1-x^k\leq1$ proves \eqref{eq:radial-tail-bound}, hence finiteness and
\eqref{eq:radial-tail-cutoff}.  Moreover,
\[
 I_h= e^{\sigma(\rho)}\frac{\rho}{2k}(1-x^k)>0.
\]
Thus a singular $2\times2$ block is nonzero and has rank one.  Direct
multiplication verifies both vectors in \eqref{eq:radial-null-vectors};
the cosine and sine copies give the stated real multiplicity.  The
zero-mode matrix is invertible by
\eqref{eq:radial-zero-mode-determinant}, while the normalization rows vanish
on every nonzero mode.  Hence the zero mode contributes neither normalized
kernel nor cokernel, the displayed sums are complete, and
\eqref{eq:radial-degeneracy-count} follows.  Before normalization, the two
zero-mode tangent directions form the concentric radius--strength family;
the last two rows of \eqref{eq:radial-zero-mode-matrix} restrict to an
isomorphism on those directions.  This is exactly
\eqref{eq:normalization-compatibility}.
\end{proof}

\begin{corollary}[Rotationally symmetric continuation or reduction]
\label{cor:radial-continuation}
For the reference state of Definition~\ref{def:rotational-setting}, all
hypotheses of Theorem~\ref{thm:unnormalized-fredholm} hold, and the
compatibility condition \eqref{eq:normalization-compatibility} is satisfied.
Consequently the conclusions of Theorem~\ref{thm:normalized-fredholm} apply.
If
$ D_k^\sigma(\rho)\ne0$ for every $k\geq1$, then for every prescribed $C^3$
conformal metric path through the radial base there is a locally unique
$C^2$ normalized branch in the chosen normal chart.

If $ D_k^\sigma(\rho)=0$ for at least one $k\geq1$, the local normalized
problem is equivalent only, for $(s,\zeta)$ in the product neighborhood of
$(0,0)$ from Definition~\ref{def:physical-reduced-equation}, to the reduced equation
\[
 B(s,\zeta)=0,
 \qquad B(s,\zeta)\in\mathcal C,
 \qquad
 \dim\mathcal K=\dim\mathcal C=d
 =2\,\bigl|\{k\in\ze_{\geq1}:D_k^\sigma(\rho)=0\}\bigr|.
\]
Every $C^2$ branch must satisfy the first- and second-order conditions of
Theorem~\ref{thm:reduced-obstructions}.  Resonance alone gives no branch,
branch count, or branch scale.
\end{corollary}

\begin{proof}
Proposition~\ref{prop:circular-reference-sheet} verifies the steady-state,
regularity, embeddedness, and stagnation-free hypotheses.
Propositions~\ref{prop:exact-radial-mode-matrices}
and~\ref{prop:radial-modal-degeneracy} verify the exact derivative,
normalization compatibility, kernel, and cokernel.  The first assertion is
Corollary~\ref{cor:nondegenerate-continuation}; the second is
Theorems~\ref{thm:physical-lyapunov-schmidt}
and~\ref{thm:reduced-obstructions}.
\end{proof}

\subsection{The constant-curvature disk family}
\label{subsec:constant-curvature-application}

Within the radial class, the constant-curvature metrics form a single
explicit family covering positive, zero, and negative curvature.  Every
geometric coefficient in the modal determinant can then be evaluated, so
the resonance set can be determined for all positive Fourier modes.

\begin{definition}[Constant-curvature setting]
\label{def:constant-curvature-setting}
Let $K\colon(-s_0,s_0)\to(-1,\infty)$ be $C^3$ and define
\begin{equation}
 e^{\sigma_s(r)}=\frac{2}{1+K(s)r^2},\qquad
 g_s=\frac{4}{(1+K(s)r^2)^2}g_{\mathrm e}.
 \label{eq:constant-curvature-family}
\end{equation}
For each $K>-1$, write
\begin{equation}
 \sigma_K(r)\coloneqq\log\frac2{1+Kr^2},\qquad
 g_K\coloneqq e^{2\sigma_K}g_{\mathrm e}.
 \label{eq:constant-curvature-static-notation}
\end{equation}
For the reference linearization write $K=K(0)$ and retain the radius
$0<\rho<1$ and strength $\gamma\ne0$ from
Definition~\ref{def:rotational-setting}.  Put
\begin{equation}
 x=\rho^2.
 \label{eq:constant-curvature-x}
\end{equation}
The restriction $K>-1$ is exactly the range in which the conformal factor is
positive on the closed unit disk.
Moreover,
\begin{equation}
 \partial_r\sigma_K=-\frac{2Kr}{1+Kr^2},
 \qquad
 \triangle_{\mathrm e}\sigma_K
   =-\frac{4K}{(1+Kr^2)^2},
 \qquad
 - e^{-2\sigma_K}\triangle_{\mathrm e}\sigma_K=K,
 \label{eq:constant-gaussian-curvature}
\end{equation}
which verifies that $g_K$ has Gaussian curvature $K$.
\end{definition}

\begin{proposition}[Constant-curvature specialization]
\label{prop:constant-curvature-specialization}
For the metric \eqref{eq:constant-curvature-family},
\begin{align}
 A(r)
 &=\frac{4\pi r^2}{1+Kr^2},
 &
 U&=\frac{x(1+K)}{1+Kx},
 \label{eq:constant-area-u}\\
  e^{\sigma_K(\rho)}&=\frac2{1+Kx},
 &
 \gamma_\sigma&=\frac{2\gamma}{1+Kx},
 \label{eq:constant-weight-gamma}\\
 p_\sigma&=\frac{1-Kx}{1+Kx},
 &
 Z&=\frac{2x(1+K)}{(1+Kx)^2},
 \label{eq:constant-p-z}\\
 \bar\omega_0&=\frac{\rho\gamma(1+K)}{1+Kx},
 &
 \langle v_0\rangle_{\tau}
 &=\gamma\left(\frac12-\frac{x(1+K)}{1+Kx}\right).
 \label{eq:constant-background-velocity}
\end{align}
The torsion function for the fixed metric $g_K$ is
\begin{equation}
 \Psi_K(r)=
 \begin{cases}
  \displaystyle\frac1K
   \log\frac{1+K}{1+Kr^2},&K\ne0,\\[1.2ex]
  1-r^2,&K=0,
 \end{cases}
 \qquad
 \partial_r\Psi_K(r)=-\frac{2r}{1+Kr^2}.
 \label{eq:constant-curvature-torsion}
\end{equation}

For every $k\geq1$, the exact entries in
\eqref{eq:radial-mode-matrix} become
\begin{align}
 I_f
 &=\frac{\gamma}{1+Kx}
   \left[\left(1+\frac{p_\sigma}{k}\right)(1-x^k)
         -2+2U\right],
 \label{eq:constant-l1f}\\
 I_h&=\frac{\rho}{k(1+Kx)}(1-x^k),
 \label{eq:constant-l1h}\\
 J_f
 &=\frac{\gamma^2}{\rho}\left[
   p_\sigma\left(U-\frac12\right)
   -\frac{2x(1+K)}{(1+Kx)^2}
   +\frac{(k+p_\sigma)x^k+k}{2}
  \right],
 \label{eq:constant-l2f}\\
 J_h&=\gamma\left(\frac{1+x^k}{2}-U\right).
 \label{eq:constant-l2h}
\end{align}
\end{proposition}

\begin{proof}
Direct integration of $4r/(1+Kr^2)^2$ gives
\eqref{eq:constant-area-u}; the remaining quantities in
\eqref{eq:constant-weight-gamma}--\eqref{eq:constant-background-velocity} follow
from the definitions in
\eqref{eq:radial-area-u}--\eqref{eq:radial-z} and
\eqref{eq:radial-background-vorticity}.  The function in
\eqref{eq:constant-curvature-torsion} vanishes at $r=1$ and satisfies
\[
 -\frac1r\frac{ d}{ d r}\left(r\frac{d\Psi_K}{dr}\right)
 =\frac4{(1+Kr^2)^2}= e^{2\sigma_K(r)};
\]
the $K=0$ expression is its continuous limit.  Inserting
\eqref{eq:constant-area-u}--\eqref{eq:constant-p-z} in the general radial
blocks yields \eqref{eq:constant-l1f}--\eqref{eq:constant-l2h}.
\end{proof}

\begin{proposition}[Persistent first-mode resonance]
\label{prop:constant-curvature-resonances}
For every $K>-1$, $0<\rho<1$, and $k\geq1$, the normalized determinant has
the exact form
\begin{equation}
 \mathring D_k^{\sigma_K}(\rho)
 =\frac1{(1+Kx)^2}\left[
  (1-x)(x-x^k)
  -\frac{k-1}{2k}(1-x^k)
      \bigl(1+K(K+2)x^2\bigr)
 \right].
 \label{eq:constant-curvature-determinant}
\end{equation}
It satisfies
\begin{equation}
 \mathring D_1^{\sigma_K}(\rho)=0,
 \qquad
 \mathring D_k^{\sigma_K}(\rho)<0\quad(k\geq2).
 \label{eq:constant-curvature-signs}
\end{equation}
At $k=1$ the first column of the matrix in
\eqref{eq:radial-mode-matrix} vanishes,
\[
 I_f=J_f=0,
\]
while $I_h=\rho(1-x)/(1+Kx)>0$.  Thus the only resonant mode is $k=1$,
the normalized degeneracy index is
\begin{equation}
 d=2,
 \label{eq:constant-curvature-d-two}
\end{equation}
and its kernel is the pure-shape space
\[
 \operatorname{span}_{\re}
 \{(\cos\theta,0,0,0),(\sin\theta,0,0,0)\}.
\]

These two profiles are normal traces of ambient space-form Killing fields,
but they are not infinitesimal symmetries of the fixed-boundary problem.
\end{proposition}

\begin{proof}
The substitutions in Proposition~\ref{prop:constant-curvature-specialization}
give
\[
 Z+p_\sigma(1-2U)
 =\frac{1+K(K+2)x^2}{(1+Kx)^2}.
\]
Inserting this and $U=x(1+K)/(1+Kx)$ into
\eqref{eq:radial-determinant-identity} first gives
\begin{align*}
 (1+Kx)^2\mathring D_k^{\sigma_K}(\rho)
 ={}&-\frac14\bigl[(1+Kx)(1+x^k)-2x(1+K)\bigr]^2\\
 &-\frac{(1+Kx)^2}{4}(1-x^{2k})
 +\frac{1-x^k}{2k}\bigl[1+K(K+2)x^2\bigr].
\end{align*}
Expanding the first two terms and factoring $1-x^k$ proves
\eqref{eq:constant-curvature-determinant}.  It vanishes when $k=1$.
Directly in the entries, the two cancellations are
\[
 \frac{(1+Kx)^2}{2\gamma}I_f
 =(1-x)-(1+Kx)+x(1+K)=0,
\]
and
\[
 ((2+K)x-1)(1-Kx)-4x(1+K)
 +(1+(2+K)x)(1+Kx)=0.
\]
The polynomial
$((2+K)x-1)(1-Kx)-4x(1+K)+(1+(2+K)x)(1+Kx)$ is the numerator
of $J_f$ after multiplication by its nonzero common denominator.  Hence the
first column vanishes.

It remains to exclude every $k\geq2$.  Since
\begin{equation}
 1+K(K+2)x^2
 =1-x^2+(K+1)^2x^2>1-x^2,
 \label{eq:constant-curvature-coefficient-bound}
\end{equation}
it is enough to prove
\[
 2kx\sum_{j=0}^{k-2}x^j
 <(k-1)(1+x)\sum_{j=0}^{k-1}x^j.
\]
The right-hand side minus the left-hand side is
\[
 (k-1)(1+x^k)-2\sum_{j=1}^{k-1}x^j.
\]
For every $1\leq j\leq k-1$,
\[
 x^j+x^{k-j}<1+x^k,
\]
because the difference is
$(1-x^j)(1-x^{k-j})>0$.  Summing these inequalities proves strict
positivity and hence the second sign in
\eqref{eq:constant-curvature-signs}.  Since $I_h>0$, the first-mode block
has rank one, and Proposition~\ref{prop:radial-modal-degeneracy} gives
\eqref{eq:constant-curvature-d-two} and the displayed pure-shape kernel.

For the geometric identification, in the complex coordinate $z$ consider
\[
 Y_\alpha(z)=\alpha+K\overline\alpha z^2,
 \qquad \alpha\in\ce.
\]
For a holomorphic planar vector field $Y$, the Killing equation for
$ e^{2\sigma_K}g_{\mathrm e}$ reduces to
$\rp Y'+ d\sigma_K(Y)=0$.  In the present case,
\[
 \rp Y_\alpha'(z)
 =2K\rp(\overline\alpha z),
 \qquad
  d\sigma_K(Y_\alpha)
 =-2K\rp(\overline\alpha z),
\]
so $Y_\alpha$ is a Killing field of
$4(1+K|z|^2)^{-2}g_{\mathrm e}$.  On $r=\rho$ its Euclidean normal
traces are
\[
 \langle Y_\alpha,n\rangle_{g_{\mathrm e}}
 =(1+K\rho^2)\rp(\alpha  e^{- i\theta}),
\]
which span the same cosine--sine space.  On the fixed outer boundary their
normal traces are
$(1+K)\rp(\alpha  e^{- i\theta})$, which are not identically
zero when $\alpha\ne0$ because $K>-1$.
Their flows therefore do not preserve the unit disk boundary.  The Killing
description identifies the kernel profile only and does not produce an
additional nonconcentric solution branch.
\end{proof}

\begin{corollary}[Constant-curvature reduced problem]
\label{cor:constant-curvature-continuation}
The nondegenerate parameter range in the constant-curvature family is empty:
every $K>-1$ and every $0<\rho<1$ has $d=2$.  Consequently
Corollary~\ref{cor:nondegenerate-continuation} does not classify the local
solution set at any such reference state.

The full normalized steady-sheet problem is locally equivalent, for
$(s,\zeta)$ in the product neighborhood of $(0,0)$ from
Definition~\ref{def:physical-reduced-equation} to
\[
 B(s,\zeta)=0,
 \qquad B(s,\zeta)\in\mathcal C,
 \qquad
 \dim\mathcal K=\dim\mathcal C=2.
\]
Any $C^2$ branch must satisfy the first- and second-order necessary
conditions in Theorem~\ref{thm:reduced-obstructions}.  Neither the first-mode
resonance nor its ambient Killing-field interpretation alone implies a
nonconcentric branch.
\end{corollary}

\begin{proof}
Proposition~\ref{prop:constant-curvature-resonances} gives $d=2$ throughout
the admissible parameter range and excludes all other resonances.  The
reduction and the necessary conditions are the degenerate alternative in
Corollary~\ref{cor:radial-continuation}.  The qualification concerning nonconcentric branches follows from the nonzero outer-boundary normal trace computed in the proof of
Proposition~\ref{prop:constant-curvature-resonances}.
\end{proof}

\section{Conclusion}
\label{sec:conclusion}

We have formulated continuation of a steady vortex sheet under a prescribed
conformal deformation of the metric on the disk as a nonlinear equation on a
fixed normal-graph chart.  Its streamline and Bernoulli components naturally
occupy different Sobolev orders.  When the two one-sided tangential
velocities do not vanish simultaneously, the resulting mixed-order
linearization is Fredholm of index two.

Appending the mean normal displacement and total circulation gives a
Fredholm derivative of index zero.  If their differential is onto
$\re^2$ on the unnormalized kernel, the normalized kernel and cokernel
are canonically reduced to equal dimension $d$.  For $d=0$, the Banach
implicit function theorem yields a unique local normalized continuation.  For
$d>0$, Lyapunov--Schmidt reduction gives a $C^2$ equation
$B\colon\re\times\mathcal K\to\mathcal C$ and explicit first- and
second-order necessary conditions.  Those conditions do not, without
additional finite-dimensional information, imply a branch count, a scale, or
the existence of a further branch.

For a radial reference metric and a concentric circular sheet, Fourier
decomposition reduces the normalized derivative to explicit $2\times2$
blocks.  The determinant estimate leaves only finitely many possible
resonant modes.  In the constant-curvature family the determinant vanishes
exactly at $k=1$, is strictly negative for every $k\geq2$, and hence gives
$d=2$.  The two resonant profiles are normal traces of ambient Killing
fields, but those fields do not preserve the fixed outer boundary.  Thus the
remaining local question is sharply bounded: one must
analyze the two-dimensional reduced equation to decide whether nonconcentric
solutions exist.

\section*{Declarations}
\paragraph{Funding.}
This work was supported by JSPS KAKENHI Grant Number JP25K17294.

\paragraph{Data availability.}
Data sharing is not applicable to this article as no datasets were generated or analyzed.

\paragraph{Consent to participate.}
Not applicable.

\paragraph{Consent to publish.}
Not applicable.

\paragraph{Competing interests.}
The author declares no competing interests.

\bibliographystyle{amsplain}

\end{document}